\documentclass[12pt,twoside,reqno]{amsart}

\usepackage{amsmath, amssymb, enumitem, esint}
\usepackage{amsthm}
\usepackage[hyperindex]{hyperref}

\usepackage{hyperref}
\hypersetup{bookmarksdepth=3}

\newcommand{\lb}{\linebreak[1]}
\newcommand{\la}{\lambda}

\newcommand{\IR}{\mathbb{R}}

\newcommand{\NN}{\mathcal{N}}
\newcommand{\WW}{\mathcal{W}}

\newcommand{\eps}{\varepsilon}
\newcommand{\ov}[1]{\overline{#1}}
\newcommand{\td}[1]{\widetilde{#1}}

\DeclareMathOperator{\sign}{sign}
\DeclareMathOperator{\Bry}{Bry}
\DeclareMathOperator*{\Var}{Var}
\DeclareMathOperator*{\osc}{osc}

\DeclareMathOperator{\Ric}{Ric}
\DeclareMathOperator{\Rm}{Rm}
\DeclareMathOperator{\tr}{tr}

\DeclareMathOperator{\supp}{supp}

\DeclareMathOperator{\newtanh}{th}

\newcommand{\rrm}{r_{\Rm}}

\renewcommand{\tanh}{\newtanh}

\newcommand{\EMPTY}[1]{}

\newtheorem{Theorem}[equation]{Theorem}
\newtheorem{Lemma}[equation]{Lemma}
\newtheorem{Corollary}[equation]{Corollary}
\newtheorem{Proposition}[equation]{Proposition}
\newtheorem{Claim}[equation]{Claim}

\theoremstyle{definition}
\newtheorem{Definition}[equation]{Definition}
\theoremstyle{remark}

\numberwithin{equation}{section}

\title{Entropy and Heat Kernel bounds on a Ricci flow background}
\author{Richard H Bamler}
\address{Department of Mathematics, UC Berkeley, CA 94720, USA}
\email{rbamler@berkeley.edu}
\thanks{This work was supported by NSF grant DMS-1906500.}
\date{\today}

\begin{document}

\begin{abstract}
In this paper we establish new geometric and analytic bounds for Ricci flows, which will form the basis of a compactness, partial regularity and structure theory for Ricci flows in \cite{Bamler_RF_compactness, Bamler_HK_RF_partial_regularity}.

The bounds are optimal up to a constant that only depends on the dimension and possibly a lower scalar curvature bound.
In the special case in which the flow consists of Einstein metrics, these bounds agree with the optimal bounds for spaces with Ricci curvature bounded from below.
Moreover, our bounds are local in the sense that if a bound depends on the collapsedness of the underlying flow, then we are able to quantify this dependence using the pointed Nash entropy based only at the point in question.

Among other things, we will show the following bounds:
Upper and lower volume bounds for distance balls, dependence of the pointed Nash entropy on its basepoint in space and time, pointwise upper Gaussian bound on the heat kernel and a bound on its derivative and an $L^1$-Poincar\'e inequality.
The proofs of these bounds will, in part, rely on a monotonicity formula for a notion, called variance of conjugate heat kernels.

We will also derive estimates concerning the dependence of the pointed Nash entropy on its basepoint, which are asymptotically optimal.
These will allow us to show that points in spacetime that are nearby in a certain sense have comparable pointed Nash entropy.
Hence the pointed Nash entropy is a good quantity to measure  local collapsedness of a Ricci flow

Our results imply a local $\eps$-regularity theorem, improving a result of Hein and Naber.
Some of our results also hold for super Ricci flows.
\end{abstract}

\maketitle

\tableofcontents

\section{Introduction}
\subsection{Introduction}
In this paper we consider Ricci flows \cite{Hamilton_3_manifolds} and super Ricci flows \cite{Topping-McCann}, i.e. families of metrics $(g_t)_{t \in I}$ on an $n$-dimensional manifold $M$ satisfying the equation
\begin{equation} \label{eq_RF_super_RF}
 \partial_t g_t = - 2\Ric_{g_t} \qquad \text{or} \qquad \partial_t g_t  + 2\Ric_{g_t} \geq 0, 
\end{equation}
where in the latter inequality the symbol ``$\geq$'' denotes non-negative definiteness of symmetric 2-tensors.
Special solutions of (\ref{eq_RF_super_RF}), against which we will frequently test our results, are the trivial solutions $g_t =( 1- 2 \lambda t ) \ov{g}$ for an Einstein metric $\ov{g}$ with $\Ric_{\ov{g}} = \lambda \ov{g}$ (for Ricci flows) or for a metric $\ov{g}$ with $\Ric_{\ov{g}} \geq \lambda \ov{g}$ (for super Ricci flows).
Our goal will be to derive geometric and analytic bounds for these flows, which when restricted to these special cases, replicate a large number of the familiar bounds for spaces with lower Ricci curvature bounds.
For example, we will prove a pointwise Gaussian bound for the conjugate heat kernel, which generalizes a famous bound of Li and Yau \cite{Li_Yau_parabolic_kernel}, bounds on the variation of the pointed Nash-entropy, several bounds on volumes of balls and and $\eps$-regularity theorem for general Ricci flows.
We will also provide a heuristic reason why most of the remaining bounds that hold in the Einstein setting, such as lower heat kernel bounds and distance expansion bounds, are expected to fail in the Ricci flow setting.

On a more philosophical level, our paper provides a new perspective on the space-time geometry of Ricci flows. 
For example, we demonstrate that the intuitive strategy of relating points in different time-slices via worldlines is unnatural in the setting of general Ricci flows.
We provide an alternative via the concept of $H_n$-centers and discuss several estimates and applications.
For example, we will show how we can bound the variation of the pointed Nash-entropy in space \emph{and} time without resorting to wordlines.


The bounds and this new philosophy will be used in \cite{Bamler_RF_compactness, Bamler_HK_RF_partial_regularity} to derive a compactness, partial regularity and structure theory for Ricci flows.

There has been some interesting recent activity aimed at deriving geometric and analytic bounds for Ricci flows, motivated by Perelman's groundbreaking work \cite{Perelman1}.
In \cite{Hein-Naber-14} Hein and Naber established an integral Gaussian bound and a Poincar\'e and log-Sobolev inequality for the conjugate heat kernel on a Ricci flows.
They also studied the pointed Nash-entropy, which is related to Perelman's $\mu$-entropy.
As an application, they proved an $\eps$-regularity theorem for Ricci flows assuming an additional \emph{global} non-collapsing condition.
We will improve these bounds and use them to derive further geometric bounds on the underlying flow.
For example, we will obtain \emph{pointwise} Gaussian bounds on the conjugate heat kernel with the optimal dependence on the pointed Nash entropy.
We will also characterize the dependence of the pointed Nash-entropy on its basepoint in an optimal sense and derive various volume bounds on distance balls involving the pointed Nash entropy.
One of our applications includes an improved $\eps$-regularity theorem, which only depends on the pointed Nash-entropy at the point in question.

In the special setting of Ricci flows with upper scalar curvature bounds, bounds were derived by Chen, Wang, Zhang and the author in \cite{Zhang-noninflating, ChenWang-2013, Bamler-Zhang-1, Bamler-Zhang-2}.
We will show that most of these bounds persist in a certain form if the scalar curvature bound is removed.
Note that the setting of general Ricci flows demands significantly different techniques.
This has to do with the fact that the concept of worldlines was, in fact, natural in the setting of bounded scalar curvature.
In addition, the scalar curvature bound guaranteed distance-distortion and lower heat kernel bounds estimates, as well as the existence of cutoff functions that allowed the localization of several analytic estimates.
Unfortunately, these helpful tools are not available in the general setting and we are forced to derive different estimates that rely on the new perspective on the spacetime geometry of Ricci flows mentioned above.


We also refer to further related work in \cite{Chow_book_series_Part_III, Cao_Zhang_2011, Chau_Tam_Yu_2011, Mantegazza_Mueller_2015, Zhu_Meng_2016, Wu_Gaussian_2020, Buzano_Yudowitz_2020,Cao_Hamilton_2009, Zhang_Yoongjia_2020, Hallgren_2020}.

\subsection{Measuring the local collapsedness of the flow}
Before describing the results of this paper, let us first review some basic concepts.
Consider a complete Riemannian manifold $(M, g)$ with $\Ric \geq - (n-1)g$.
It has become customary to measure the degree of (local) collapsedness of $(M,g)$ at some point $x \in M$ and scale $r > 0$ by the normalized volume of a distance ball:
\begin{equation} \label{eq_normalized_volume}
 r^{-n} |B(x,r)| .
\end{equation}
Various important bounds, such as the upper Gaussian heat kernel bounds in \cite{Li_Yau_parabolic_kernel}, involve terms of the form (\ref{eq_normalized_volume}), which compensate for a possible collapse.
A useful property of the quantity (\ref{eq_normalized_volume}) is that it depends on $x$ and $r$ in a controlled way.
More specifically, by Bishop-Gromov volume comparison for any $x_1, x_2 \in M$, $0 < r_1, r_2 \leq A$
\begin{equation} \label{eq_norm_volume_variation}
 r_2^{-n} |B(x_2,r_2)| \geq c \bigg(A,  \frac{r_2}{r_1}, \frac{d(x_1, x_2)}{r_1} \bigg) r_1^{-n} |B(x_1,r_1)|. 
\end{equation}
In other words, the degrees of collapsedness at neighboring points and/or similar scales is comparable.

In this paper, the local collapsedness of a Ricci flow will be characterized by the pointed Nash entropy.
The pointed Nash entropy is a quantity that naturally arises from Perelman's work \cite{Perelman1}, generalizing the more commonly used $\mu$-functional.
Its importance was first highlighted in the work of Hein and Naber \cite{Hein-Naber-14}.\footnote{We remark that in \cite{BWang_local_entropy} a different quantity of measuring the local collapsedness was introduced.
We will, however, not use this quantity in this paper.}
The pointed Nash entropy is a quantity of the form $\NN_{x,t} (r^2)$, where $(x,t) \in M \times I$ should be viewed as a point in spacetime and $r > 0$ should be viewed as a scale at which we wish to measure the collapse (see Section~\ref{sec_Nash_entropy} for more details).
The quantity (\ref{eq_normalized_volume}) will turn out to be comparable to $\exp ( \NN_{x,t} (r^2))$ in many ways; in particular, if the flow consists of Einstein metrics, then both quantities are bounded by each other up to a multiplicative constant.
One of the main results in this paper will be a bound on the dependence of $\NN_{x,t} (r^2)$ on $(x,t)$ and $r > 0$ that is similar to (\ref{eq_norm_volume_variation}).
A number of geometric and analytic bounds in this paper will contain a term involving the pointed Nash entropy, usually via a factor of the form $\exp ( \NN_{x,t} (r^2) )$ or $\exp (- \NN_{x,t} (r^2) )$.
This dependence will always be optimal, so for example the corresponding result would be false after replacing $\NN_{x,t} (r^2)$ with $a\NN_{x,t} (r^2)$ for some $a > 1$ or $a <1$, depending on the situation.
Our bounds will also be local in the sense that, they will only involve bounds on the pointed Nash-entropy at the points in question, plus a possible global lower scalar curvature bound, which is natural.
So for example, they may include the pointed Nash entropy at a specific point, but  no term depending on Perelman's $\mu$-functional.

\subsection{Outline}
Given the nature of this paper, it would be uneconomical to list the precise statements of all results in this introduction.
Instead, we will only provide a brief outline of each section and a rough description of its results.
Each section is organized in a way that allows the reader to gain a quick overview of its content.
It starts with a subsection listing all of its main results, including some further explanations.
The proofs of these results can be found in the following subsections.

In Section~\ref{sec_variance}, we introduce a new notion, called \emph{variance}, which can be viewed as a form of $L^2$-concentration or $L^2$-distance between one or two probability measures.
A key property of the variance will be a new monotonicity formula under the Ricci flow or super Ricci flow, which will have two applications.
When applied to a single conjugate heat kernel, we obtain a variance bound that can be viewed as a concentration bound.
When applied to two conjugate heat kernels based at different points, but the same time, we obtain a variance bound that resembles a lower distance distortion bound.
The concentration bound will motivate the definition of ``centers'' (called \emph{$H_n$-centers}) of conjugate heat kernels.
Moreover, it implies a Gaussian integral bound on the conjugate heat kernel, which improves a result of Hein and Naber \cite{Hein-Naber-14}.

In Section~\ref{sec_gradient_estimate}, we derive a gradient estimate for solution to the heat equation on a Ricci flow or super Ricci flow background.
This bound is optimal in the sense that equality is attained for 1-dimensional solutions to the heat equation starting from the step function $\chi_{[0, \infty)}$.
The bound will imply an important integral bound on the gradient of the heat kernel, which will be important later.

In Section~\ref{sec_Nash_entropy} we first recall the definition of the pointed Nash entropy and its basic properties.
Next, we derive new bounds on the pointed Nash entropy that characterize its dependence on the basepoint and scale.
These bounds imply a bound that is comparable to (\ref{eq_norm_volume_variation}).

Our goal in Section~\ref{sec_lower_vol} is to prove lower volume bounds on distance balls of the form
\begin{equation} \label{eq_lower_vol_bound_form}
 |B(x,t, r)|_t \geq c \exp (\NN_{x,t} (r^2)) r^n.  
\end{equation}
In general such a bound is false, as one may observe on a round shrinking cylinder or Bryant soliton.
We will, however, show that a bound of the form (\ref{eq_lower_vol_bound_form}) does, in fact, hold in two cases: first if we assume that $x$ is close to an $H_n$-center of a conjugate heat kernel and second if we assume an upper scalar curvature bound near $x$.

In Section~\ref{sec_upper_HK_bounds} we establish upper bounds on the heat kernel and its gradient.
We start out by proving an $L^\infty$-bound on the heat kernel of the form
\[ K(x,t;y,s) \leq \frac{C \exp (-\NN_{x,t}(t-s))}{(t-s)^{n/2}}. \]
Next, we obtain a stronger, pointwise Gaussian estimate of the form
\[ K(x,t;y,s) \leq  \frac{C(\eps) \exp (-\NN_{x,t}(t-s))}{(t-s)^{n/2}} \exp \bigg( {- \frac{d^2_s (z,y)}{(8+\eps) (t-s)} }\bigg), \]
where $(z,s)$ is an $H_n$-center of $(x,t)$.
Lastly, we derive a pointwise bound on the gradient of the heat kernel in terms of its value, which is asymptotically similar to the gradient bound in Section~\ref{sec_gradient_estimate}:
\[ \frac{|\nabla_x K|(x,t; y,s)}{K(x,t;y,s)} \leq \frac{C}{(t-s)^{1/2}} \sqrt{ \log \bigg( \frac{C_0 \exp (- \NN_{x,t}(t-s))}{(t-s)^{n/2} K(x,t;y,s)} \bigg) }. \]

In Section~\ref{sec_upper_vol} we prove an upper volume bounds on distance balls of the form
\[ |B(x,t, r)|_t \leq C \exp (\NN_{x,t} (r^2)) r^n.  \]
This bound is an improvement of a result due to Zhang, Chen and Wang \cite{Zhang-noninflating, ChenWang-2013}, as it does not require an upper scalar curvature bound and only depends on the pointed Nash entropy.

In Section~\ref{sec_variance_parab_nbhd} we introduce the concept of $P^*$-parabolic neighborhoods, which are defined using the $W_1$-Wasserstein distance of conjugate heat kernel measures.
$P^*$-parabolic neighborhoods are similar to conventional parabolic neighborhoods, but are better behaved and in some sense more natural in the absence of curvature bounds.
We will see that $P^*$-parabolic neighborhoods and conventional parabolic neighborhoods share many useful properties.
Moreover, given local curvature bounds, both are comparable to one another.
We will also prove a volume bound and a covering lemma for $P^*$-parabolic neighborhoods.

In Section~\ref{sec_eps_reg} we combine several of our results obtained so far to prove an $\eps$-regularity theorem of the form
\[ \NN_{x,t} (r^2) \geq - \eps_n \qquad \Longrightarrow \qquad |{\Rm}| \leq r^{-2} \quad \text{on} \quad P(x,t;r), \]
where the latter denotes the two-sided parabolic ball around $(x,t)$ of radius $r$.
This is an improvement of a result of Hein and Naber \cite{Hein-Naber-14}, as it does not require any global scalar curvature or entropy bound.

In Section~\ref{sec_Poincare} we extend an $L^2$-Poincar\'e inequality due to Hein and Naber \cite{Hein-Naber-14} to any exponent $p \geq 1$.
This $L^p$-Poincar\'e inequality will be of the form
\[ \int_M h  \, d\nu_{t_0-\tau} = 0 \qquad \Longrightarrow \qquad \int_M |h|^p  d\nu_{t_0-\tau} \leq C(p) \tau^{p/2} \int_M |\nabla h|^p d\nu_{t_0-\tau}.  \]

In Section~\ref{sec_hypercontractive}, we prove a hypercontractivity estimate for solutions to the heat equation with on a Ricci flow background that is equipped with a conjugate heat kernel measure.
\[ \frac{\tau_2}{\tau_1} \geq \frac{p-1}{q-1} \qquad \Longrightarrow \qquad
 \bigg( \int_M |u|^p d\nu_{t_0 - \tau_1} \bigg)^{1/p} \leq \bigg( \int_M |u|^q d\nu_{t_0 - \tau_2} \bigg)^{1/q}.  \]
This bound is the Ricci flow analog of \cite{Gross_1975}; we mention it for completeness.

\subsection{Limitations and outlook}
We find it interesting to discuss some of the limitations of our techniques.
Most strikingly, our results lack the following two types of bounds:
\begin{itemize}
\item A pointwise lower bound on the conjugate heat kernel that depends on the distance to an $H_n$-center.
\item A bound that resembles an upper distance distortion bound.
\end{itemize}
We briefly argue why our techniques are unlikely to provide these kinds of bounds and why we believe that such bounds are unnatural if no further assumptions are made.
To see this, we first remark that all our techniques could be generalized to singular (3-dimensional) Ricci flows, as introduced in \cite{Kleiner_Lott_singular}; see also \cite{bamler_kleiner_uniqueness_stability}.
However, the following examples show that in this setting bounds of the form above are (likely) false.

Consider first a singular Ricci flow on $S^3$ that forms a single non-degenerate neckpinch with bounded diameter, for example as in \cite{Angenent_Knopf_precise_asymptotics}.
A conjugate heat kernel that is based in one component of the flow at a time past the neckpinch vanishes on the other component and is arbitrarily small near this component slightly before the neckpinch occurs.
So a lower bound on the conjugate heat kernel that only depends on the distance is false if the diameter remains bounded close to the singular time.
Second, consider a singular Ricci flow on $S^2 \times S^1$ that forms a single non-degenerate neckpinch, after which the flow becomes diffeomorphic to $S^3$.
Such a flow can be constructed using similar techniques as in \cite{Angenent_Knopf_precise_asymptotics}.
If the $S^1$-factor is chosen large enough, then distances between nearby points on either side of the neckpinch become very large after the neckpinch occurs.
So an upper distance distortion bound does not hold in this case.

One of the main goals in the study of Ricci flows in higher dimensions is the construction of a ``Ricci flow through singularities'' that generalizes the concept of a singular Ricci flow in dimension~3.
So we regard any bound that does not generalize to such a flow as unnatural and indicative of potential issues, which we may prefer to sidestep in future research.

\subsection{Acknowledgements}
I am grateful to Hans-Joachim Hein for teaching a wonderful and very inspiring course on heat equations at the summer school \emph{Advanced School on PDEs in Geometry and Physics} at USTC, Hefei in 2014.

I thank Gang Tian, Guofang Wei and Bennett Chow for inspiring conversations.
I also thank Sigurd Angenent and Dan Knopf for useful advice on their work \cite{Angenent_Knopf_precise_asymptotics}.
Lastly, I am indebted to Bennett Chow for helping to improve the paper and to Wangjian Jian for pointing out many typos in an earlier version of the manuscript.

\section{Conventions and basic definitions}
\subsection{Constants}
Unless stated otherwise, capital roman or greek letters will denote large constants (larger than $1$), while small roman or greek letters will denote small constants (smaller than $1$).
The letter $C$ (respectively $c$) will mainly denote a large (respectively small) generic constant.
We will express the dependence of constants in parentheses; i.e. $A(B, \beta)$ means that $A$ depends only on $B$ and $\beta$ in a continuous fashion.
A condition of the form ``if $\eps \leq \ov\eps$'' or ``if $A \geq \underline{A} (\eps)$'' will mean ``there is a universal constant $\ov\eps > 0$ such that if $\eps \leq \ov\eps$, then \ldots'' or ``there is a universal continuous function $\underline{A} : (0,1) \to (1, \infty)$ such that if $A \geq \underline{A} (\eps)$, then \ldots''.

\subsection{Ricci flows and super Ricci flows}
Unless specified differently, $M$ will always denote a compact, smooth manifold of dimension $n$.
Throughout the entire paper constants may depend on the dimension $n$ and this dependence will usually be omitted.
We will often analyze Ricci flows or super Ricci flows $(g_t)_{t \in I}$ on $M$ as in (\ref{eq_RF_super_RF}); it will be understood that $I \subset \IR$ is an interval, called the {\bf time-interval.}
We will often switch between the conventional picture (in which $x \in M$ are points and $t \in I$ are times) and the spacetime picture (in which $(x,t) \in M \times I$ are points), as long as this does not create any confusion.

\subsection{Heat operators and heat kernels}
Let $(M, (g_t)_{t \in I})$ be a super Ricci flow on a compact manifold.
We will frequently consider the {\bf heat operator}
\[ \square := \partial_t - \triangle_{g_t}, \]
which is coupled to the Ricci flow and can be applied to functions of the form $u \in C^2 (M \times I')$ for any non-trivial subinterval $I' \subset I$.
A solution to the equation $\square u = 0$ is called a solution to the {\bf heat equation (coupled with the Ricci flow).}
Correspondingly, the operator 
\[ \square^* := -\partial_t - \triangle_{g_t} - \tfrac12 \tr (\partial_t g_t) \]
is called the {\bf conjugate heat operator}.
Note that if $(M, (g_t)_{t \in I})$ is a Ricci flow, then
\[ \square^* := -\partial_t - \triangle_{g_t} + R, \]
where $R$ denotes the scalar curvature at time $t$.

For any $u, v \in C^2 (M \times I')$ we have
\begin{equation} \label{eq_duality_dt}
    \int_M ( \square u) v \, dg_t -  \int_M u (\square^* v) \, dg_t = \frac{d}{dt} \int_M uv \, dg_t . 
 \end{equation}
So if $I' = [t_1, t_2]$, then this gives
\begin{equation} \label{eq_duality_int}
 \int_{t_1}^{t_2} \int_M ( \square u) v \, dg_t dt - \int_{t_1}^{t_2} \int_M u (\square^* v) \, dg_t dt = \int_M u v \, dg_t \bigg|_{t=t_1}^{t=t_2}.  
\end{equation}
We also recall that for any solution to the {\bf conjugate heat equation} $\square^* v = 0$ we have
\begin{equation} \label{eq_ddt_intv_0}
 \frac{d}{dt} \int_M v \, dg_t = 0. 
\end{equation}

For any $x, y \in M$ and $s,t \in I$, $s < t$ we denote by $K(x,t;y,s)$ the heat kernel of $\square$.
That is, for fixed $(y,s)$, the function $K(\cdot, \cdot; y,s)$ is a {\bf heat kernel based at $(y,s)$}, i.e.
\[ \square K(\cdot, \cdot; y,s) = 0, \qquad \lim_{t \searrow s} K(\cdot, t; y,s) = \delta_{y}. \]
By duality, for fixed $(x,t)$, the function $K(x,t; \cdot, \cdot)$ is a {\bf conjugate heat kernel based at $(x,t)$}, i.e.
\[ \square^* K(x,t; \cdot, \cdot) = 0, \qquad \lim_{s \nearrow t} K(x,t; \cdot,s) = \delta_{x}. \]
Note that $K(x,t;y,s) > 0$ and due to (\ref{eq_ddt_intv_0}) we have
\[ \int_M K(x,t; \cdot, s) \, dg_s = 1. \]
Hence we will frequently use the following abbreviated notion.

\begin{Definition}
For $(x,t) \in M \times I$ and $s \in I$, $s \leq t$, we denote by $\nu_{x,t;s} = \nu_{x,t}(s)$ the {\bf conjugate heat kernel measure,} i.e. the probability measure on $M$ defined by
\[ d\nu_{x,t;s} := K(x,t; \cdot, s) \, dg_s, \qquad \nu_{x,t;t} := \delta_x. \]
\end{Definition}

We will often omit the index $s$ and view $\nu_{x,t} = (\nu_{x,t})_{s \in I, s \leq t}$ as a family of probability measures.
Moreover, we will often write
\[ d\nu_{x,t;s} = (4\pi \tau)^{-n/2} e^{-f} dg_s, \]
where $\tau(s) = t - s$ and $f \in C^\infty ( M \times ( I \cap (-\infty,t)))$ is called the {\bf potential}.
Note that on Euclidean space we have $f = \frac1{4\tau} d^2(x, \cdot)$.
If $u \in C^0 (M \times I')$ is a function defined on a time-slab of $M \times I$, $(x_0, t_0) \in M \times I$ and $s \in I'$, $s < t_0$, then we will often use the abbreviations
\[ \int_M u \, d\nu_{x_0, t_0; s} = \int_M u \, d\nu_{x_0, t_0} (s)  = \int_M u(\cdot, s) d\nu_{x_0, t_0; s}. \]
By (\ref{eq_duality_dt}), (\ref{eq_duality_int}), we obtain that for any $u \in C^2 (M \times [t_1, t_2])$, $[t_1, t_2] \subset I$,
\[ \frac{d}{dt} \int_M u \, d\nu_{x_0,t_0; t} = \int_M \square u \,  d\nu_{x_0,t_0; t} , \qquad
\int_M u \, d\nu_{x_0, t_0; t} \bigg|_{t= t_1}^{t=t_2} = \int_{t_1}^{t_2} \int_M \square u \, d\nu_{x_0, t_0; t} dt. \]
Moreover, if $u \in C^\infty (M \times I')$ is a solution to the heat equation $\square u =0$ and $(x_0, t_0) \in M \times I'$, then for all $t \in I'$, $t < t_0$
\[ \int_M u \, d\nu_{x_0,t_0; t}  = u(x_0, t_0). \]

We will frequently use the fact that due to the Bochner formula for any solution $u \in C^3(M \times I')$ to the heat equation $\square u = 0$ we have
\[ \square |\nabla u|^2
 =  2 \nabla \triangle u \cdot \nabla u  - 2 (\partial_t g_t) (\nabla u, \nabla u) - \triangle |\nabla u|^2 
 \leq 2 \triangle \nabla u \cdot \nabla u - \triangle |\nabla u|^2 
 =  -2 |\nabla^2 u|^2. \]
 By Kato's inequality, this implies
 \[ \square |\nabla u| \leq 0. \]
 So by the maximum principle, we obtain:
 
\begin{Lemma} \label{Lem_Lipschitz_preserved}
If $u(\cdot, t_1)$ is $L$-Lipschitz for some $t_1 \in I'$ and $L \geq 0$, then so is $u(\cdot, t)$ for all $t \geq t_1$, $t \in I'$.
\end{Lemma}
\medskip

\subsection{Monotonicity of the \texorpdfstring{$W_1$}{W\_1}-Wasserstein distance}
If $\mu_1, \mu_2$ denote two probability measures on a complete manifold $M$ and $g$ is a Riemannian metric on $M$, then we define the $W_1$-distance between $\mu_1, \mu_2$ by
\begin{equation} \label{eq_d_W1_def}
 d_{W_1}^g (\mu_1, \mu_2) := \sup_f  \bigg( \int_M f \, d\mu_1  - \int_M f \, d\mu_2 \bigg), 
\end{equation}
where the supremum is taken over all bounded, $1$-Lipschitz functions $f : M \to \IR$.
Equivalently, we may also take the supremum over all bounded $f \in C^\infty (M)$ with $|\nabla f|_g \leq 1$.
We remark that $d_{W_1}^g$ defines a complete metric on the space of probability measures on $M$ if we allow infinite distances \cite[Theorem~7.3]{Villani_topics_in_OT}.
Moreover, it is more common to define $d_{W_1}^g$ using couplings \cite[Definition~7.1.1]{Villani_topics_in_OT}; the characterization (\ref{eq_d_W1_def}) holds due to the Kantorovich-Rubinstein Theorem \cite[Theorem~1.14]{Villani_topics_in_OT}.
For the purpose of this paper, however, the characterization (\ref{eq_d_W1_def}) is sufficient and we will only use the fact that $d_{W_1}^g$ defines a metric on the space of probability measures on $M$, which can be checked easily.

The following monotonicity result will be used frequently (see, for example, \cite[Theorem~3.1]{CR_Topping_2012}):

\begin{Lemma} \label{Lem_monotonicity_W1}
Let $(M, (g_t)_{t \in I})$ be a super Ricci flow on a compact manifold and denote by $v_1, v_2 \in C^\infty(M \times I')$, $I' \subset I$, two non-negative solutions to the conjugate heat equation $\square^* v_1 = \square^* v_2 = 0$ such that $\int_M v_i (\cdot, t) dg_t = 1$ for all $t \in I'$, $i=1,2$.
Denote by $\mu_{1,t}, \mu_{2,t}$ the associated probability measures with $d\mu_{i,t} = v_i(\cdot, t) dg_t$, $i=1,2$.
Then
\[ I' \longrightarrow [0, \infty], \qquad t \longmapsto d^{g_t}_{W_1} ( \mu_{1,t}, \mu_{2,t} ) \]
is non-decreasing.
Moreover, for any two points $x_1, x_2 \in M$ and $t_0 \in I$ we have for all $t \leq t_0$, $t \in I$
\[ d^{g_t}_{W_1} (\nu_{x_1,t_0; t}, \nu_{x_2, t_0; t} ) \leq d_{t_0} (x_1, x_2). \]
\end{Lemma}

We remark that McCann and Topping \cite{Topping-McCann} have shown a similar monotonicity result for the $W_2$-distance.

\begin{proof}
Let $t_1 \leq t_2$, $t_1, t_2 \in I'$ and consider a function $f \in C^\infty(M)$ with $|\nabla f|_{g_{t_1}} \leq 1$.
Let $u \in C^\infty (M \times [t_1, t_2])$ be the solution to the heat equation $\square u = 0$ with initial condition $u(\cdot, t_1) = f$.
By Lemma~\ref{Lem_Lipschitz_preserved} we have $|\nabla u| \leq 1$ at all times.
So
\begin{align*}
  \int_M f \, d\mu_{1,t_1} - \int_M f \, d\mu_{2,t_1}
&= \int_M u(\cdot, t_1) v_1(\cdot, t_1) \, dg_{t_1} - \int_M u(\cdot, t_1) v_2(\cdot, t_1) \, dg_{t_1} \\
&= \int_M u(\cdot, t_2) v_1(\cdot, t_2) \, dg_{t_2} - \int_M u(\cdot, t_2) v_2(\cdot, t_2) \, dg_{t_2} \\
&=  \int_M u(\cdot, t_2) \, d\mu_{1,t_2} - \int_M u(\cdot, t_2) \, d\mu_{2,t_2}
\leq d^{g_{t_2}}_{W_1} (\mu_{1,t_2}, \mu_{2,t_2}). 
\end{align*}
Taking the supremum over all such $f$ implies
\[ d^{g_{t_1}}_{W_1} (\mu_{1,t_1}, \mu_{2,t_1}) \leq d^{g_{t_2}}_{W_1} (\mu_{1,t_2}, \mu_{2,t_2}), \]
which finishes the proof of the first statement.
The second statement follows from the first since $\lim_{t \nearrow t_0} d^{g_t}_{W_1} (\nu_{x_1,t_0; t}, \nu_{x_2, t_0; t} ) = d^{g_{t_0}}_{W_1} (\delta_{x_1}, \delta_{x_2} ) = d_{t_0} (x_1, x_2)$.
\end{proof}

\subsection{Lower bounds on the scalar curvature}
Let $(M, (g_t)_{t \in I})$ be a Ricci flow on a compact manifold.
Some of our following results will depend on a lower scalar curvature bound.
We recall why such a bound is frequently available.

The evolution equation for the scalar curvature reads
\[ \partial_t R = \triangle R + 2|{\Ric}|^2 \geq \triangle R  + \tfrac{2}n R^2. \]
An application of the maximum principle implies:

\begin{Lemma} \label{Lem_lower_scal}
If $R (\cdot, t_0) \geq R_{\min}$, then for all $t \geq t_0$, $t \in I$, we have
\[ R (\cdot, t) \geq \frac{n}2 \frac{R_{\min}}{\frac{n}2 - R_{\min} (t - t_0)}. \]
Moreover, if $t_0 := \inf I \in [-\infty, \infty)$, then
\[ R(\cdot, t) \geq \begin{cases} - \frac{n}{2(t - t_0)} &\text{if $t_0 > -\infty$} \\ 0 & \text{if  $t_0 = -\infty$} \end{cases}. \]
\end{Lemma}

\section{The variance of conjugate heat kernels and its monotonicity} \label{sec_variance}
\subsection{Definition of the variance and statement of the results}
In this section we introduce a new notion called \emph{variance} between one or two probability measures on a Riemannian manifold.
We will then show that the variance satisfies a certain monotonicity property on a super-Ricci flow if these probability measures evolve by the conjugate heat equation.
This monotonicity will imply concentration and integral Gaussian bounds of a single conjugate heat kernel measure, as well as proximity bound of two conjugate heat kernel measures.
Most results in the remainder of this paper will rely on these bounds.

We mention that McCann and Topping \cite{Topping-McCann} have obtained a monotonicity property for the Wasserstein distance between two probability measures evolving by the conjugate heat equation.
Their result resembles the monotonicity of the variance, but is of different spirit, as it does not imply any concentration bound.

We begin with the definition of the variance.

\begin{Definition}[Variance] \label{Def_Variance}
The {\bf variance} between two probability measures $\mu_1, \mu_2$ on a Riemannian manifold $(M,g)$ is defined as
\[ \Var (\mu_1, \mu_2) := \int_M \int_M d^2 (x_1, x_2) d\mu_1 (x_1) d\mu_2 (x_2). \]
In the case $\mu_1 = \mu_2 = \mu$, we also write
\[ \Var (\mu) = \Var (\mu, \mu) = \int_M \int_M d^2 (x_1, x_2) d\mu (x_1) d\mu (x_2). \]
If $(M, (g_t)_{t \in I})$ is a Ricci flow and if there is chance of confusion, then we will also sometimes write $\Var_t$ for the variance with respect to the metric $g_t$.
\end{Definition}

Note that $\Var (\mu_1, \mu_2)$ is linear in each argument and $\Var (\delta_x, \delta_y) = d^2(x,y)$.
Moreover, we have the following triangle inequality and relation to the $W_1$-Wasserstein distance:

\begin{Lemma} \label{Lem_Var_triangle_inequ_W_1_vs_Var}
For any three probability measures $\mu_1, \mu_2, \mu_3$ on a Riemannian manifold $(M, g)$ we have
\begin{equation} \label{eq_Var_triangle}
 \sqrt{ \Var (\mu_1, \mu_3) } \leq \sqrt{ \Var (\mu_1, \mu_2) } + \sqrt{ \Var (\mu_2, \mu_3) }, 
\end{equation}
\begin{equation} \label{eq_Var_W1_relation}
 d_{W_1}^g (\mu_1, \mu_2) \leq \sqrt{\Var(\mu_1, \mu_2)} 
\leq d_{W_1}^g (\mu_1, \mu_2) + \sqrt{\Var(\mu_1)} + \sqrt{\Var(\mu_2)}. 
\end{equation}
\end{Lemma}

Let from now on $(M, (g_t)_{t \in I})$ be a super-Ricci flow on a compact manifold and denote by $d_t : M \times M \to [0, \infty)$ the distance function at time $t$.
The results in this section rely on the following theorem.

\begin{Theorem} \label{Thm_HE_dist}
The following bound holds in the barrier and viscosity sense:
\begin{equation} \label{eq_HE_dist}
 (\partial_t - \triangle_x - \triangle_y) d^2_t (x,y) \geq -  \frac{(n-1)\pi^2}{2} - 4 . 
\end{equation}
\end{Theorem}

Here $\triangle_x$, $\triangle_y$ denote the Laplacians taken with respect to the $x$ and $y$ variable, respectively.
So $\triangle_x + \triangle_y$ equals the Laplacian on the Cartesian product $(M \times M, g_t \oplus g_t)$.
For the remainder of this paper, we will fix the constant
\[ H_n :=    \frac{(n-1)\pi^2}{2} + 4. \]
We will now present two important corollaries of Theorem~\ref{Thm_HE_dist}, which we will mainly use in this paper.
The first corollary expresses a monotonicity property of the variance if the probability measures evolve by the conjugate heat equation.

\begin{Corollary} \label{Cor_conj_HK_monotone}
Consider two non-negative solutions $v_1, v_2 \in C^\infty(M)$ to the conjugate heat equation $\square^* v_i =0$ with $\int_M v_i dg_t = 1$,  $i =1,2$, and let $\mu_{i,t} := v_i dg_t$ be the corresponding probability measures.
Then $t \mapsto {\Var}_t (\mu_{1,t},\mu_{2,t}) + H_n t$ is non-decreasing.
\end{Corollary}

If $(x_i, t_i) \in M$, $i=1,2$, and $\nu_{x_i, t_i;t}$ denote the corresponding conjugate heat kernel measures, then Corollary~\ref{Cor_conj_HK_monotone} states that
\[ t \longmapsto {\Var}_t (\nu_{x_1,t_1;t} ,\nu_{x_2,t_2;t}) + H_n t \]
is non-decreasing for $t \leq \min \{ t_1, t_2 \}$.
If $t_0 = t_1$, then we obtain.

\begin{Corollary} \label{Cor_HK_monotone_Hn}
For any two points $x_1, x_2 \in M$ and any time $t_0 \in I$ we have for $t \leq t_0$
\begin{equation} \label{eq_Var_bound_HK_same_time}
 {\Var}_t( \nu_{x_1,t_0;t} ,\nu_{x_2,t_0;t}) \leq d^2_{t_0} (x_1, x_2) + H_n (t_0 - t), \qquad
{\Var}_t ( \nu_{x_1,t_0}(t) ) \leq H_n (t_0 - t). 
\end{equation}
\end{Corollary}

The last bound in (\ref{eq_Var_bound_HK_same_time}) can be viewed as a concentration inequality.
Motivated by this, we define:

\begin{Definition}[$H_n$-center]
A point $(z,t) \in M \times I$ is called an {\bf $H_n$-center} of a point $(x_0, t_0) \in M \times I$ if $t \leq t_0$ and
\begin{equation} \label{eq_def_Hn_center}
 {\Var}_t( \delta_z , \nu_{x_0,t_0;t} ) \leq H_n (t_0 - t). 
\end{equation}
\end{Definition}

Note that by Lemma~\ref{Lem_Var_triangle_inequ_W_1_vs_Var} the bound (\ref{eq_def_Hn_center}) implies
\[ d^{g_t}_{W_1}( \delta_z , \nu_{x_0,t_0;t} )  \leq \sqrt{{\Var}_t ( \delta_z , \nu_{x_0,t_0;t} ) } \leq \sqrt{H_n (t_0 - t)}. \]
The following proposition ensures the existence of $H_n$-centers.

\begin{Proposition} \label{Prop_exist_H-center}
Given $(x_0, t_0) \in M$ and $t \leq t_0$ there is (at least) one point $z \in M$ such that $(z,t)$ is an $H_n$-center of $(x_0, t_0)$ and for any two such points $z_1, z_2 \in M$ we have $d_t (z_1, z_2) \leq 2 \sqrt{H_n (t_0 -t)}$.
\end{Proposition}

Next, we discuss integral distribution inequalities for the conjugate heat kernel.
The first result is a direct consequence of the definition of the variance.

\begin{Proposition} \label{Prop_nu_ball_bound}
If $(z,t)$ is an $H_n$-center of $(x_0, t_0)$, then for $A > 0$
\[ \nu_{x_0,t_0;t} \big(  B(z, t, \sqrt{A H_n (t_0 - t)} ) \big) \geq 1- \frac1{A}. \]
\end{Proposition}

In combination with \cite{Hein-Naber-14}, we obtain the following stronger integral Gaussian bound:

\begin{Theorem} \label{Thm_Gaussian_integral_bound}
If $(z,t)$ is an $H_n$-center of $(x_0, t_0)$, then for all $r \geq 0$ and $\eps >0$
\begin{multline*} 
 \nu_{x_0,t_0;t} \big( M \setminus B(z, t, r) \big)
=  \int_{M \setminus B(z, t,r)} K(x_0,t_0; \cdot, t) dg_t \\
 \leq 2 \exp \bigg( {- \frac{\big(r - \sqrt{2H_n (t_0-t)} \big)_+^2}{8(t_0-t)}  }\bigg)
 \leq C(\eps) \exp \bigg( - \frac{r^2}{(8+\eps)(t_0 - t)} \bigg). 
\end{multline*}
\end{Theorem}

Lastly, we remark that the $H_n$-center $(z, t)$ of a point $(x_0, t_0) \in M \times I$, $t < t_0$, may lie far from the point $x_0$.
To see this, consider the Bryant soliton \cite{Bryant2005}, which is a rotationally symmetric steady gradient soliton $(M_{\Bry}, (g_{\Bry, t})_{t \in \IR})$ on $\IR^n$ that is asymptotic to a metric of the form $dr^2 + r g_{S^{n-1}}$ and satisfies $|{\Rm}| \sim r^{-1}$ as $r \to \infty$, up to a multiplicative constant.
Denote by $x_{\Bry} \in M_{\Bry}$ its center of rotation.
It will follow from Theorem~\ref{Thm_lower_volume_H_center}, and Perelman's Pseudolocality Theorem \cite[10.3]{Perelman1} that for $t \ll 0$ and any $H_n$-center $(z,t)$ of $(x_{\Bry}, 0)$ we have $|{\Rm}| (z,t) \sim |t|^{-1}$ and therefore $d_t (z,x_{\Bry}) \sim |t|$, up to a multiplicative constant.
So the use of $z$ in Proposition~\ref{Prop_nu_ball_bound} and Theorem~\ref{Thm_Gaussian_integral_bound} as the center of distance balls is essential.

\subsection{Proofs}
\begin{proof}[Proof of Lemma~\ref{Lem_Var_triangle_inequ_W_1_vs_Var}.]
To see (\ref{eq_Var_triangle}), we estimate
\begin{align*}
\sqrt{ \Var (\mu_1, \mu_3) }
&= \bigg( \int_M \int_M \int_M d^2 (x_1, x_3) d\mu_1(x_1) d\mu_2 (x_2) d\mu_3 (x_3) \bigg)^{1/2} \displaybreak[1] \\
&\leq \bigg( \int_M \int_M \int_M \big( d (x_1, x_2) + d (x_2, x_3) \big)^2 d\mu_1(x_1) d\mu_2 (x_2) d\mu_3 (x_3) \bigg)^{1/2} \displaybreak[1] \\
&\leq \bigg( \int_M \int_M \int_M d^2 (x_1, x_2) d\mu_1(x_1) d\mu_2 (x_2) d\mu_3 (x_3) \bigg)^{1/2} \\
&\qquad + \bigg( \int_M \int_M \int_M d^2 (x_2, x_3) d\mu_1(x_1) d\mu_2 (x_2) d\mu_3 (x_3) \bigg)^{1/2} \\
&= \sqrt{ \Var (\mu_1, \mu_2) } + \sqrt{ \Var (\mu_2, \mu_3) }. 
\end{align*}
Let us now show (\ref{eq_Var_W1_relation}).
The following proof uses the definition (\ref{eq_d_W1_def}) of $d_{W_1}^g$.
For a proof using the definition of $d_{W_1}^g$ using couplings, see \cite{Bamler_RF_compactness}.
For the first bound in (\ref{eq_Var_W1_relation}), consider a bounded $1$-Lipschitz function $f : M \to \IR$ and observe that
\begin{multline*}
 \int_M f \, d\mu_1 - \int_M f \, d\mu_2
= \int_M \int_M (f(x_1) - f(x_2)) d\mu_1 (x_1) d\mu_2 (x_2) \\
\leq \int_M \int_M d(x_1, x_2) d\mu_1 (x_1) d\mu_2 (x_2) 
\leq \bigg( \int_M \int_M d^2(x_1, x_2) d\mu_1 (x_1) d\mu_2 (x_2) \bigg)^{1/2}
= \sqrt{\Var (\mu_1, \mu_2)}. 
\end{multline*}
For the second bound, define
\[ f(x) := \sqrt{\Var(\delta_{x}, \mu_2)} \geq   \sqrt{\Var(\mu_1, \mu_2)} - \sqrt{\Var(\delta_{x}, \mu_1)}. \]
By (\ref{eq_Var_triangle}) the function $f$ is $1$-Lipschitz, since
\[ |f(x) - f(x')| 
= \big| \sqrt{\Var(\delta_{x}, \mu_2)} - \sqrt{\Var(\delta_{x'}, \mu_2)} \big|
\leq \sqrt{\Var(\delta_{x}, \delta_{x'})} = d(x, x'). \]
Next observe that if we set $f_A := \min \{ f, A \} \geq 0$, then
\[ \int_M f_A \, d\mu_2 \leq \int_M f \, d\mu_2 \leq \int_M \sqrt{\Var(\delta_{x}, \mu_2)} d\mu_2(x) 
\leq \bigg( \int_M \Var(\delta_{x}, \mu_2) d\mu_2 (x) \bigg)^{1/2}
= \sqrt{\Var (\mu_2)}. \]
So letting $A \to \infty$ implies that if $\Var (\mu_2) < \infty$
\begin{align*}
 d_{W_1}^g (\mu_1, \mu_2) 
&\geq \int_M f \, d\mu_1 - \int_M f \, d\mu_2 \\
&\geq  \sqrt{\Var(\mu_1, \mu_2)} - \int_M \sqrt{\Var(\delta_{x}, \mu_1)} d\mu_1(x)  - \int_M \sqrt{\Var(\delta_{x}, \mu_2)} d\mu_2(x) \\
&\geq  \sqrt{\Var(\mu_1, \mu_2)} - \bigg( \int_M \Var(\delta_{x}, \mu_1) d\mu_1(x) \bigg)^{1/2}  - \bigg( \int_M \Var(\delta_{x}, \mu_2) d\mu_2(x) \bigg)^{1/2}  \\
&= \sqrt{\Var(\mu_1, \mu_2)} - \sqrt{\Var(\mu_1)} -\sqrt{\Var(\mu_2)}. 
\end{align*}
This finishes the proof.
\end{proof}
\bigskip

\begin{proof}[Proof of Theorem~\ref{Thm_HE_dist}.]
The case $x=y$ is clear since the right-hand side is smaller than $-4n$.
So assume that $x \neq y$.
Set $d := d_t (x,y)$ and let $\gamma : [0,d] \to M$ be an arclength minimizing geodesic between $x,y$.
It suffices to prove the inequality in the barrier sense, which implies the viscosity sense.

We first reduce the lemma to the case in which $x,y$ do not lie in each other's cut locus.
Assume that the lemma is already known in this case.
Let $\eps > 0$ be a small constant.
Consider two points $x' = \exp_x (u)$, $y' = \exp_y(v)$ near $x, y$, respectively.
Denote by $u_\eps, v_\eps$ the parallel transports of $u,v$ to $x_\eps := \gamma(\eps)$, $y_\eps := \gamma(d-\eps)$, respectively and set $x'_\eps := \exp_{x_\eps}(u_\eps)$, $y'_\eps := \exp_{y_\eps} (v_\eps)$.
Then for $t' \in I$ near $t$ we have
\begin{multline*}
 b_\eps (x',y',t') :=  b^{(1)}_\eps (x',y',t') +  b^{(2)}_\eps (x',y',t') +  b^{(3)}_\eps (x',y',t') \\
 := d_{t'} (x', x'_\eps) + d_{t'} (x'_\eps, y'_\eps) + d_{t'} (y'_\eps, y') \geq d_{t'} (x',y')
\end{multline*}
and equality holds for $x' = x$, $y' = y$.
So $b_\eps$ is an upper barrier for the distance function  at $(x,y, t)$.
Note that $b_\eps$ is smooth near $(x,y,t)$, because $x_\eps, y_\eps$ do not lie in each other's cut locus.
For the same reason we have
\[ \liminf_{\eps \to 0} (\partial_t - \triangle_x - \triangle_y)  b^2_\eps (x,y,t) =  \liminf_{\eps \to 0} (\partial_t - \triangle_x - \triangle_y)  (b^{(2)}_\eps)^2 (x,y,t) \geq -  \frac{(n-1)\pi^2}{2} - 4. \]
This shows that (\ref{eq_HE_dist}) holds in the barrier sense.

Assume from now on that $x,y$ are not located in each other's cut locus.
Unless stated otherwise, all geometric quantities will be taken at time $t$.
Then $(x',y',t') \mapsto d^2_{t'}(x',y')$ is smooth near $(x,y,t)$ and
\begin{equation} \label{eq_dtd2}
 \partial_t d^2_t (x,y) \geq - 2d \int_0^d \Ric (\gamma'(s),\gamma'(s)) ds. 
\end{equation}
Let $e_1 (s) = \gamma' (s), e_2(s), \ldots, e_n(s) \in T_{\gamma(s)}M$ be a parallel orthonormal frame along $\gamma$.
For $i = 2, \ldots, n$ let $\gamma^i_u(s)$ be a variation of $\gamma$ such that \[ \partial_u \big|_{u=0} \gamma^i_u (s) = v_i (s) := \sin \Big(\frac{\pi}{2d} s \Big) e_i (s) \]
and $(D/\partial u)\partial_u \gamma^i_u = 0$.
The corresponding energy function (at time $t$)
\[ E_i (u) := \frac12 \int_0^d |\gamma^{i,\prime}_u(s)|^2_t ds \]
satisfies $E'_i (0) = \langle \gamma'(s) , v_i(s) \rangle_t |_{s=0}^{s=d}  = 0$ and
\begin{multline*}
E_i''(0) = \int_0^d \bigg( \bigg| \frac{D v_i}{\partial s} \bigg|^2 - R (\gamma' (s), v_i(s), v_i(s), \gamma'(s) \bigg) ds \\
= \int_0^d \bigg( \Big( \frac{\pi}{2d} \Big)^2  \cos^2 \Big(\frac{\pi}{2d} s \Big) - \sin^2 \Big( \frac{\pi}{2d}s \Big) R (\gamma' (s), e_i(s), e_i(s), \gamma'(s))  \bigg) ds.
\end{multline*}
Since $d^2_t (x, \gamma^i_u(d))\leq L_t^2(\gamma^i_u) \leq 2d E_i(u)$, where $L_t$ denotes the length at time $t$, with equality for $u = 0$, we obtain
\[ \frac{d^2}{du^2} \bigg|_{u= 0} d_t^2 (x,\gamma^i_u(d)) \leq 2dE''_i (0). \]
Since $v_i (d) = e_i(d)$ and $(D/\partial u)\partial_u \gamma^i_u = 0$, we can sum over $i = 2, \ldots, n$ and obtain
\[ \triangle_y d^2_t (x,y) \leq 2+  \int_0^d  \bigg( 2 (n-1) d \Big(  \frac{\pi}{2d} \Big)^2   \cos^2 \Big(\frac{\pi}{2d} s \Big) - 2d \sin^2 \Big( \frac{\pi}{2d} s \Big) \Ric (\gamma' (s), \gamma'(s))  \bigg) ds. \]
By reversing the roles of $x, y$ we obtain similarly that
\[ \triangle_x d^2_t (x,y) \leq 2+  \int_0^d   \bigg( 2 (n-1)d \Big(  \frac{\pi}{2d} \Big)^2 \sin^2 \Big( \frac{\pi}{2d} s \Big) - 2d \cos^2 \Big( \frac{\pi}{2d} s \Big) \Ric (\gamma' (s), \gamma'(s))  \bigg) ds. \]
Adding both inequalities and combining the result with (\ref{eq_dtd2}) implies (\ref{eq_HE_dist}).
\end{proof}
\bigskip

\begin{proof}[Proof of Corollaries~\ref{Cor_conj_HK_monotone}, \ref{Cor_HK_monotone_Hn}.]
Fix some $t_1 \in I$ and consider the solution $u \in C^0 (M \times M \times ([t_1, \infty) \cap I)) \cap C^\infty (M \times M \times ((t_1, \infty) \cap I))$ to the heat equation
\[ (\partial_t - \triangle_x - \triangle_y) u = - H_n, \qquad u (\cdot, t_1) = d^2_{t_1}, \]
where we use the evolving background metric $g_t \oplus g_t$ on $M \times M$.
By Theorem~\ref{Thm_HE_dist} and the maximum principle, we have $u \leq d^2$.

Next, we have for $t \in I$, $t > t_1$
\begin{align*}
 \frac{d}{dt} \int_M \int_M & u (x,y, t) v_1(x,t) v_2(y,t) dg_t(x) dg_t(y) \\
&= \int_M \int_M \Big(  \partial_t u (x,y,t)  v_1(x,t) v_2(y,t)  \\
&\qquad\qquad\qquad + u(x,y,t)  \partial_t v_1 (x,t)   v_2(y,t) + u (x,y,t) v_1(x,t) \partial_t v_2(y,t) \\ 
&\qquad\qquad\qquad + \frac12 u(x,y,t) v_1(x,t) v_2(y,t) (\tr (\partial_t g_t))(x,t)   \\
&\qquad\qquad\qquad + \frac12 u(x,y,t) v_1(x,t) v_2(y,t) (\tr (\partial_t g_t))(y,t)  \Big) dg_t(x) dg_t(y) \\
&= \int_M \int_M \big( ((\triangle_x + \triangle_y) u) (x,y, t) v_1(x,t) v_2(y,t) - H_n v_1(x,t) v_2(y,t) \\
&\qquad\qquad - \triangle v_1 (x,t) v_2(y,t) - v_1(x,t) \triangle v_2(y,t) \big) dg_t(x) dg_t(y) = -H_n.
\end{align*}
So for any $t > t_1$
\begin{align*}
 \int_M \int_M  & d^2_t (x,y) v_1(x,t) v_2(y,t) dg_t(x) dg_t(y)  + H_n t \\
&\geq \int_M \int_M  u (x,y,t) v_1(x,t) v_2(y,t) dg_t(x) dg_t(y) + H_n t \\
&= \int_M \int_M  u (x,y,t_1) v_1(x,t_1) v_2(y,t_1) dg_{t_1} (x) dg_{t_1} (y) + H_n t_1 \\
&= \int_M \int_M  d_{t_1}^2 (x,y,t_1) v_1(x,t_1) v_2(y,t_1) dg_{t_1} (x) dg_{t_1} (y) + H_n t_1, 
\end{align*}
which proves Corollary~\ref{Cor_conj_HK_monotone}.

For Corollary~\ref{Cor_HK_monotone_Hn}, observe that $\lim_{t \nearrow t_0} {\Var}_t( \nu_{x_1,t_0;t} ,\nu_{x_2,t_0;t}) = {\Var}_{t_0} (\delta_{x_1}, \delta_{x_3}) = d^2_{t_0} (x_1, x_2)$.
\end{proof}
\bigskip

\begin{proof}[Proof of Proposition~\ref{Prop_exist_H-center}.]
The existence of $z$ follows from Corollary~\ref{Cor_HK_monotone_Hn} and the definition of the variance.
The last statement can be seen as follows, using Lemma~\ref{Lem_Var_triangle_inequ_W_1_vs_Var}:
 \[ d_t(z_1, z_2) = \sqrt{ {\Var}_t (\delta_{z_1}, \delta_{z_2}) }
\leq \sqrt{ {\Var}_t (\delta_{z_1},\nu_{x_0,t_0;t}  ) } + \sqrt{ {\Var}_t (\nu_{x_0,t_0;t}, \delta_{z_2}) }
\leq 2 \sqrt{H_n (t_0 -t)}. \qedhere \]
\end{proof}
\bigskip

\begin{proof}[Proof of Proposition~\ref{Prop_nu_ball_bound}.]
Let $B := B(z, t, \sqrt{AH_n (t_0 - t)})$.
Using Corollary~\ref{Cor_HK_monotone_Hn}, we obtain that
\[ \nu_{x_0, t_0; t} (M \setminus B) \leq \frac1{AH_n (t_0 - t)} \int_{M \setminus B} d^2_t (z, \cdot) d\nu_{x_0, t_0; t}
\leq \frac1{AH_n (t_0 - t)} {\Var}_t ( \nu_{x_0, t_0; t} ) \leq \frac1A, \]
which implies the proposition.
\end{proof}
\bigskip

\begin{proof}[Proof of Theorem~\ref{Thm_Gaussian_integral_bound}.]
The first bound follows from \cite[Theorem~1.13]{Hein-Naber-14} applied to the subsets $B(z, t, \sqrt{2H_n (t_0 - t)})$, $M \setminus B(z,t,r)$ using Proposition~\ref{Prop_nu_ball_bound}.
Note that this theorem also holds for super-Ricci flows, as one may verify easily.
The second bound is a direct consequence of the first.
\end{proof}

\section{An improved Gradient estimate and its consequences} \label{sec_gradient_estimate}
\subsection{Statement of the results}
The main result of this section is a slight, but ---  as we will soon see --- important improvement of a gradient bound due to Zhang \cite[Theorem 3.2]{Zhang-gradient} and Cao-Hamilton \cite[Theorem 5.1]{Cao-Hamilton-2009}.

Let $\Phi : \IR \to (0,1)$ be the following antiderivative
\[ \Phi' (x) = (4\pi)^{-1/2} e^{-x^2/4}, \qquad \lim_{x \to -\infty} \Phi (x) = 0, \qquad  \lim_{x \to \infty} \Phi (x) = 1. \]
Then $\Phi_t (x) := \Phi (t^{-1/2} x)$ is a solution to the 1-dimensional heat equation $\partial_t \Phi_t = \Phi''_t$ with initial condition $\chi_{[0, \infty)}$.
For any $t > 0$ denote by $\Phi^{-1}_t : (0,1) \to \IR$ the inverse function of $\Phi_t : \IR \to (0,1)$.

The following theorem states that any solution $(u_t)_{t > 0}$ to the scalar heat equation on a super Ricci flow background  that only takes values in $(0,1)$ has a gradient that is bounded by the corresponding derivative of $\Phi_t$.
In other words, we have the gradient bound $|\nabla u|(x,t) \leq \Phi'_t(s)$, if $s \in \IR$ is chosen such that $u(x,t) = \Phi_t (s)$

For the remainder of this section let $(M, (g_t)_{t \in I})$ be a super Ricci flow on a compact manifold.

\begin{Theorem} \label{Thm_gradient_estimate}
Consider a solution $u \in C^\infty (M \times  [t_0, t_1])$, $[t_0, t_1] \subset I$, to the heat equation $\partial_t u = \triangle_{g_t} u$ coupled with the super Ricci flow $(M, (g_t)_{t \in I})$ and let $T \geq 0$.
Suppose that $u$ only takes values in $(0,1)$ and suppose that $| \nabla (\Phi_T^{-1} ( u (\cdot , t_0) ))| \leq 1$ if $T > 0$.

Then $| \nabla (\Phi^{-1}_{T+t - t_0} ( u(\cdot, t) ))| \leq 1$ for all $t \in [t_0, t_1]$.
\end{Theorem}

Note that the theorem is sharp, because equality is attained if we consider the trivial super Ricci flow on $\IR$ and let $u_t := \Phi_t$.
Taking the Cartesian product with any other super Ricci flow produces similar examples in all dimensions.

Theorem~\ref{Thm_gradient_estimate} will imply the following integral bounds on the gradient of the heat kernel, which will become important later.

\begin{Proposition} \label{Prop_nab_K_bounds}
Let $[s,t] \subset I$ and write $d\nu := d\nu_{x,t;s} = K(x,t;\cdot, s) dg_s$.
Then for any $1 \leq p < \infty$
\begin{equation} \label{eq_Prop_nab_K_bounds_1}
 (t-s)^{p/2} \int_M \bigg( \frac{|\nabla_x K(x,t; \cdot, s)|}{K(x,t; \cdot, s)} \bigg)^p d\nu \leq C(n,p). 
\end{equation}
Moreover, for any measurable subset $X \subset M$ we have
\[ (t-s)^{p/2} \int_X \bigg( \frac{|\nabla_x K(x,t; \cdot, s)|}{K(x,t; \cdot, s)} \bigg)^p d\nu \leq C(n,p) \nu(X) \big({- \log  (\nu(X)/2)} \big)^{p/2}. \]
If $p = 2$, then we can take $C(n,p) = \frac{n}2$ in (\ref{eq_Prop_nab_K_bounds_1}) and we even have for any $v \in T_x M$ with $|v|_t = 1$
\[  (t-s) \int_M \bigg( \frac{\partial_v K(x,t; \cdot, s)}{K(x,t; \cdot, s)} \bigg)^2 d\nu \leq \frac12.  \]
\end{Proposition}
\bigskip

\subsection{Proof of the gradient estimate}
\begin{proof}[Proof of Theorem~\ref{Thm_gradient_estimate}.]
Let us first reduce the theorem to the case $T > 0$.
Since $M$ is compact, $u_{t_0}$ takes values in $(\eps, 1-\eps)$ for some $\eps > 0$.
Thus $|\nabla (\Phi_T^{-1} (u(\cdot, t_0)))| \to 0$ uniformly as $T \searrow 0$.
So we may prove the theorem for small $T$ and then let $T \searrow 0$.

Assume that $T > 0$.
By shifting the flow in time, we may additionally assume that $T = t_0$.
Write $u_t =: \Phi_t \circ h_t$ for some smooth family $(h_t)_{t \in [t_0, t_1]}$.
If we abbreviate $h = h_t$, $u = u_t$ and $\Phi'_t = \Phi'_t \circ h_t$, $\partial_t \Phi_t = \partial_t  \Phi_t \circ h_t$ etc., then we obtain
\[   \partial_t h \, \Phi'_t + \Phi''_t=  \partial_t h \, \Phi'_t + \partial_t \Phi_t = \partial_t u_t=  \triangle u = \triangle h \, \Phi'_t +  |\nabla h|^2 \, \Phi''_t. \]
Since $\Phi'_t (x) = (4\pi t)^{-1/2} e^{-x^2/4t}$ and $\Phi''_t(x) = - \frac{x}{2t}  (4\pi t)^{-1/2} e^{-x^2/4t} = - \frac{x}{2t} \Phi'_t$, we obtain
\[ \partial_t h - \frac{1}{2t} h = \triangle h - \frac{1}{2t} |\nabla h|^2 h. \]
Therefore, by Bochner's identity
\begin{multline*}
 \nabla \partial_t h \cdot \nabla h = \nabla \triangle h \cdot \nabla h  + \frac1{2t} \nabla \big( ( 1- |\nabla h|^2) h \big) \cdot \nabla h \\
= \frac12 \triangle |\nabla h |^2 - |\nabla^2 h|^2 - \Ric (\nabla h, \nabla h)  - \frac1{2t} \nabla h \cdot h \nabla |\nabla h|^2 +  \frac1{2t}    (1- |\nabla h|^2)  |\nabla h|^2 . 
\end{multline*}
It follows that
\[ \frac12 \partial_t |\nabla h|^2 
= \nabla \partial_t h \cdot \nabla h - \frac12 (\partial_t g_t)(\nabla h, \nabla h)
\leq \frac12 \triangle |\nabla h|^2 - \frac1{2t} \nabla h \cdot h \nabla |\nabla h|^2  + \frac1{2t}   (1- |\nabla h|^2)  |\nabla h|^2 . \]
By the maximum principle, the bound $|\nabla h|^2 \leq 1$ remains preserved, which finishes the proof.
\end{proof}
\bigskip

\subsection{Proof of Proposition~\ref{Prop_nab_K_bounds}}

\begin{proof}[Proof of Proposition~\ref{Prop_nab_K_bounds}.]
By parabolic rescaling and application of a time-shift, we may assume without loss of generality that $[s,t] = [0,1]$.
Fix a vector $v \in T_x M$ with $|v|_1 = 1$ and write
\begin{equation} \label{eq_def_q}
 q := \frac{\partial_v K(x,1; \cdot, 0)}{K(x,1; \cdot, 0)}, 
\end{equation}
where the $\partial_v$-derivative is taken with respect to the first entry.
Let $X \subset M$ be a measurable set and consider the solution $(u_t)_{t \in [0, 1]}$ to the heat equation with initial condition $\chi_X$.
Since
\[ \nu(X) = \int K(x,1;\cdot, 0) \chi_X \, dg_0 = u(x,1), \qquad \int_X q \, d\nu = \int_M \partial_v K(x,1; \cdot, 0)\chi_X \, dg_0 = \partial_v u(x,1). \]
Applying Theorem~\ref{Thm_gradient_estimate} to $(1-2\eps) u + \eps$ on $[t_0,1]$ for $t_0$ and then letting $\eps, t_0 \to 0$ implies that
\begin{equation} \label{eq_der_bound_int_form}
 \int_X q \, d\nu  \leq \Phi' \big( \Phi^{-1} \big( \nu(X) \big) \big) =: F \big( \nu(X) \big). 
\end{equation}

The idea of the proof is to exploit this bound.
For this purpose, define the function $h : (0,1) \to \IR$ by
\[ h(a) := \sup \big\{ h' \in \IR \;\; : \;\; \nu (\{ q \geq h' \}) \geq a \big\}. \]
The following claim will allow us to reduce the proposition to a problem concerning $h$ only.

\begin{Claim}
$h$ is non-increasing, uniformly bounded and for any measurable subset $X \subset M$ and $1 \leq p < \infty$ we have
\begin{equation} 
 \int_M |q|^p d\nu 
=  \int_0^{1} |h|^p(\td a) d\td a, \label{eq_int_q_p_h_p} 
\end{equation}
\begin{equation}
\int_X q_+^p d\nu 
\leq  \int_0^{\nu(X)} h_+^p(\td a) d\td a, \label{eq_int_q_p_h_p_X}
\end{equation}
\begin{equation} \label{eq_int_haa_0}
\int_0^1 h(\td a) d\td a = 0,
\end{equation}
\begin{alignat}{2} \label{eq_haa_F_a_1}
 a \, h(a)  &\leq \int_0^{a} h(\td a) d \td a 
&&\leq F(a), \\ 
(1-a)  \, h(a) &\geq \int_a^{1} h(\td a) d \td a 
&&\geq - F(a). \label{eq_haa_F_a_2}
\end{alignat}
\end{Claim}

\begin{proof}
By definition, $h$ is non-increasing, uniformly bounded from above and below,    upper semi\-con\-tin\-u\-ous.
We claim that for any $b \in \IR$
\begin{equation} \label{eq_open_closed_super_levels}
  (0, \nu ( \{ q \geq b \} )] = h^{-1} ( [b, \infty ) ), \qquad (0, \nu ( \{ q > b \} )) \subset h^{-1} ( (b, \infty ) ) \subset (0, \nu ( \{ q > b \} )]. 
\end{equation}
To see the first identity, observe that if $a \leq \nu (\{ q \geq b \})$, then by the definition of $h$ we have $h(a) \geq b$.
The second identity follows from the first using $\lim_{b' \searrow b} \nu (\{ q \geq b' \}) = \nu (\{ q > b \})$ and
\[ h^{-1} ( (b, \infty)) = \bigcup_{b' > b} h^{-1} ((b', \infty)) = \bigcup_{b' > b} (0, \nu (\{ q \geq b' \})]. \]
We obtain from (\ref{eq_open_closed_super_levels}) that
\[  
 \nu ( \{ q \geq b \}) = |h^{-1}([b, \infty))|
\qquad \nu ( \{ q > b \}) = |h^{-1}((b, \infty))|,
\qquad  \nu ( \{ q = b \}) = |h^{-1}(\{ b \})|. \]

Next, we claim that for any $b \in \IR$ and any measurable subset $X \subset M$ with
\[ \{ q > b \} \subset X \subset \{ q \geq b \} \]
and any continuous function $f : \IR \to \IR$ we have
\begin{equation} \label{eq_int_q_h_a_int_fh}
 \int_{X} f( q ) d\nu = \int_0^{\nu(X)} f(h(\td a)) d\td a. 
\end{equation}
To see this, fix some $\eps > 0$ and choose $b = b_0 < b_1 < \ldots < b_m$ such that $b_m > \max_M q$ and $\osc_{[b_{i-1}, b_i]} f  \leq \eps$ for all  $i \geq 1, \ldots, m$.
Then by the previous paragraph we have for $i = 1, \ldots, m$
\[  \nu (X \cap \{ q = b \}) 
= \nu ( X ) - \nu ( \{ q > b \} ) 
= | [0, \nu(X) ] \cap \{ h = b \} |, \]
\[ \nu ( \{ b_{i-1} < q \leq b_i \} ) = | h^{-1} ((b_{i-1}, b_i]) |. \]
It follows that
\begin{multline*}
 \bigg| \int_{X} f( q ) d\nu - \int_0^{\nu(X)} f(h(\td a)) d\td a \bigg| \\
\leq \bigg| \int_{X \cap \{  q= b \}} f(q) d\nu - \int_{[0, \nu(X)] \cap \{  h = b \} } f (h) \bigg|  +
 \sum_{i=1}^m \bigg| \int_{\{ b_{i-1} < q \leq b_i \}} f(q) d\nu - \int_{ h^{-1} (( b_{i-1} , b_i ])} f (h) \bigg|  \\
 \leq 
  \sum_{i=1}^m \eps \nu ( \{ b_{i-1} < q \leq b_i \} ) \leq \eps. 
\end{multline*}
Letting $\eps \to 0$ implies (\ref{eq_int_q_h_a_int_fh}), which implies (\ref{eq_int_q_p_h_p}).

To see (\ref{eq_int_q_p_h_p_X}) let $a := \nu(X)$ and choose
\[  \{ q > h(a) \} \subset X' \subset \{ q \geq h(a) \} \]
such that $\nu (X') = a = \nu(X)$.
Then
\begin{multline} \label{eq_diff_int_X_int_Xp}
 \int_{X'} q_+^p d\nu
- \int_X q_+^p d\nu
= \int_{X' \setminus X} q_+^p d\nu - \int_{X \setminus X'} q_+^p d\nu \\
\geq  (h(a))_+^p \nu ( X' \setminus X )
- (h(a))_+^p \nu (X \setminus X' )
=  0.
\end{multline}
So setting $f(x) = x_+^p$ in (\ref{eq_int_q_h_a_int_fh}) yields
\begin{equation*}
 \int_X q_+^p d\nu 
\leq \int_{X'} q_+^p d\nu
= \int_0^{a} (h(\td a))_+^p d\td a. 
\end{equation*}

For (\ref{eq_int_haa_0}), observe that by (\ref{eq_int_q_h_a_int_fh}) we have
\[ \int_0^1 h(\td a)d\td a = \int_M q \, d\nu 
= \int_M \partial_v K(x, 1; \cdot, 0) dg_0 
= \partial_v \int_M  K(x, 1; \cdot, 0) dg_0 
= \partial_v 1
= 0.  \]
The first bound in (\ref{eq_haa_F_a_1}), (\ref{eq_haa_F_a_2}) follows from the monotonicity of $h$.
For the second bound, choose $b \in \IR$ and $\{ q > b \} \subset X \subset \{ q \geq b \}$ such that $\nu(X) = a$.
Then by (\ref{eq_int_haa_0}), (\ref{eq_int_q_h_a_int_fh}), (\ref{eq_der_bound_int_form})
\begin{equation*}
 -\int_a^1 h(\td a) d \td a = \int_0^{a} h(\td a) d \td a = \int_{X} q \, d\nu \leq F( \nu ( X)) = F(a). \qedhere
\end{equation*}
\end{proof}
\medskip

We claim that for $a \leq \frac14$
\begin{equation} \label{eq_F_a_log_a_a}
 F(a) \leq C (- \log a)^{1/2} a. 
\end{equation}
To see this, choose $s := \Phi^{-1} (a) \leq \Phi^{-1} ( \frac14 ) < 0$.
We obtain that
\[ \Phi (s) = \int_{-\infty}^s (4\pi)^{-1/2}e^{-\td s^2/4} d\td s
\leq C \int_{-\infty}^s e^{-s^2/8 - (s - \td s)^2/8} d\td s \leq C e^{-s^2/8} \]
and therefore, since $\td{s}^2 \leq s^2 + C$ for all $\td{s} \in [s + s^{-1},s ]$, we have
\begin{multline*}
 F(a) = \Phi'(s) = (4\pi)^{-1/2} e^{-s^2/4} \leq C (-s) \int_{s+s^{-1}}^s  (4\pi)^{-n/2} e^{-\td {s}^2/4} d\td{s} \\ \leq C (-s) \Phi(s) = C (C- \log a)^{1/2} a \leq C (- \log a)^{1/2} a. 
\end{multline*}

By (\ref{eq_haa_F_a_1}), (\ref{eq_haa_F_a_2}) we have
\[ - \frac{F(a)}{1-a} \leq h(a) \leq \frac{F(a)}{a} \leq C (-\log a)^{1/2}. \]
So the  proposition, except for the statement about $p=2$, follows from (\ref{eq_int_q_p_h_p}), (\ref{eq_int_q_p_h_p_X}), (\ref{eq_F_a_log_a_a}) and the fact that for $a \leq c(p)$
\[ \int_0^a (- \log \td a)^{p/2} d \td a 
\leq  2 \int_0^a \Big( (- \log \td a)^{p/2} - \frac{p}2  (- \log \td a)^{p/2-1} \Big)d \td a 
= 2 a (-\log a)^{p/2}. \]

Lastly, we prove the statement about $p=2$.
By (\ref{eq_def_q}), (\ref{eq_int_q_p_h_p}) it suffices to show that
\begin{equation} \label{eq_int01_h2_12}
  \int_0^1 h^2(a) da \leq \frac12. 
\end{equation}
To see this, we first compute that
\[ F' (a) = \frac{\Phi''}{\Phi'} ( \Phi^{-1} (a)), \qquad
F''(a) = \frac{\Phi''' \Phi' - (\Phi'')^2}{(\Phi')^3}  ( \Phi^{-1} (a)) . \]
Since for $s := \Phi^{-1}(a)$ we have
\[ \Phi'''(s) \Phi'(s) - (\Phi''(s))^2
= \Big( -\frac12 + \frac{s^2}{4} \Big) ( \Phi'(s))^2 -  \frac{s^2}{4} ( \Phi'(s))^2 \leq 0, \]
we obtain $F''(a) \leq 0$.
Next define $H : [0,1] \to \IR$ by
\[ H(a) :=  \int_0^{a} h(\td a) d \td a. \]
Then $H'' \leq 0$ in the weak sense and by (\ref{eq_haa_F_a_1}), (\ref{eq_int_haa_0}) we have 
\[  H \leq F, \qquad H(0) = H(1) = 0. \]
Thus
\begin{align*}
 \int_0^1 h^2(a) da
&= \lim_{\eps \to 0} \int_{\eps}^{1-\eps} (H'(a))^2 da
= \lim_{\eps \to 0} \bigg( H H' \bigg|_{\eps}^{1-\eps} - \int_{\eps}^{1-\eps} H(a) H''(a) da \bigg) \\
&\leq  - \int_{\eps}^{1-\eps} F(a) H''(a) da  
=  \lim_{\eps \to 0} \bigg( (- F H' + F' H) \bigg|_{\eps}^{1-\eps} - \int_{\eps}^{1-\eps} F''(a) H(a) da \bigg) \\
&\leq - \int_{0}^{1} F''(a) F(a) da =  \lim_{\eps \to 0} \bigg( - F F'  \bigg|_{\eps}^{1-\eps} + \int_{\eps}^{1-\eps} (F'(a) )^2 da \bigg) \\
&= \int_0^1 (F'(a) )^2 da
= \int_0^1 \bigg( \frac{\Phi''}{\Phi'} \bigg)^2 (\Phi^{-1}(a)) da \\
&= \int_{-\infty}^\infty \frac{(\Phi'')^2}{\Phi'} (s) ds
= \int_{-\infty}^\infty \frac{s^2}4 (4\pi)^{-1/2} e^{-s^2/4} ds \\
&=  \frac{d}{du} \bigg|_{u=1} \int_{-\infty}^\infty  (4\pi)^{-1/2} e^{-s^2/4u} ds
=  \frac{d}{du} \bigg|_{u=1} u^{1/2} = \frac12.
\end{align*}
This proves (\ref{eq_int01_h2_12}), finishing the proof.
\end{proof}

\section{Bounds on the pointed Nash entropy} \label{sec_Nash_entropy}
\subsection{Statement of the results}
In this section we first recall the definition of the Nash entropy based at a given point from \cite{Hein-Naber-14}.
Next, we prove new results concerning the dependence of the Nash entropy  on the basepoint and the scale parameter.

If $(M,g)$ is a Riemannian manifold, $\tau > 0$ and $d\nu = (4\pi \tau)^{-n/2} e^{-f} dg$ is a probability measure on $M$, then we define
\[ \NN [g,f,\tau] = \int_M f \, d\nu - \frac{n}2, \qquad
\WW[g,f,\tau] = \int_M \big( \tau (|\nabla f|^2 + R) + f - n \big) d\nu. \]
Let now $(M, (g_t)_{t \in I})$ be a Ricci flow on a compact manifold and consider a conjugate heat kernel measure $d\nu_{x_0,t_0} = (4\pi \tau)^{-n/2} e^{-f} dg = K(x_0, t_0; \cdot, \cdot)dg$ based at some point $(x_0, t_0) \in M \times I$, where $\tau = t_0 - t$.

\begin{Definition}
The {\bf pointed Nash entropy at $(x_0, t_0)$} is defined as
\[ \NN_{x_0, t_0} (\tau) := \NN [g_{t_0 - \tau}, f_{t_0-\tau},\tau]. \]
We set $\NN_{x_0, t_0} (0) := 0$.
For $s < t_0$, $s \in I$, we also write
\[ \NN^*_s (x_0, t_0) := \NN_{x_0, t_0} (t_0 - s). \]
\end{Definition}

The following proposition summarizes the basic (mostly well known) properties of the pointed Nash entropy (see also \cite{Hein-Naber-14}).

\begin{Proposition} \label{Prop_NN_basic_properties}
The expression $\NN_{x_0, t_0} (\tau)$ is continuous for $\tau \geq 0$ and if $R(\cdot, t_0 - \tau) \geq R_{\min}$, then for $\tau > 0$
\begin{equation} \label{eq_NN_basic_1}
  \NN_{x_0, t_0} (0)  = 0,  
\end{equation}
\begin{equation} \label{eq_NN_basic_2}
 \frac{d}{d\tau} \big(\tau \NN_{x_0, t_0} (\tau)\big) = \WW[g_{t_0-\tau}, f_{t_0 - \tau}, \tau] \leq 0, 
\end{equation}
\begin{equation} \label{eq_NN_basic_22}
\frac{d^2}{d\tau^2} \big( \tau \NN_{x_0, t_0} (\tau) \big) = - 2\tau \int_M \Big| \Ric_{t_0 - \tau} + \nabla^2 f_{t_0-\tau} - \frac1{2\tau} g_{t_0 - \tau} \Big|^2 d\nu_{t_0 - \tau} \leq 0,
\end{equation}
\begin{equation} \label{eq_NN_basic_3}
 - \frac{n}{2\tau} + R_{\min} \leq  \frac{d}{d\tau} \NN_{x_0, t_0} (\tau) \leq 0. 
\end{equation}
If $\tau_1 \leq \tau_2$ and $R \geq R_{\min}$ on $M \times [t_0 - \tau_2, t_0 - \tau_1]$, then
\begin{equation} \label{eq_NN_doubling}
 \NN_{x_0, t_0} (\tau_1) - \frac{n}2 \log \Big( \frac{\tau_2}{\tau_1} \Big( 1 - \frac2n R_{\min} (\tau_2 - \tau_1) \Big) \Big) \leq \NN_{x_0, t_0} (\tau_2) \leq \NN_{x_0, t_0} (\tau_1). 
\end{equation}
\end{Proposition}

Note that (\ref{eq_NN_basic_2}) implies that
\[ \NN_{x_0, t_0} (\tau) \geq \frac1{\tau} \int_0^{\tau} \mu[g_{t_0- \td\tau}, \td\tau] d\td\tau \geq  \mu[g_{t_0- \tau}, \tau] , \]
where the latter denotes Perelman's $\mu$-functional.
So a lower bound on the pointed Nash entropy is common in Ricci flows with non-degenerate initial data.

As mentioned in the introduction, $\NN_{x_0, t_0} (\tau)$ is comparable to $V(x_0, \sqrt{\tau})$ for spaces with lower Ricci curvature bounds, where
\begin{equation} \label{eq_def_VV}
 V(x_0,r) := \log \big( r^{-n} |B(x, r)| \big). 
\end{equation}
In fact, in the case of the static, Ricci flat Ricci flow, the quantity $\exp (\NN_{x_0, t_0} (\tau))$ can be bounded from above and below by $\tau^{-n/2} |B(x_0, t_0, \sqrt{\tau})|_{t_0}$ up to a uniform factor (see Theorems~\ref{Thm_NLC}, \ref{Thm_upper_volume_bound}).
So the bound (\ref{eq_NN_doubling}) is similar to a doubling property.

The main new result of this section is a result concerning the dependence of $\NN_{x_0, t_0} (\tau)$ on $(x_0, t_0)$.
For the purpose of clarity, we will consider the expression $\NN_s^* (x, t)$ instead, where $s \in I$ denotes a fixed time, and we view $x \in M, t > s$ as free variables.

\begin{Theorem} \label{Thm_NN_dependence}
If $R (\cdot, s) \geq R_{\min}$ for some fixed $s \in I$, then on $M \times (I \cap (s, \infty))$ we have
\begin{equation} \label{eq_Thm_NN_dependence}
 |\nabla \NN^*_s |\leq \Big( \frac{n}{2(t-s)}  -  R_{\min} \Big)^{1/2}, \qquad - \frac{n}{2(t-s)} \leq  \square \NN^*_s \leq 0. 
\end{equation}
\end{Theorem}

Let us digest the statement of this theorem.
First, note that the lower scalar curvature bound is often available on a Ricci flow due to Lemma~\ref{Lem_lower_scal}.
Apart from this bound, the statement of Theorem~\ref{Thm_NN_dependence} is local in the sense that it does not depend on any global geometric bound.
In this respect, the gradient bound in (\ref{eq_Thm_NN_dependence}) is an improvement of \cite[Theorem~1.17]{Hein-Naber-14}.
Next, consider the second bound (\ref{eq_Thm_NN_dependence}).
The lower bound on $\square \NN^*_s$ is a consequence of the convexity of $\NN$.
The upper bound on $\square \NN^*_s$ is somewhat surprising and is new to the author's knowledge\footnote{The bound even seems new in the static case. Please email me if that's not the case.}.
Note that both bounds in (\ref{eq_Thm_NN_dependence}) are optimal (possibly up to a dimensional constant) and collapsing-independent.
To see this, let us entertain one more time the analogy with the steady case, in which $\NN^*_s (x,t) \approx V (x, \sqrt{t-s})$ from (\ref{eq_def_VV}).
Under a collapse, we have $V(x,r) \to -\infty$.
However, due to volume comparison (see also (\ref{eq_norm_volume_variation})), differences of the form $V (x_1, r_1) - V(x_2, r_2)$ remain uniformly bounded, in terms of $d(x_1, x_2), r_1, r_2$.

The bounds from Theorem~\ref{Thm_NN_dependence} allow us to compare $\NN^*_s$ based at different points in spacetime:

\begin{Corollary} \label{Cor_NN_variation}
If $R (\cdot, t^*) \geq R_{\min}$ and $s < t^* \leq t_1, t_2$, $s, t^*, t_1, t_2 \in I$, then for $x_1, x_2 \in M$
\begin{equation} \label{eq_NNx1t1_NNx2t2}
 \NN^*_s(x_1, t_1) - \NN^*_s(x_2, t_2)  \leq \Big( \frac{n}{2(t^*-s)} -  R_{\min} \Big)^{1/2}   d_{W_1}^{g_{t^*}} (\nu_{x_1, t_1}(t^*), \nu_{x_2, t_2}(t^*))    + \frac{n}2 \log \Big( \frac{t_2-s}{t^*-s} \Big) . 
\end{equation}
\end{Corollary}

Note, again that the right-hand side of (\ref{eq_NNx1t1_NNx2t2}) is independent of any pointed Nash entropy.
So the estimate does not deteriorate under a collapse.
We remark that in Section~\ref{sec_variance_parab_nbhd} we will define of the notion of a $P^*$-parabolic neighborhood $P := P^* (x,t; A, - T^-, T^+)$.
Using this notion, Corollary~\ref{Cor_NN_variation} can be viewed as a bound on the oscillation of $\NN^*_s$ over a $P^*$-parabolic neighborhood of the following form:
If $T^- < t-s$, then
\[ \osc_{P}  \NN^*_s \leq 2 \Big( \frac{n}{2(t-s - T^-)} -  R_{\min} \Big)^{1/2}   A   + \frac{n}2 \log \Big( \frac{t-s + T^+}{t-s - T^-} \Big). \]

Lastly, we record some further interesting bounds, which may have been known before.

\begin{Proposition} \label{Prop_NN_more_basic}
If $d\nu = (4\pi \tau)^{-n/2} e^{-f}dg$ denotes the conjugate heat kernel based at $(x_0, t_0)$ and $R(\cdot, t_0 - \tau) \geq R_{\min}$, then
\begin{align}
 \int_M \tau ( |\nabla f|^2 + R) d\nu_{t_0-\tau} &\leq \frac{n}2, \label{eq_nabf_R} \\
\int_M  \Big( f - \NN_{x_0, t_0} (\tau) - \frac{n}2 \Big)^2 d\nu_{t_0-\tau} &\leq n- 2 R_{\min}\tau .\label{eq_fNNn_sq}
\end{align}
\end{Proposition}
\medskip

\subsection{Proofs of Propositions~\ref{Prop_NN_basic_properties} and \ref{Prop_NN_more_basic}.}

\begin{proof}[Proof of Proposition~\ref{Prop_NN_basic_properties}.]
The bound (\ref{eq_NN_basic_1}) holds by definition and the bounds (\ref{eq_NN_basic_2}), (\ref{eq_NN_basic_22}) are known from \cite{Perelman1} or follow from a direct computation (see also \cite{Topping-book}).
The first bound in (\ref{eq_NN_basic_3}) follows from
\begin{multline*}
 \frac{d}{d\tau} \big(\tau \NN_{x_0, t_0}(\tau) \big) = \WW [g_{t_0-\tau}, f_{t_0 - \tau}, \tau] 
= \int_M \tau \big( |\nabla f|^2 + R  \big)d\nu_{t_0 - \tau} + \NN_{x_0, t_0}(\tau) - \frac{n}2 \\
\geq \tau R_{\min} + \NN_{x_0, t_0}(\tau) - \frac{n}2.
\end{multline*}
The second bound in (\ref{eq_NN_basic_3}) follows using (\ref{eq_NN_basic_22}):
\begin{multline*}
\tau \frac{d}{d\tau} \NN_{x_0, t_0} (\tau) 
= \frac{d}{d\tau} \big( \tau \NN_{x_0, t_0} (\tau) \big) - \NN_{x_0, t_0} (\tau)
= \frac{d}{d\tau} \big( \tau \NN_{x_0, t_0} (\tau) \big) - \frac{1}\tau \int_0^\tau \frac{d}{d\td\tau} \big( \td\tau \NN_{x_0, t_0} (\td\tau) \big) d\td\tau \\
\leq \frac{d}{d\tau} \big( \tau \NN_{x_0, t_0} (\tau) \big) - \frac{1}\tau \int_0^\tau \frac{d}{d\tau} \big( \tau \NN_{x_0, t_0} (\tau) \big) d\td\tau
= 0.
\end{multline*}
The continuity at $\tau = 0$ follows from an asymptotic expansion of the heat kernel or from the fact that $\lim_{\tau \to 0} \WW[g_{t_0-\tau}, f_{t_0 - \tau}, \tau] = 0$ combined with (\ref{eq_NN_basic_2}), (\ref{eq_NN_basic_3}).

For (\ref{eq_NN_doubling}) recall that by Lemma~\ref{Lem_lower_scal}
\[ R \geq  \frac{n}2 \frac{R_{\min}}{\frac{n}2 - R_{\min} (t - t_0 +\tau_2) }. \]
Thus, using (\ref{eq_NN_basic_3}),
\begin{multline*}
 \NN_{x_0, t_0} (\tau_2) - \NN_{x_0, t_0}(\tau_1)
\geq  \int_{\tau_1}^{\tau_2} \frac{n}2 \Big( - \frac{1}{\tau} + \frac{R_{\min}}{\frac{n}2 - R_{\min} (\tau_2 - \tau)  } \Big) d\tau \\
\geq - \frac{n}2 \Big( \log \Big( \frac{\tau_2}{\tau_1} \Big) + \log \Big( \frac{\frac{n}2 - R_{\min} (\tau_2 - \tau_1 ) }{ \frac{n}2 } \Big) \Big) 
\geq - \frac{n}2 \log \Big( \frac{\tau_2}{\tau_1} \Big(1 - \frac2n R_{\min} (\tau_2 - \tau_1) \Big) \Big). \qedhere
\end{multline*}
\end{proof}
\bigskip

\begin{proof}[Proof of Proposition~\ref{Prop_NN_more_basic}.]
For (\ref{eq_nabf_R}) we compute using (\ref{eq_NN_basic_2}), (\ref{eq_NN_basic_3}) that
\begin{multline*}
 \int_M \tau ( |\nabla f|^2 + R ) d\nu_{t_0 - \tau}
=  \frac{n}2 + \WW[g_{t_0 - \tau}, f_{t_0 -\tau}, \tau] - \NN_{x_0, t_0} (\tau) \\
= \frac{n}2 +  \frac{d}{d\tau} \big( \tau \NN_{x_0, t_0} (\tau) \big)  - \NN_{x_0, t_0} (\tau) 
= \frac{n}2 + \tau  \frac{d}{d\tau} \NN_{x_0, t_0} (\tau) \leq \frac{n}2  . 
\end{multline*}
To see (\ref{eq_fNNn_sq}), we apply the $L^2$-Poincar\'e inquality from \cite[Theorem~1.10]{Hein-Naber-14} (see also Theorem~\ref{Thm_Poincare}) to
\[ \int_M  \Big( f - \NN_{x_0, t_0} (\tau) - \frac{n}2 \Big) d\nu_{t_0-\tau} = 0 \]
and
\[ 2\tau \int_M |\nabla f|^2 d\nu_{t_0 - \tau} \leq n - 2\tau \int_{M} R \, d\nu_{t_0 - \tau}   \leq  n - 2R_{\min} \tau. \qedhere \]
\end{proof}
\medskip

\subsection{Proof of Theorem~\ref{Thm_NN_dependence}}

\begin{proof}[Proof of Theorem~\ref{Thm_NN_dependence}.]
After application of a time-shift, we may assume without loss of generality that $s = 0$.

Let us first express the pointed Nash entropy in terms of the heat kernel.
For this purpose, consider a point $(x,t) \in M$ and write $K(x,t;y, 0) = (4\pi t)^{-n/2} e^{-f(y,0)}$ and $d\nu_0 = (4\pi t)^{-n/2} e^{-f(\cdot, 0)} dg$.
Then
\[ f(y,0) = -\frac{n}2 \log (4\pi t) - \log K(x,t;y, 0). \]
Therefore
\begin{multline*}
 \NN^*_0 (x,t) = \int_M f(y,0) K(x,t;y,0) dg_0 (y) - \frac{n}2 \\
= - \int_M  K(x,t;y,0)  \log K(x,t;y, 0) dg_0 (y)   -\frac{n}2 \log (4\pi t) - \frac{n}2.
\end{multline*}
Next, for any vector $v \in T_x M$ with $|v|_t = 1$ we have
\[ \int_M \partial_v K(x,t;y,0) dg_0(y)  = \partial_v \int_M K(x,t;y,0) dg_0(y) = \partial_v 1 = 0. \]
Thus
\begin{align*}
 \partial_v \NN^*_0 (x,t) &= - \int_M \big( \partial_v K(x,t;y,0) \log K(x,t;y,0)   + \partial_v K(x,t;y,0) \big) dg_0 (y) \\
&= \int_M  \partial_v K(x,t;y,0) \Big( f (y,0) + \frac{n}2 \log(4\pi t) - 1 \Big) dg_0 (y)  \\
&= \int_M  \partial_v K(x,t;y,0) \Big( f (y,0) - \NN^*_0 (x,t) - \frac{n}2 \Big) dg_0 (y) \\
&= \int_M  \frac{\partial_v K(x,t;y,0)}{K(x,t;y,0)} \Big( f (y,0) - \NN^*_0 (x,t) - \frac{n}2 \Big) d\nu_0 (y) \\
&\leq \bigg( \int_M \Big(  \frac{\partial_v K(x,t;y,0)}{K(x,t;y,0)} \Big)^2 d\nu_0 \bigg)^{1/2} \bigg( \int_M  \Big( f - \NN^*_0 (x,t) - \frac{n}2 \Big)^2 d\nu_0 \bigg)^{1/2} .
\end{align*}
So by Propositions~\ref{Prop_nab_K_bounds}, \ref{Prop_NN_more_basic} we have
\[ |\nabla \NN^*_0|^2(x,t) \leq \frac{1}{2t} (n - 2R_{\min} t). \]

Lastly, we compute
\begin{align*}
 \square \NN^*_0 (x,t) 
&= - \int_M \square_x \big( K(x,t;y,0) \log K(x,t;y,0) \big) dg_0 (y) - \frac{n}{2t} \\
&= - \int_M \Big( \partial_t K(x,t;y,0) \log K(x,t;y,0) + \partial_t K(x,t;y,0) \\
&\qquad - \triangle_x K(x,t;y,0) \log K(x,t;y,0) - \triangle_x K(x,t;y,0) - \frac{|\nabla_x K(x,t;y,0)|^2}{K(x,t;y,0)} \Big)dg_0(y) - \frac{n}{2t} \\
&= \int_M \bigg( \frac{|\nabla_x K(x,t;y,0)|}{K(x,t;y,0)}  \bigg)^2 d\nu_0 - \frac{n}{2t}. 
\end{align*}
The desired bounds now follow using Proposition~\ref{Prop_nab_K_bounds}.
\end{proof}

\subsection{Proof of Corollary~\ref{Cor_NN_variation}}

\begin{proof}[Proof of Corollary~\ref{Cor_NN_variation}.]
Without loss of generality, we may assume that $s = 0$.
Denote by $d\nu^i = (4\pi \tau_i)^{-n/2} e^{-f_i}dg$ the conjugate heat kernels based at $(x_i, t_i)$, $i = 1,2$.
By Theorem~\ref{Thm_NN_dependence} we have
\begin{equation} \label{eq_NN_xiti_int}
 \NN^*_0 (x_i, t_i) \leq \int_M \NN^*_0 (\cdot, t^*) d\nu^i_{t^*} \leq \NN^*_0 (x_i, t_i) + \frac{n}2 \int_{t^*}^{t_i} \frac{dt}{t} = \NN^*_0 (x_i, t_i) + \frac{n}2  \log \Big( \frac{t_i}{t^*} \Big). 
\end{equation}
Furthermore, due to the gradient bound in (\ref{eq_Thm_NN_dependence}) we have
\begin{equation}
 \bigg| \int_M \NN^*_0 (\cdot, t^*) d\nu^1_{t^*} - \int_M \NN^*_0 (\cdot, t^*) d\nu^2_{t^*} \bigg| \leq \Big( \frac{n}{2t^*} - 2 R_{\min} \Big)^{1/2}
d_{W_1}^{g_{t^*}} ( \nu^1_{t^*}, \nu^2_{t^*} ) . \label{eq_intNN_intNN}
\end{equation}
The desired bound now follows by combining (\ref{eq_NN_xiti_int}) and (\ref{eq_intNN_intNN}).
\end{proof}
\bigskip

\section{Lower volume bounds on distance balls} \label{sec_lower_vol}
\subsection{Statement of the results}
In this section we discuss how entropy bounds guarantee lower volume bounds, aka conventional non-collapsing bounds, of distance balls.
These bounds will be of the form $c \exp ( \NN_{x,t}(r^2)) r^n$, which is optimal, since we will prove a reverse bound in Theorem~\ref{Thm_upper_volume_bound}.

Let $(M, (g_t)_{t \in I})$ be a Ricci flow on a compact manifold.

Our first main result is a slight generalization of Perelman's No Local Collapsing Theorem \cite{Perelman1}.
It states that, given a local upper scalar curvature bound, we obtain a lower volume bound on a distance ball.

\begin{Theorem} \label{Thm_NLC}
If $R \leq r^{-2}$ on $B(x,t,r)$ and $[t-r^2, t] \subset I$, then
\[ |B(x,t, r)|_t \geq c \exp ( \NN_{x,t}(r^2)) r^n. \]
\end{Theorem}

Note that the upper scalar curvature bound is indispensable, as one can for example see on the round shrinking sphere, cylinder or Bryant soliton.
However, in the absence of an upper curvature bound, we can still find a similar volume bound if $(x,t)$ is an $H_n$-center of a conjugate heat kernel.

\begin{Theorem} \label{Thm_lower_volume_H_center}
Suppose that $[t-r^2, t] \subset I$ and that $(z,t-r^2) \in M \times I$ is an $H_n$-center of some point $(x,t) \in M \times I$.
If $R(\cdot, t-r^2) \geq R_{\min}$, then we have
\[ \big| B(z,t-r^2, \sqrt{2H_n} r)\big|_{t-r^2} \geq c(R_{\min} r^2) \exp ( \NN_{x,t}(r^2)) r^n,\]
where $c(R_{\min}r^2) =  c \exp ( - 2(n-2R_{\min}r^2)^{1/2}) $.
\end{Theorem}

Note that by a simple containment relationship, Theorem~\ref{Thm_lower_volume_H_center} also implies a lower volume bound for a ball of radius $(\sqrt{2H_n} +D)r$ if $(z, t-r^2)$ is $Dr$-far away from an $H_n$-center.
Furthermore, using Proposition~\ref{Prop_NN_basic_properties} and Corollary~\ref{Cor_NN_variation} it is not hard to see that $\NN_{x,t}(r^2)$ can be replaced by $\NN_{z, t-r^2} (a r^2)$, if $[t- (1+a)r^2, t] \subset I$ for some $a > 0$, at the expense of a smaller constant $c$, which may depend on $a$.
Observe also that the example of the Bryant soliton $(M_{\Bry}, (g_{\Bry, t})_{t \in \IR})$ does not contradict Theorem~\ref{Thm_lower_volume_H_center}, because for large $r$ its center of rotation $x_{\Bry} \in M_{\Bry}$ at time $t-r^2$ has distance $\gtrsim r^2$ from any $H_n$-center of any point at time $t$.

Lastly, we remark that Theorems~\ref{Thm_NN_dependence}, \ref{Thm_NLC} combined with Theorem~\ref{Thm_upper_volume_bound} towards the end of this paper imply an alternative proof of \cite[8.2]{Perelman1} and \cite[Theorem~1.1]{BWang_local_entropy}.

\subsection{Proof of Theorem~\ref{Thm_NLC}}

\begin{proof}[Proof of Theorem~\ref{Thm_NLC}.]
The proof is similar to that of \cite[4.1]{Perelman1}, see also \cite[Remark~13.13]{Kleiner_Lott_notes}.
Without loss of generality, we may assume that $t=0$ and after replacing $r$ with $\frac12 r$, we may assume that $[-4r^2,0] \subset I$, which implies that $R \geq - \frac{n}2 r^{-2}$ on $M \times[-4r^2,0]$ via Lemma~\ref{Lem_lower_scal}.
Next we argue that we may impose the additional assumption $|B(x,0,r)|_0 \leq 3^n |B(x,0,\frac12 r)|_0$.
To see this, assume that we can show the lemma for some fixed $c > 0$ under this additional assumption.
Thus if the lemma fails for some $r$, then the additional assumption must be violated.
This would imply that the lemma also fails for $\frac12 r$.
Repeating this process implies that the lemma fails for $r, \frac12 r, (\frac12)^2 r, \ldots$, in contradiction to the fact that it has to hold for small enough $r$ if $c$ is chosen appropriately.
So assume from now on that $|B(x,0,r)|_0 \leq 3^n |B(x,0,\frac12 r)|_0$ and by parabolic rescaling that $r =1$.

Let $h := \varphi(d_0 (x, \cdot)) \in C^\infty_c (B(x,0,1))$, where $\varphi \equiv 1$ on $[0, \frac12]$, $0 \leq \varphi \leq 1$ and $|\varphi'| \leq 10$.
Let
\[ a := \int_{B(x,0,1)} h^2 dg_0 , \qquad  |B(x, 0, \tfrac12)|_0 \leq a \leq |B(x, 0, 1)|_0 . \]
Write $\tau := 1 - t$ and let $v = (4\pi \tau)^{-n/2} e^{-f} \in C^\infty (M \times [-1, 0])$ be the solution to the conjugate heat equation $\square^* v = 0$ with initial condition $v(\cdot, 0) = a^{-1} h^2(\cdot, 0)$.
Then $\int_M v (\cdot, -1) dg_{-1} = \int_M a^{-1} h^2(\cdot, 0) dg_0 = 1$ and
\[ v_{-1} = \int_M K(y,0; \cdot, -1) a^{-1} h^2 (y)  dg_0 (y). \]
So since 
\begin{align*}
 f_{-1} &= - \log v_{-1} -\frac{n}2  \log  (8\pi) ,  \\
 \NN_{y,0}(1) &= - \int_M K(y, 0; y', -1) \log K(y, 0; y', -1) dg_{-1} (y')  - \frac{n}2 \log (4\pi) - \frac{n}2,
\end{align*}
we obtain by Jensen's inequality and Theorem~\ref{Thm_NN_dependence} that
\begin{align*}
 \NN[g_{-1}, f_{-1}, 2] &= \int_M f_{-1} v_{-1} dg_{-1} - \frac{n}2 
\geq - \int_M  v_{-1} \log v_{-1} dg_{-1} - C \\
&\geq - \int _M  \int_M K(y, 0; y', -1) \log K(y, 0; y', -1) dg_{-1} (y') a^{-1} h^2(y) dg_0 (y) - C \\
&\geq \int_{B(x,0,1)}   \NN_{y, 0}(1) a^{-1} h^2(y,0) dg_0 (y) - C
\geq \NN_{x, 0}(1) - C.
\end{align*}
By the monotonicity of the $\WW$-functional \cite{Perelman1}, we have
\begin{multline} \label{eq_WW_geq_NN_1_C}
\lim_{\eps \searrow 0} \WW[g_{-\eps}, f_{-\eps}, 1+\eps] \geq \int_{0}^1 \WW[g_{-\tau}, f_{-\tau}, 1+\tau] d\tau 
= \int_0^1 \frac{d}{d\tau} \big((1+\tau) \NN[g_{-\tau}, f_{-\tau}, 1+\tau] \big) d\tau \\ 
= 2\NN[g_{-1}, f_{-1}, 2] - \NN[g_{0}, f_{0}, 1] \geq 2\NN_{x, 0}(1) - C - \NN[g_{0}, f_{0}, 1]. 
\end{multline}

The remainder of the proof is similar to that of \cite[4.1]{Perelman1}.
We first express (\ref{eq_WW_geq_NN_1_C}) in terms of $a, h$ and use the bound $R \leq 1$ on $B(x, 0, 1)$:
\[ \int_{B(x, 0, 1)} \bigg( \frac{|\nabla a^{-1} h^2|^2}{a^{-1} h^2} + (-n+1) a^{-1} h^2 - 2a^{-1} h^2 \log (a^{-1} h^2) \bigg) dg_0 \geq 2\NN_{x, 0}(1) - C.  \]
This implies
\[ a^{-1} \int_{B(x, 0, 1)} ( 4 |\nabla h|^2 - h^2 \log (h^2) ) dg_0 + 2 \log a \geq 2\NN_{x, 0}(1) - C.  \]
Using the additional assumption, it follows that
\[ \tfrac12 \, 3^n (400+ e^{-1} ) + \log |B(x, 0, 1)|_0 \geq \tfrac12 (400 + e^{-1} ) \frac{|B(x,0, 1)|_0}a + \log a \geq \NN_{x, 0}(1) - C,  \]
which implies the desired bound.
\end{proof}
\bigskip

\subsection{Proof of Theorem~\ref{Thm_lower_volume_H_center}}

\begin{proof}[Proof of Theorem~\ref{Thm_lower_volume_H_center}.]
After parabolic rescaling and application of a time-shift, we may assume that $r = 1$, $t = 1$.
Denote by $d\nu = (4\pi \tau)^{-n/2} e^{-f} dg$ the measure of the conjugate heat kernel based at $(x,0)$ and set $B := B(z,0, 2 \sqrt{H_n} )$.
Then by Proposition~\ref{Prop_nu_ball_bound}
\begin{equation} \label{eq_nu_0_B_12}
 \nu_{0} ( B)  \geq \frac12. 
\end{equation}

By Proposition~\ref{Prop_NN_more_basic} we have
\[ \int_M \Big| f - \NN_{x,1}(1) - \frac{n}2 \Big| d\nu_{0}
\leq
\bigg( \int_M \Big( f - \NN_{x,1}(1) - \frac{n}2 \Big)^2 d\nu_{0} \bigg)^{1/2}  \leq (n- 2 R_{\min})^{1/2}. \]
So
\[ \frac1{\nu_{0} (B)} \int_B f \, d\nu_{0} 
\geq \NN_{x,1}(1) + \frac{n}2 - \frac1{\nu_{0} (B)} \int_B \Big| f - \NN_{x,1}(1) - \frac{n}2 \Big| d\nu_{0} 
\geq  \NN_{x,1}(1) + \frac{n}2 - 2(n- 2 R_{\min})^{1/2}. \]
Let $u := (4\pi )^{n/2} e^{-f} / \nu_{0} (B)$.
It follows using (\ref{eq_nu_0_B_12}) that $\int_B u \, dg_{0} = 1$ and
\begin{multline*}
 \int_B (\log u) u \, dg_{0}
 = - \frac1{\nu_0 (B)} \int_B f \, d\nu_0 - \log (\nu_0 (B)) + \frac{n}2 \log (4\pi) \\
  \leq  -\NN_{x,1}(1) -  \frac{n}2 + 2(n- 2 R_{\min})^{1/2} + \log 2 + \frac{n}2 \log (4\pi). 
\end{multline*}
By Jensen's inequality applied to the convex function $x \mapsto x\log x$ we get
\[ \log \bigg( \frac1{|B|_{0}} \int_B  u \, dg_{0} \bigg) \frac1{|B|_{0}} \int_B  u \, dg_{0} \leq \frac1{|B|_{0}} \int_B  (\log u) u \, dg_{0},   \]
which implies
\[ - \log  |B|_{0} \leq  \int_B  (\log u) u \, dg_{0} \leq - \NN_{x,1}(1) + C + 2 (n- 2 R_{\min})^{1/2}.  \]
This finishes the proof of the theorem.
\end{proof}
\bigskip

\section{Upper bounds on the heat kernel and its gradient} \label{sec_upper_HK_bounds}
\subsection{Statement of the results}
The results in this section concern the heat kernel $K(x,t;y,s)$.
Let again $(M,(g_t)_{t \in I})$ be a Ricci flow on a compact $n$-dimensional manifold.

We will first prove an $L^\infty$-bound of $K(x,t;y,s)$ in terms of the pointed Nash entropy at $(x,t)$.
This bound is optimal up to a factor, which is independent of the degree of collapsedness.

\begin{Theorem} \label{Thm_HK_Linfty_bound}
If $[s,t] \subset I$ and $R \geq R_{\min}$ on $M \times [s,t]$, then
\[ K(x,t;y,s) \leq \frac{C(R_{\min} (t-s))}{(t-s)^{n/2}} \exp ( - \NN_{x,t}(t-s) ), \]
where we may choose $C(R_{\min} (t-s)) = C_0 \cdot (n -R_{\min}(t-s))^{n/2}$ for some dimensional constant $C_0 < \infty$.
\end{Theorem}

A similar bound was shown in \cite{Cao_Zhang_conj_h_e, Zhang-noninflating} using different methods.
However this bound depended on a {\it global} bound on the $\mu$-functional.

Theorem~\ref{Thm_HK_Linfty_bound} will imply the following stronger upper pointwise Gaussian bounds.

\begin{Theorem} \label{Thm_HK_pointwise_Gaussian}
Suppose that $[s,t] \subset I$ and $R \geq R_{\min}$ on $M \times [s,t]$.
Let $(z,s) \in M \times I$ be an $H_n$-center of a point $(x,t) \in M \times I$.
Then for any $\eps > 0$ and $y, y' \in M$
\begin{align}
 K(x,t;y,s) & \leq  \frac{C(R_{\min} (t-s), \eps) \exp ( - \NN_{x,t}(t-s) )}{(t-s)^{n/2}} \exp \bigg({- \frac{d^2_s(z, y)}{(8+\eps) (t-s)}   }\bigg) , \label{eq_Gaussian_HK_bound_1} \\
 K(x,t;y,s) K(x,t;y',s) & \leq  \frac{C(R_{\min} (t-s), \eps) \exp ( -2 \NN_{x,t}(t-s) )}{(t-s)^{n}} \exp \bigg({- \frac{d^2_s(y, y')}{(8+\eps) (t-s)}   }\bigg) .\label{eq_Gaussian_HK_bound_2} 
\end{align}
Moreover, if $[s - r^2, t] \subset I$ for some $r^2 \geq \eps (t-s)$, then we may replace $\NN_{x,t}(t-s)$ by $\NN_{y,s} (r^2)$ in (\ref{eq_Gaussian_HK_bound_1}), (\ref{eq_Gaussian_HK_bound_2}).
\end{Theorem}

Note that the factor $8+\eps$ is not optimal in (\ref{eq_Gaussian_HK_bound_1}), but it is almost optimal in (\ref{eq_Gaussian_HK_bound_2}).

Lastly, we will establish a gradient bound on the heat kernel.

\begin{Theorem} \label{Thm_HK_gradient_bound}
If $[s,t] \subset I$ and $R \geq R_{\min}$ on $M \times [s,t]$, then there are constants $C( R_{\min}(t-s)), C_0( R_{\min}(t-s)) < \infty$ such that
\[ \frac{|\nabla_x K|(x,t; y,s)}{K(x,t;y,s)} \leq \frac{C}{(t-s)^{1/2}} \sqrt{ \log \bigg( \frac{C_0 \exp (- \NN_{x,t}(t-s))}{(t-s)^{n/2} K(x,t;y,s)} \bigg) }. \]
\end{Theorem}

Note that this bound is asymptotically similar to the bound in Theorem~\ref{Thm_gradient_estimate}.
In fact, given a {\it global} bound on the Nash entropy, it is not hard to deduce a similar bound by combining Theorems~\ref{Thm_gradient_estimate}, \ref{Thm_HK_Linfty_bound}.
The difficulty in the proof of Theorem~\ref{Thm_HK_gradient_bound} is, however, that the bound only depends on the pointed Nash entropy at $(x,t)$.

\subsection{Proof of Theorem~\ref{Thm_HK_Linfty_bound}}

\begin{proof}[Proof of Theorem~\ref{Thm_HK_Linfty_bound}.]
We first show:

\begin{Claim}
If the first part of the theorem is true, then the additional statement on the constant $C$ is also true.
\end{Claim}

\begin{proof}
After application of a time-shift, we may assume that $s = 0$.
By Lemma~\ref{Lem_lower_scal} we have
\begin{equation} \label{eq_lower_R}
 R (\cdot, t) \geq \frac{n}2 \frac{R_{\min}}{\frac{n}2 - R_{\min} t }. 
\end{equation}
Since the flow exists up to time $t$, this implies that $R_{\min} \leq \frac{n}{2t}$.
By applying the theorem to $[\frac12 t, t]$, we obtain that
\begin{equation} \label{eq_upper_K}
 K(x,t; \cdot, \tfrac12 t) \leq \frac{C}{(\frac12 t)^{n/2}} \exp (- \NN_{x,t} (\tfrac12 t) )  \leq \frac{C}{t^{n/2}} \exp (- \NN_{x,t} ( t) ), 
\end{equation}
where $C < \infty$ is a dimensional constant.
Since $K(x,t; \cdot, \cdot)$ satisfies the conjugate heat equation $\square^* K(x,t; \cdot, \cdot) = 0$, we obtain from the maximum principle and (\ref{eq_lower_R}), (\ref{eq_upper_K}) that
\begin{align*}
 \max_M K(x,t; \cdot, 0)
&\leq \exp \bigg( - \int_0^{t/2} \min_M R(\cdot, t') dt'  \bigg) \max_M K(x,t; \cdot, \tfrac12 t) \\
&\leq  \exp \bigg( - \int_0^{t/2}  \frac{n}2 \frac{R_{\min}}{\frac{n}2 - R_{\min} t' } dt' \bigg) \max_M K(x,t; \cdot, \tfrac12 t) \\
&=\bigg( \frac{n}2 - R_{\min} \frac{t}2 \bigg)^{n/2} \max_M K(x,t; \cdot, \tfrac12 t) \\
&\leq \frac{C (n - R_{\min} t)^{n/2} }{t^{n/2}}  \exp (- \NN_{x,t} ( t) ). \qedhere 
\end{align*}
\end{proof}
\medskip

After parabolic rescaling and application of a time-shift, we may assume that $s = 0$ and $t = 1$.
By the Claim we may disregard the dependence of $C$ on $R_{\min}$.
Therefore, from now on all generic constants may depend on $n, R_{\min}$, without further specification.

Fix $y \in M$ and write
\[ u := K(\cdot, \cdot; y,0). \]
So $\square u = 0$.
In the following we will show that
\begin{equation} \label{eq_K_bound_W}
 u(x,t) \leq \frac{Z}{t^{n/2}} \exp ( - \NN^*_0 (x,t)) \qquad \text{for all} \quad (x,t) \in M \times (0,1], 
\end{equation}
where $Z$ is a constant whose value will only depend on $n$ and $R_{\min}$.
Note that such a bound holds for \emph{some} large $Z$, which may depend on the underlying Ricci flow; see for instance \cite[Theorem~26.25]{Chow_book_series_Part_III}.
So by induction, it is enough to show that if (\ref{eq_K_bound_W}) holds for some $Z$, then it also holds for $Z$ replaced with $Z/2$ if $Z \geq \underline{Z}$.
We can therefore assume that (\ref{eq_K_bound_W}) is true for some $Z$.
Again by parabolic rescaling, it suffices to show that (\ref{eq_K_bound_W}) holds at time $t = 1$ for $Z$ replaced with $Z/2$ if $Z \geq \underline{Z}$.

Since
\[ \frac{d}{dt} \int_M u \, dg_t = - \int_M R u\, dg_t \leq - R_{\min} \int_M u \, dg_t, \]
we have
\begin{equation} \label{eq_int_u_L_infty_HK}
 \int_M u \, dg_t \leq e^{- R_{\min}} \leq C \qquad \text{for all} \quad t \in (0,1]. 
\end{equation}
Note that this bound is independent of $Z$.

Next, we will derive a gradient bound on $u$.
For this purpose set
\[ v := \big( t - \tfrac12 \big) |\nabla u|^2 + u^2, \]
and observe that for $t \in [\frac12 ,1]$ we have
\[ \square v = |\nabla u|^2 - 2\big( t - \tfrac12 \big) |\nabla^2 u|^2 - 2 |\nabla u|^2 \leq 0. \]
So for any $(x, t) \in M \times (\frac12, 1]$ we have using (\ref{eq_K_bound_W})
\begin{multline} \label{eq_t_m_12_int_gradient}
 \big( t - \tfrac12 \big) |\nabla u|^2 (x,t) 
\leq v(x,t)
\leq \int_M K(x,t; \cdot, \tfrac12) v(\cdot, \tfrac12 ) dg_{\frac12} \\
= \int_M K(x,t; \cdot, \tfrac12) u^2(\cdot , \tfrac12 ) dg_{\frac12}
\leq 2^{n}Z^2 \int_M K(x,t ; \cdot, \tfrac12 ) \exp \big({ - 2 \NN^*_0 (\cdot, \tfrac12 ) } \big) dg_{\frac12} .
\end{multline}
Let $(z,\frac12)$ be an $H_n$-center of $(x,t)$ and recall that $d_{W_1}^{g_{\frac12}} ( \delta_z, \nu_{x,t; \frac12} ) \leq \sqrt{\frac12 H_n}$.
By Theorem~\ref{Thm_NN_dependence}, Corollary~\ref{Cor_NN_variation}  we have for any $y' \in M$
\begin{equation} \label{eq_z12xtNN}
- \NN^*_0 (y',\tfrac12) \leq - \NN^*_0 (z, \tfrac12) + C d_{\frac12} (z, y')
\leq -\NN^*_0 (x,t) + C + C d_{\frac12} (z, y').
\end{equation}
So by Theorem~\ref{Thm_Gaussian_integral_bound}
\begin{align*} 
 \int_M K(x,t ; \cdot, \tfrac12 ) & \exp \big({ - 2 \NN^*_0 (\cdot, \tfrac12 ) } \big) dg_{\frac12}
\leq C \exp ({ - 2 \NN^*_0 (x, t ) } )  \int_M K(x,t ; \cdot, \tfrac12 ) \exp (C d_{\frac12} (z, \cdot)) dg_{\frac12} \displaybreak[1] \\
&\leq C \exp ({ - 2 \NN^*_0 (x, t ) } ) \bigg( \int_0^\infty e^{Cr} \int_{M \setminus B(z,\frac12,r)} K(x,t ; \cdot, \tfrac12 )  dg_{\frac12} dr + 1 \bigg)\displaybreak[1] \\
&\leq C \exp ({ - 2 \NN^*_0 (x, t ) } )\bigg( \int_0^\infty e^{Cr} \exp \bigg({ - \frac{\big(r - \sqrt{2H_n (t-\frac12)} \big)_+^2}{8 (t - \frac12)}}\bigg) dr +1 \bigg) \displaybreak[1] \\
&\leq C \exp ({ - 2 \NN^*_0 (x, t ) } ) \bigg( \sqrt{t- \tfrac12} \int_0^\infty e^{C\sqrt{t- \frac12} \, r} \exp \bigg({ - \frac{\big(r - \sqrt{2H_n} \big)_+^2}{8}}\bigg) dr +1 \bigg)\\
&\leq C \exp ({ - 2 \NN^*_0 (x, t ) } ).
\end{align*}
Combining this with (\ref{eq_t_m_12_int_gradient}) implies that for all $(x,t) \in M \times [\frac34, 1]$
\begin{equation} \label{eq_nabu_W}
 |\nabla u|(x,t) \leq C 
  Z \exp \big({ -  \NN^*_0 (x, t ) } \big). 
\end{equation}

We will now combine (\ref{eq_int_u_L_infty_HK}) and (\ref{eq_nabu_W}) to show that (\ref{eq_K_bound_W}) holds at time $1$ for $Z$ replaced with $Z/2$ if $Z \geq \underline{Z}$.
Fix $(x,1) \in M \times I$ and write $d\nu = K(x,1; \cdot, \cdot) dg$.
 Let $\rho \in (0, \frac12]$ be a constant whose value we will determine later, set $t_1 := 1- \rho^2 \in [\frac34,1]$ and let $(z_1, t_1)$ be an $H_n$-center of $(x,1)$.
By Corollary~\ref{Cor_NN_variation} we find that if $\rho \leq \ov\rho$, then
\begin{multline*}
  - \NN^*_0 (z_1, t_1) \leq  - \NN^*_0 (x,1) + C d^{g_{t_1}}_{W_1} (\delta_{z_1}, \nu_{x,1;t_1})
\leq - \NN^*_0 (x,1) + C\sqrt{\Var (\delta_{z_1}, \nu_{x,1;t_1})} \\
\leq - \NN^*_0 (x,1) + C\sqrt{H_n \rho^2}
\leq - \NN^*_0 (x,1) + \log 2 
\end{multline*}
and therefore, by Theorem~\ref{Thm_NN_dependence}
\begin{equation} \label{eq_NN_NN_log2}
 - \NN^*_0 (\cdot, t_1) \leq - \NN^*_0 (x,1) + \log 2 +  C d_{t_1} (z_1, \cdot) . 
\end{equation}
 Set
 \[ B := B(z_1, t_1, \sqrt{100H_n} \,\rho)  \]
 and recall that
\begin{equation} \label{eq_ux1_int_int}
  u(x,1) = \int_B u \,d\nu_{t_1} + \int_{M \setminus B} u  \,d\nu_{t_1}. 
\end{equation}

We will first bound the first integral in (\ref{eq_ux1_int_int}).
By Theorem~\ref{Thm_lower_volume_H_center} we have
\begin{equation} \label{eq_B_vol_HK_L_infty}
 |B|_{t_1} \geq c \exp (\NN_{x,1} ( \rho^2 ) ) \rho^n
 \geq c \exp ( \NN^*_0 (x,1) ) \rho^n. 
\end{equation}
Combining (\ref{eq_nabu_W}) and (\ref{eq_NN_NN_log2}) yields
\[ |\nabla u| (\cdot, t_1) \leq C Z \exp ( - \NN^*_0(x,1) ) \qquad \text{on} \quad B. \]
Thus for any $x', x'' \in B$
\[  u (x', t_1)   \leq u(x'', t_1) +  C  Z  \exp ( - \NN^*_0 (x,1) ) \rho. \]
Integrating this over $B$ and using (\ref{eq_int_u_L_infty_HK}), (\ref{eq_B_vol_HK_L_infty}) implies
\begin{equation*}
 u(x', t_1)  \leq  \frac1{ |B|_{t_1}} \int_B u \, dg_{t_1} + C Z \exp ( - \NN^*_0 (x,1) ) \rho 
\leq C( C \rho^{-n} + Z \rho ) \exp ( - \NN^*_0 (x,1)). 
\end{equation*}
It follows that
\begin{equation} \label{eq_int_B_u}
 \int_B u \, d\nu_{t_1} \leq C( C \rho^{-n} + Z \rho ) \exp ( - \NN^*_0 (x,1)). 
\end{equation}

To bound the second term in (\ref{eq_ux1_int_int}), first observe that by Proposition~\ref{Prop_nu_ball_bound}
\[ \nu_{t_1} (M \setminus B)
\leq \frac1{100}. \]
So if $\rho \leq \ov\rho$, then we obtain using (\ref{eq_K_bound_W}), (\ref{eq_NN_NN_log2}), the bound
\[ e^s \leq 1 + \rho e^{s/\rho} \]
and Theorem~\ref{Thm_Gaussian_integral_bound} that
\begin{align}
 \int_{M \setminus B} u \, d\nu_{t_1} 
&\leq 2 Z \int_{M \setminus B} \exp ( - \NN^*_0 (\cdot, t_1) ) d\nu_{t_1}  \notag \\
&\leq 4 Z \exp ( - \NN^*_0 (x, 1) ) \int_{M \setminus B} \exp \big( C d_{t_1} (z_1, \cdot) \big) d\nu_{t_1} \notag \displaybreak[1] \\
&\leq 4 Z \exp ( - \NN^*_0 (x, 1) )  \int_{M \setminus B} \big( 1 + \rho \exp \big( C \rho^{-1} d_{t_1} (z_1, \cdot) \big) \big) d\nu_{t_1} \notag \displaybreak[1]  \\
&\leq \bigg( \frac{4}{100} + 4 \rho  \int_{M}\exp \big( C  \rho^{-1} d_{t_1} (z_1, \cdot) \big) d\nu_{t_1} \bigg) Z \exp ( - \NN^*_0 (x, 1) ) \notag \displaybreak[1]  \\
&\leq \bigg( \frac{4}{100} + 4 \rho + C   \int_0^\infty e^{C r/\rho}  \int_{M \setminus B(z_1, t_1, r)}  d\nu_{t_1} dr \bigg) Z \exp ( - \NN^*_0 (x, 1) ) \notag \displaybreak[1]  \\
&\leq \bigg( \frac{5}{100} + C   \int_0^\infty e^{C r/\rho}  \exp \bigg({- \frac{ ( r - \sqrt{2H_n} \rho)_+^2}{8 \rho^2} }\bigg)  dr \bigg) Z \exp ( - \NN^*_0 (x, 1) ) \notag \displaybreak[1]  \\
&\leq \bigg( \frac{5}{100} + C  \rho  \int_0^\infty e^{C r}  \exp \bigg({- \frac{ ( r - \sqrt{2H_n} )_+^2}{8 } }\bigg)  dr \bigg) Z \exp ( - \NN^*_0 (x, 1) ) \notag \\
&\leq \bigg( \frac1{10} + C \rho \bigg) Z \exp ( - \NN^*_0 (x, 1) ). \label{eq_int_MB_u}
\end{align}
Combining (\ref{eq_ux1_int_int}), (\ref{eq_int_B_u}) and (\ref{eq_int_MB_u}) yields that if $\rho \leq \ov\rho$, then
\[ u(x,1) \leq \bigg(C  \rho^{-n} + \frac{Z}{10} + (Z+C) \rho  \bigg)  \exp ( - \NN^*_0 (x, 1) ). \]
So if we fix some $\rho \leq \frac1{10}$ and assume $Z \geq \underline{Z}$, then the expression in the parentheses is $\leq Z/2$, as desired.
\end{proof}
\bigskip

\subsection{Proof of Theorem~\ref{Thm_HK_pointwise_Gaussian}}
Theorem~\ref{Thm_HK_pointwise_Gaussian} will be a consequence of the following lemma:

\begin{Lemma} \label{Lem_Gaussian_estimate}
If $Q < \infty$, $Z \geq \underline{Z}(R_{\min}, \eps, Q)$, then the following holds.
Suppose that $[0,1] \subset I$ and $R \geq R_{\min}$ on $M \times [0,1]$. Assume that for any $(x,t) \in (0, \frac12] \times M$ and any $H_n$-center $(z, 0)$ of $(x,t)$ we have
\begin{equation} \label{eq_Z_Gaussian_condition_asspt}
  K(x,t; y,0)  \leq \frac{2Z \exp ( - \NN^*_0 (x,t))}{t^{n/2}} \exp \bigg({ - \frac{d^2_0(z, y)}{Q t} }\bigg).  
\end{equation}
Then for any $(x,t) \in M \times (0,1]$, $y_1, y_2 \in M$ and any $H_n$-center $(z,0)$ of $(x,t)$ we have
\begin{align}
 K(x,t; y,0)  &\leq \frac{Z \exp ( - \NN^*_0 (x,t))}{t^{n/2}} \exp \bigg({ - \frac{d^2_0(z, y)}{(8+\eps) t} }\bigg),  \label{eq_Z_Gaussian_condition_1} \\
 K(x,t; y_1,0) K(x,t; y_2, 0) &\leq \frac{Z \exp ( -2 \NN^*_0 (x,t))}{t^{n}} \exp \bigg({ - \frac{d^2_0(y_1, y_2)}{(8+\eps) t} }\bigg).\label{eq_Z_Gaussian_condition_2}
\end{align}
\end{Lemma}
\bigskip

\begin{proof}[Proof of Theorem~\ref{Thm_HK_pointwise_Gaussian} using Lemma~\ref{Lem_Gaussian_estimate}.]
After application of a time-shift and parabolic rescaling, we may assume that $s = 0$ and $t = 1$.

We claim that the bound (\ref{eq_Z_Gaussian_condition_asspt}) holds for some uniform $Q, Z$, which may depend on the geometry of the flow.
To see this, observe that by \cite[Theorem~26.25]{Chow_book_series_Part_III} there is a constant $C^*_1 < \infty$, which may depend on the geometry of the flow, such that 
\begin{equation} \label{eq_Kxty_C_star_1}
 K(x,t; y,0) \leq \frac{C^*_1}{t^{n/2}} \exp \bigg({ - \frac{d^2_0(x,y)}{C^*_1 t} }\bigg). 
\end{equation}
It follows that for some  $C^{*}_2 < \infty$, which may depend on the geometry of the flow,
\begin{equation} \label{eq_Kxty_C_star_2}
 \nu_{x,t;0} (M \setminus B(x,0, A \sqrt{t}) )
\leq \int_{A \sqrt{t}}^\infty C^*_2 r^{n-1} \frac{C^*_1}{t^{n/2}} \exp \bigg({ - \frac{r^2}{C^*_1 t} }\bigg)dr
\leq C^*_1 C^*_2 \int_{A}^\infty r^{n-1}   \exp \bigg({ - \frac{r^2}{C^*_1 } }\bigg) dr. 
\end{equation}
So choosing $A$ large enough, we find using Proposition~\ref{Prop_nu_ball_bound} that there is a constant $C^*_3 < \infty$ such that for any $H_n$-center $(z,0)$ of $(x,t)$
\[ d_0 (x, z) \leq C^*_3 \sqrt{t}. \]
Combining this with (\ref{eq_Kxty_C_star_1}) implies that for some $C^*_4 < \infty$
\begin{equation} \label{eq_Kxty_C_star_4}
 K(x,t; y,0) \leq \frac{C^*_4}{t^{n/2}} \exp \bigg({ - \frac{d^2_0(z,y)}{C^*_4 t} }\bigg). 
\end{equation}
So since $\NN^*_0 \leq 0$, we obtain that (\ref{eq_Z_Gaussian_condition_asspt}) holds for some uniform $Q, Z$.

Since (\ref{eq_Z_Gaussian_condition_asspt}) holds for some $Q, Z$, we may apply Lemma~\ref{Lem_Gaussian_estimate} and conclude that (\ref{eq_Z_Gaussian_condition_asspt}) holds for $Q=8+\eps$ and $Z$ replaced with $Z =Z/2$ as long as $Z \geq \underline{Z} (R_{\min}, \eps, Q)$.
So, in fact, (\ref{eq_Z_Gaussian_condition_asspt}) holds for $Q = 8+\eps$ and some $Z$.
We can now use induction over $Z$, while keeping $Q := 8+\eps$, and conclude that (\ref{eq_Z_Gaussian_condition_1}), (\ref{eq_Z_Gaussian_condition_2}) hold for $Z = \underline{Z} (R_{\min}, \eps, 8+\eps)$.

This implies the first part of the theorem.
It remains to argue that we may replace $\NN_{x,t} (t-s)$ by $\NN_{y,s} (r^2)$ in (\ref{eq_Gaussian_HK_bound_1}), (\ref{eq_Gaussian_HK_bound_2}).
For this purpose, assume that $[-r^2, 1] \subset I$ for some $r^2 \geq \eps$. (We still assume that $s = 0$, $t=1$.)
After replacing $\eps$ with $\eps / 2$ and setting $R_{\min} := - \frac{n}2 \eps^{-1}$, we may assume by Lemma~\ref{Lem_lower_scal} that $R \geq R_{\min}$ on $M \times [-\eps, 1]$.
By Corollary~\ref{Cor_NN_variation}, we have
\begin{align*}
 - \NN_{x,1} (1) &\leq - \NN_{x,1} (1+\eps) 
\leq - \NN_{y,0} (\eps) + C(\eps) d_{W_1}^{g_0} (\nu_{x,1;0} , \delta_y ) + C(\eps) \\
&\leq - \NN_{y,0} (r^2) + C(\eps) d_{W_1}^{g_0} (\nu_{x,1;0} , \delta_z )  + C(\eps)d_{W_1}^{g_0} (\delta_z, \delta_y ) + C(\eps) \\
&\leq - \NN_{y,0} (r^2)  + C(\eps) d_0(z,y) + C(\eps) .
\end{align*}
So by (\ref{eq_Gaussian_HK_bound_1}), with $\eps$ replaced with $\eps/2$
\begin{multline*}
 K(x,1;y,0) 
\leq C(\eps) \exp ( - \NN_{y,0} (r^2)) \exp \bigg({ -\frac{d^2_0(z,y)}{8+\frac12\eps} + C(\eps) d_0(z,y)  }\bigg) \\
\leq C(\eps) \exp ( - \NN_{y,0} (r^2)) \exp \bigg({ -\frac{d^2_0(z,y)}{8+\eps} }\bigg). 
\end{multline*}
Next, by combining (\ref{eq_Gaussian_HK_bound_1}), (\ref{eq_Gaussian_HK_bound_2}), with $\eps$ is replaced with $\eps/2$, we obtain
\begin{align*}
 K(x,1; y,0)& K(x,1; y', 0)
= \big( K(x,1; y,0) K(x,1; y', 0) \big)^{\frac{8+\eps/2}{8+\eps}} K^{\frac{\eps/2}{8+\eps}} (x,1; y,0) K^{\frac{\eps/2}{8+\eps}} (x,1; y',0)  \\
&\leq C(\eps) \exp ( - 2\NN_{x,1} (1)) \exp \bigg({-\frac{d^2_0(y, y')}{8+\eps} - \frac{\eps/2}{(8+\eps)^2} \big( d^2_0(z, y) + d^2_0(z,y') \big) }\bigg)  \displaybreak[1] \\
&\leq C(\eps) \exp ( - 2\NN_{y,0} (r^2)) \exp \bigg({-\frac{d^2_0(y, y')}{8+\eps} }\bigg).
\end{align*}
This finishes the proof of the theorem.
\end{proof}
\bigskip

\begin{proof}[Proof of Lemma~\ref{Lem_Gaussian_estimate}.]
In the following, all generic constants may depend on $n, R_{\min}$, without further specification.
By parabolic rescaling, it suffices to prove (\ref{eq_Z_Gaussian_condition_1}), (\ref{eq_Z_Gaussian_condition_2}) for $t = 1$.

Let $\alpha > 0$ be a constant whose value we will determine later.
We claim that, given the assumption of the lemma, we even have for any $x, y_1, y_2 \in M$
\begin{equation} \label{eq_Gaussian_bound_alpha}
  K(x,1; y_1,0) K(x,1; y_2, 0) \leq \alpha Z \exp ( -2 \NN^*_0 (x,t)) \exp \bigg({ - \frac{d^2_0(y_1, y_2)}{8+\eps/2} }\bigg). 
\end{equation}
If $\alpha \leq 1$, then (\ref{eq_Gaussian_bound_alpha}) implies (\ref{eq_Z_Gaussian_condition_2}).
To see how (\ref{eq_Gaussian_bound_alpha}) implies (\ref{eq_Z_Gaussian_condition_1}), let $(z, 0)$ be an $H_n$-center of $(x,1)$ and write $K(x,1; \cdot, \cdot) = (4\pi \tau)^{-n/2} e^{-f}$.
By combining Propositions~\ref{Prop_nu_ball_bound}, \ref{Prop_NN_more_basic}, we obtain that there is a point $y_2 \in B(z, 0,\sqrt{2 H_n})$ with
\[ \Big( f (y_2) - \NN^*_0 (x,1) - \frac{n}2 \Big)^2 \leq 2 (n- 2R_{\min} \tau). \]
Thus
\[ K(x,1; y_2, 0) \geq c \exp ( - \NN_0^*(x,1) ), \]
which implies in combination with (\ref{eq_Gaussian_bound_alpha}) that
\begin{align*}
   K(x,1; y_1,0)  
&\leq C\alpha Z \exp ( - \NN^*_0 (x,t)) \exp \bigg({ - \frac{d^2_0(y_1, y_2)}{8+\eps/2} }\bigg) \\
&\leq C\alpha Z  \exp ( - \NN^*_0 (x,t)) \exp \bigg({ - \frac{(d_0(y_1, z) - \sqrt{2H_n})^2}{8+\eps/2} }\bigg) \displaybreak[1] \\
&\leq C(\eps) \alpha Z  \exp ( - \NN^*_0 (x,t)) \exp \bigg({ - \frac{d_0^2(y_1, z) }{8+\eps} }\bigg) . 
\end{align*}
This implies (\ref{eq_Z_Gaussian_condition_1}) for $\alpha \leq \ov\alpha (R_{\min}, \eps)$. 

Fix $\alpha$ for the remainder of this proof.
We will also omit the dependence of any generic constant on $\alpha$ from now on.
It remains to show (\ref{eq_Gaussian_bound_alpha}).
To do this, we argue by contradiction and assume that (\ref{eq_Gaussian_bound_alpha}) is violated for some points $x, y_1, y_2 \in M$.

Set $d\nu := K(x,1; \cdot, \cdot) dg$.
Let $\theta \in (0,\frac12]$ be some constant whose value we will determine later and set
\[
d := d_0 (y_1, y_2), \qquad
 a_i := K(x, 1; y_i, 0), \qquad 
V_i := \{ K(\cdot, \theta; y_i, 0) \geq a_i /2 \} \subset M, \qquad
i = 1,2. \]
So by our contradiction assumption we have
\begin{equation} \label{eq_a1a2_contr_asspt}
 a_1 a_2 \geq \alpha Z \exp (- 2 \NN^*_0 (x,1)) \exp \bigg({ - \frac{d^2}{8+\eps/2} }\bigg). 
\end{equation}
Assuming $Z \geq \underline{Z}$, we obtain using Theorem~\ref{Thm_HK_Linfty_bound} that 
\[ d \geq 10 + \sqrt{2H_n}. \]

Let $(z', \theta) \in M \times (0,1]$ be an $H_n$-center of $(x,1)$ and set $B' := B(z', \theta, 10 d)$.
Assume that $\theta \leq \ov\theta (R_{\min})$ such that $- \theta^{-1} \leq R_{\min}$.
Then by Corollary~\ref{Cor_NN_variation} we have
\[ -\NN^*_0 (z', \theta) \leq -\NN^*_0 (x,1) + C  \theta^{-1/2} \]
and by Theorem~\ref{Thm_NN_dependence}
\begin{equation} \label{eq_Gaussian_lin_NN}
 -\NN^*_0 (\cdot, \theta) \leq -\NN^*_0 (x,1) + C  \theta^{-1/2} + C  \theta^{-1/2}  d_\theta (z', \cdot). 
\end{equation}
Using Theorems~\ref{Thm_HK_Linfty_bound}, \ref{Thm_Gaussian_integral_bound}, we can estimate
\begin{align}
 a_i &= \int_M  K(\cdot ,\theta; y_i,0) d\nu_\theta 
 \leq \frac{a_i}{2}  \int_{M \setminus V_i } d\nu_\theta + \frac{C}{\theta^{n/2}} \int_{V_i} \exp(-\NN^*_0 (\cdot, \theta) )d\nu_\theta  \notag \\
&\leq \frac{a_i}{2} + \frac{C \exp(-\NN_0^* (x,1))}{\theta^{n/2}} \int_{V_i } \exp \big(C \theta^{-1/2} (d_\theta (z', \cdot) + 1) \big)d\nu_\theta  \notag  \displaybreak[1] \\
&\leq \frac{a_i}{2} + \frac{C \exp (- \NN^*_0 (x,1) + C \theta^{-1/2} d )}{\theta^{n/2}} \bigg( \nu_\theta (V_i ) + \theta^{-1/2} \int_{10d}^\infty e^{C \theta^{-1/2} r} \nu_\theta (V_i \setminus B(z', \theta, r)) dr \bigg)  \notag  \displaybreak[1] \\
&\leq \frac{a_i}{2} + \frac{C \exp (- \NN^*_0 (x,1)  + C \theta^{-1/2} d)}{\theta^{n/2}} \bigg( \nu_\theta (V_i \cap B') 
+ \exp \bigg({ - \frac{(10d - \sqrt{2H_n})^2}{8(1-\theta)} }\bigg) \notag \\
&\qquad\qquad\qquad\qquad\qquad\qquad\qquad\qquad\qquad
+ \theta^{-1/2} \int_{10d}^\infty e^{C \theta^{-1/2} r} \exp \bigg({ - \frac{ (r - \sqrt{2H_n})^2}{8(1-\theta)} }\bigg) dr \bigg)  \notag \displaybreak[1] \\
&\leq \frac{a_i}{2} + \frac{C \exp (- \NN^*_0 (x,1) + C \theta^{-1/2} d  )}{\theta^{n/2}} \bigg( \nu_\theta (V_i \cap B') + e^{-d^2} + \theta^{-1/2} \int_{10d}^\infty  \exp \bigg({ \frac{C}{\theta} - \frac{ r^2}{10} }\bigg) dr \bigg) \notag \\
&\leq \frac{a_i}{2} + \frac{C \exp (- \NN^*_0 (x,1) + C \theta^{-1/2} d  )}{\theta^{n/2}} \bigg( \nu_\theta (V_i \cap B') + e^{C \theta^{-1} -d^2} \bigg). \label{eq_ai_ai_2_Gaussian}
\end{align}
Note that in the second last step we have used the bound
\[ C \theta^{-1/2} r - \frac{(r-\sqrt{2H_n})^2}{8(1-\theta)}
\leq C \theta^{-1} + \frac{r^2}{100} - \frac{r^2}{9} + C
\leq C \theta^{-1} - \frac{r^2}{10} + C \]
and in the last step, we have used the bound $\theta^{-1/2} \leq C e^{ \theta^{-1}}$.
Subtracting $\frac{a_i}2$ on both sides of (\ref{eq_ai_ai_2_Gaussian}), multiplying the inequality for $i = 1,2$ and applying \cite[Theorem~1.13]{Hein-Naber-14} implies, using $C \theta^{-1/2} d \leq C \theta^{-1} + d^2/2$, that
\begin{align}
 \theta^n a_1 a_2 \exp (2\NN^*_0 (x,1)) 
 &\leq C e^{ C \theta^{-1/2} d}  \big(  \nu_\theta (V_1 \cap B') + e^{C\theta^{-1} -d^2} \big) \big(  \nu_\theta (V_2 \cap B') +  e^{C \theta^{-1} - d^2} \big) \notag \\
&\leq C e^{ C \theta^{-1/2} d} \nu_\theta (V_1 \cap B') \nu_\theta (V_2 \cap B') + Ce^{C \theta^{-1/2} d+C \theta^{-1} - d^2} \notag \\
 &\leq C  e^{ C \theta^{-1/2} d}\exp \bigg({  -\frac{d^2_\theta ( V_1 \cap B', V_2 \cap B')}{8(1-\theta)} }\bigg) +C e^{C \theta^{-1} - d^2/2}. \label{eq_mutliply_a1a2}
 \end{align}
By (\ref{eq_a1a2_contr_asspt}), we have
\[  \theta^n a_1 a_2 \exp (2\NN^*_0 (x,1)) 
\geq \alpha Z \theta^n \exp \bigg({ - \frac{d^2}{8+\eps/2} }\bigg)
\geq  \theta^n \alpha Z  e^{-d^2/2}. \]
So assuming $Z \geq \underline{Z} (\theta)$, we may drop the last term in (\ref{eq_mutliply_a1a2}) and obtain that  
\[  \theta^n a_1 a_2 \exp (2\NN^*_0 (x,1)) \leq C  e^{ C \theta^{-1/2} d}\exp \bigg({  -\frac{d^2_\theta ( V_1 \cap B', V_2 \cap B')}{8(1-\theta)} }\bigg). \]
Choose points $v_i \in V_i \cap B'$ such that $d_\theta (v_1, v_2) = d_\theta ( V_1 \cap B', V_2 \cap B')$.
Then
\begin{equation} \label{eq_a1a2_bound}
\theta^n a_1 a_2 \exp (2\NN^*_0 (x,1)) \leq C \exp \bigg({ C \theta^{-1/2} d  -\frac{d^2_\theta ( v_1, v_2)}{8(1-\theta)} }\bigg). 
\end{equation}

Next, choose $H_n$-centers  $(z_i, 0)$ of $(v_i, \theta)$.
By the assumption of the lemma and (\ref{eq_Gaussian_lin_NN})
\begin{equation} \label{eq_ai_bound}
 \theta^{n/2} a_i 
\leq 4Z \exp ( - \NN^*_0 (v_i, \theta) ) \exp \bigg({ - \frac{ d^2_0 (z_i, y_i)}{Q \theta} }\bigg)
\leq 4Z \exp ( - \NN^*_0 (x, 1) ) \exp \bigg({ - \frac{ d^2_0 (z_i, y_i)}{Q \theta} + C\theta^{-1/2} d}\bigg). 
\end{equation}
Let $\beta \in (0,1)$ be a constant whose value we will determine later.
By combining (\ref{eq_a1a2_contr_asspt}), (\ref{eq_a1a2_bound}), (\ref{eq_ai_bound}) we obtain
\begin{multline} \label{eq_Gaussian_Z_3ds}
\alpha^{1+\beta} \theta^{n (1+\beta)}  Z^{1+\beta}  \exp \bigg({ -  \frac{(1+\beta) d^2}{8+\eps/2} }\bigg)
\leq 
\big( \theta^n a_1 a_2 \exp (2\NN^*_0 (x,1))  \big)^{1+\beta} \\
\leq Z^{2\beta} \exp \bigg({ - \frac{d^2_\theta (v_1, v_2)}{8 (1-\theta)} - \frac{\beta}{\theta Q} \big( d^2_0 (z_1, y_1)+ d^2_0 (z_2, y_2) \big) +  C \theta^{-1/2} d  + C }\bigg).
\end{multline}
Using Corollary~\ref{Cor_HK_monotone_Hn}, we find that
\begin{align*}
d=  d_0(y_1, y_2) 
&\leq d_0 (z_1, y_1) + d^{g_0}_{W_1}  (\delta_{z_1}, \nu_{v_1,\theta;0}) + d^{g_0}_{W_1}  ( \nu_{v_1,\theta;0},  \nu_{v_2,\theta;0}) +  d^{g_0}_{W_1}  ( \nu_{v_2,\theta;0}, \delta_{z_2})   + d_0 (z_2, y_2) \\
&\leq d_0 (z_1, y_1) +  
d_0 (z_2, y_2) + d_\theta (v_1, v_2)  + \sqrt{ {\Var}_0 (\delta_{z_1}, \nu_{v_1, \theta; 0}) }+ \sqrt{ {\Var}_0 ( \nu_{v_1, \theta; 0}, \delta_{z_2}) } \\
&\leq d_0 (z_1, y_1) +  d_0 (z_2, y_2) + d_\theta (v_1, v_2) + 2\sqrt{H_n \theta} .
\end{align*}
So by inequality between arithmetic and quadratic mean
\begin{multline*}
 d^2 \leq (1+\beta) d_\theta ^2(v_1, v_2) + \frac{1+\beta}{\beta} \big( d_0 (z_1, y_1) +  d_0 (z_2, y_2) + 2\sqrt{H_n \theta}  \big)^2 \\
 \leq (1+\beta) d_\theta ^2(v_1, v_2) + 3 \frac{1+\beta}{\beta} \big( d^2_0 (z_1, y_1) +  d^2_0 (z_2, y_2) + 4 H_n \theta  \big),
\end{multline*}
which implies
\begin{equation*}
 d^2 = -\beta d^2 + (1+\beta )d^2
\leq  -\beta d^2 + (1+\beta)^2 d_\theta ^2(v_1, v_2) + 3 \frac{(1+\beta)^2}{\beta} \big( d^2_0 (z_1, y_1) +  d^2_0 (z_2, y_2) + 4 H_n \theta  \big). 
\end{equation*}
Combining this with (\ref{eq_Gaussian_Z_3ds}) gives
\begin{multline} \label{eq_bound_exp_Z_1_beta}
 \alpha^{1+\beta} \theta^{n (1+\beta)} Z^{1- \beta}
\leq \exp \bigg({ - \frac{(1+\beta)\beta}{8+\eps/2} d^2 + C  \theta^{-1/2} d + C
+ \Big( \frac{(1+\beta)^3}{8+\eps/2} - \frac1{8(1-\theta)} \Big) d^2_\theta (v_1, v_2) }\\
{+ \Big( 3\frac{(1+\beta)^3}{\beta (8+\eps/2)} - \frac{\beta}{\theta Q} \Big) \big( d^2_0 (z_1, y_1)+ d^2(z_2, y_2) \big) + 12 H_n \theta  \frac{(1+\beta)^3}{\beta (8+\eps/2)}   }\bigg) .
\end{multline}
Now choose $\beta \leq \ov\beta (\eps)$ and then $\theta \leq \ov\theta (\eps, Q, \beta)$ such that
\[ \frac{(1+\beta)^3}{8+\eps/2} \leq \frac{1}{8},
\qquad 3\frac{(1+\beta)^3}{\beta (8+\eps/2)} \leq \frac{\beta}{\theta Q}. \]
Then (\ref{eq_bound_exp_Z_1_beta}) implies that
\[  \alpha^{1+\beta} \theta^{n (1+\beta)} Z^{1- \beta}
\leq \exp \bigg({ - \frac{(1+\beta)\beta}{8+\eps/2} d^2 + C  \theta^{-1/2} d + C + 12 H_n \theta  \frac{(1+\beta)^3}{\beta (8+\eps/2)}  }\bigg) \leq C(\eps, \beta, \theta). \]
Therefore, we obtain a contradiction by choosing $Z \geq \underline{Z}(\eps, \beta, \theta)$.
\end{proof}
\bigskip

\subsection{Proof of Theorem~\ref{Thm_HK_gradient_bound}}

\begin{proof}[Proof of Theorem~\ref{Thm_HK_gradient_bound}.]
After parabolic rescaling and application of a time-shift, we may assume that $s = 0$ and $t = 2$.
In the following, all generic constants may depend on $n, R_{\min}$, without further specification.

Let 
\[ u := K(\cdot,1;y,0). \]
Write $d\nu := d\nu_{1} = K(x,2; \cdot, 1) dg_{1}$.
Then
\begin{align} 
K(x,2; y, 0) &= \int_M u \, d\nu,  \notag \\ 
|\nabla_x K|(x,2;y,0) &\leq \int_M \frac{|\nabla_x K|(x,2;\cdot, 1)}{K(x,2; \cdot, 1)} u \, d\nu.  \label{eq_nabK_int_u}
\end{align}
Let $(z,1)$ be an $H_n$-center of $(x,2)$.
Then by Corollary~\ref{Cor_NN_variation} we have
\[ - \NN^*_0 (z, 1) \leq - \NN^*_0 (x,2) + C. \]
Thus by Theorem~\ref{Thm_NN_dependence}
\[  - \NN^*_0 (\cdot , 1) \leq - \NN^*_0 (x,2) + C \big( d_1 (z, \cdot) + 1 \big). \]
Therefore, by Theorem~\ref{Thm_HK_Linfty_bound} 
\[ u  \leq C \exp ( - \NN^*_0 (\cdot, 1) )
\leq  C \exp ( - \NN^*_0 (x, 2) ) \exp \big( C d_1 (z, \cdot) \big). \]
So using Theorem~\ref{Thm_Gaussian_integral_bound}, we obtain
\begin{align}
 \int_M u^2 d\nu 
&\leq C \exp ( -2 \NN^*_0 (x, 2) ) \int_M  \exp \big( C d_1 (z, \cdot) \big) d\nu \notag  \\
&\leq C \exp ( -2 \NN^*_0 (x, 2) ) \bigg( \int_0^\infty  e^{Cr} \int_{M \setminus B(z,1,r)}   d\nu dr + 1 \bigg) \notag \displaybreak[1] \\
&\leq C \exp ( -2 \NN^*_0 (x, 2) ) \bigg( \int_0^\infty  e^{Cr} \exp \bigg({- \frac{ (r - \sqrt{2H_n} )_+^2}{8}  }\bigg) dr +1 \bigg) \notag \\
&\leq C \exp ( -2 \NN^*_0 (x, 2) ).  \label{eq_int_u_squared}
\end{align}

By Theorem~\ref{Thm_HK_Linfty_bound} we can fix a constant $C_0 ( R_{\min}) < \infty$ such that
\[ \int_M u \, d\nu = K(x,2;y,0) \leq \frac12 \frac{C_0}{2^{n/2}} \exp ( -\NN^*_0 (x,2) ). \]
Set
\[ a :=\bigg( K(x,2;y,0) \bigg/ \frac{C_0}{2^{n/2}} \exp (-\NN^*_0 (x,2)) \bigg)^2 \leq \frac14 \]
and choose $h \geq 0$ such that we can find a subset $X_h \subset M$ with
\[  \bigg\{ \frac{|\nabla_x K|(x,2;\cdot, 1)}{K(x,2; \cdot, 1)} > h \bigg\} \subset X_h \subset  \bigg\{ \frac{|\nabla_x K|(x,2;\cdot, 1)}{K(x,2; \cdot, 1)} \geq h \bigg\} \]
and
\[ \nu (X_h) = a. \]
Then by Proposition~\ref{Prop_nab_K_bounds}
\[ a h \leq \int_{ X_{h} }  \frac{|\nabla_x K|(x,2;\cdot, 1)}{K(x,2; \cdot, 1)} d\nu
\leq C\nu(X_{h}) (-\log \nu(X_{h}))^{1/2} 
=  C a (-\log a)^{1/2} , \]
and therefore
\begin{equation} \label{eq_h_leq_log_a}
  h \leq C (-\log a)^{1/2}. 
\end{equation}
Using (\ref{eq_nabK_int_u}), (\ref{eq_int_u_squared}), Proposition~\ref{Prop_nab_K_bounds}, (\ref{eq_h_leq_log_a}), we finally obtain
\begin{align*}
|\nabla_x K|(x,2;y,0) 
&\leq \int_M \frac{|\nabla_x K|(x,2;\cdot, 1)}{K(x,2; \cdot, 1)} u \, d\nu \\
&= \int_{X_h} \frac{|\nabla_x K|(x,2;\cdot, 1)}{K(x,2; \cdot, 1)} u \, d\nu + \int_{M \setminus X_h} \frac{|\nabla_x K|(x,2;\cdot, 1)}{K(x,2; \cdot, 1)} u\, d\nu \\
&\leq \bigg( \int_{X_h} \bigg( \frac{|\nabla_x K|(x,2;\cdot, 1)}{K(x,2; \cdot, 1)} \bigg)^{2} d\nu \bigg)^{1/2} \bigg( \int_{X_h} u^2 d\nu \bigg)^{1/2} + h \int_M  u \, d\nu \\
&\leq C\exp(-\NN^*_0 (x,2)) \big( {- a \log a } \big)^{1/2}  + C(-\log a)^{1/2} K(x,2;y,0) \\
&\leq C (- \log a)^{1/2} K(x,2;y,0)  .
\end{align*}
This implies the desired bound.
\end{proof}
\bigskip

\section{Upper volume bounds on distance balls} \label{sec_upper_vol}
\subsection{Statement of the result}
In this section we establish an upper bound on any distance ball, depending only on the pointed Nash entropy at the center.
In \cite{Zhang-noninflating, ChenWang-2013} a similar bound was shown under an additional upper scalar curvature bound and without the pointed Nash entropy term.

Let $(M, (g_t)_{t \in I})$ be a Ricci flow on a compact manifold.

\begin{Theorem} \label{Thm_upper_volume_bound}
If $[t -r^2, t] \subset I$ and $R \geq R_{\min}$ on $M \times [t-r^2, t]$, then for any $1 \leq A < \infty$
\[ |B(x,t,Ar)|_t \leq C(R_{\min} r^2) \exp ( \NN_{x,t}(r^2))  \exp ( C_0  A^2 ) r^n. \]
Here $C_0$ denotes constant a dimensional constant.
\end{Theorem}

We remark that the theorem is also new in the case $A=1$.
Furthermore, note that $ \exp ( \NN_{x,t}(r^2)) \leq 1$, so the pointed Nash entropy term is a good term.

Due to Theorem~\ref{Thm_lower_volume_H_center} this bound is optimal near $H_n$-centers, up to a multiplicative constant.
Note also that Theorem~\ref{Thm_upper_volume_bound} can be used to deduce a lower bound on the pointed Nash entropy.
More specifically, a lower bound on the volume of the form $|B(x,t,r)|_t \geq e^{-Y}$ implies a lower bound of the form $\NN_{x,t} (r^2) \geq - Y - C$.
In \cite{Bamler_HK_RF_partial_regularity} we will moreover see that this volume bound also implies that $(x,t)$ is close to an $H_n$-center.
Therefore, we obtain some sort of reverse statement of Theorem~\ref{Thm_lower_volume_H_center}; in other words, non-collapsedness of a distance ball is equivalent to a lower entropy bound plus proximity to an $H_n$-center.

\subsection{Proof}

\begin{proof}
Part of this proof is similar to \cite[Theorem 1.1]{Zhang-noninflating}.
After parabolic rescaling and application of a time-shift, we may assume that $t=0$ and $r =2$.
Since by Proposition~\ref{Prop_NN_basic_properties} we have $\NN_{x,t} (1) \leq \NN_{x,t} (4) + C(R_{\min})$, it suffices to prove in the following that for some dimensional constant $C < \infty$
\begin{equation} \label{eq_easier_upper_vol_bound}
 |B(x,t,A)|_0 \leq C \exp ( \NN_{x,t}(1))  \exp ( C A^2 ). 
\end{equation}
Due to Lemma~\ref{Lem_lower_scal}, we have $R \geq - \frac{n}2$ on $M \times [-1,0]$.

By Theorem~\ref{Thm_NN_dependence} we have
\begin{equation} \label{eq_NN_bound_x_0_upper_vol}
 - \NN_{\cdot, 0} (1) \leq - \NN_{x,0} (1) + C A \qquad \text{on} \quad B(x,0,A). 
\end{equation}

Next, recall that we have
\[ \NN_{x,0}(1) =  - \int_M \big( \log K(x,0;\cdot, -1) \big) K(x,0; \cdot, -1) dg_{-1} - \frac{n}2 - \frac{n}2 \log(4\pi). \]
We can therefore find a point $y \in M$ with
\[ \log K(x,0;y, -1) \geq - \NN_{x,0}(1) - \frac{n}2 - \frac{n}2 \log(4\pi), \]
which implies that
\begin{equation} \label{eq_lower_HK_bound}
 K(x,0; y, -1) \geq c \exp ( - \NN_{x,0}(1)  ). 
\end{equation}

Set
\[ u := K(\cdot, 0; y;-1). \]
By Theorem~\ref{Thm_HK_gradient_bound} there is a constant $C_0  < \infty$ such that
\[ \frac{|\nabla u|}{u} \leq C \sqrt{ \log \bigg( \frac{C_0 \exp ( - \NN_{\cdot,0}(1))}{u} \bigg)}. \]
So by (\ref{eq_NN_bound_x_0_upper_vol}) there is a constant $C_1 < \infty$ such that on $B(x,0,A)$
\begin{equation} \label{eq_nabu_u_C_sqrt_log}
 \frac{|\nabla u|}{u} \leq C \sqrt{ \log \bigg( \frac{C_1 \exp ( - \NN_{x,0}(1) + C_1 A)}{u} \bigg)}. 
\end{equation}
So if we set
\[ v := \sqrt{ \log \bigg( \frac{C_1 \exp ( - \NN_{x,0}(1) + C_1 A)}{u} \bigg)}, \]
then by (\ref{eq_nabu_u_C_sqrt_log}), (\ref{eq_lower_HK_bound}) we have
\[ |\nabla v| \leq C \quad \text{on} \quad B(x,0,A), \qquad 
v(x) \leq C\sqrt{A}. \]
This implies that $v \leq C\sqrt{A} + C A \leq CA$ on $B(x,0,A)$, so
\begin{equation} \label{eq_u_lower_bound_ball}
 u \geq c  \exp (- C A^2) \exp ( - \NN_{x,0}(1)) \qquad \text{on} \quad B(x,0,A).  
\end{equation}

Lastly, notice that since
\[ \frac{d}{dt} \int_M K(\cdot,t;y, -1) dg_t = - \int_M R \, K(\cdot,t;y, -1) dg_t \leq C \int_M K(\cdot,t;y, -1) dg_t, \]
we have using (\ref{eq_u_lower_bound_ball})
\begin{equation*} 
c \exp (- C A^2) \exp ( - \NN_{x,0}(1)) |B(x,0,A)|_0 \leq \int_M u \, dg_0 \leq C . 
\end{equation*}
This finishes the proof of (\ref{eq_easier_upper_vol_bound}).
\end{proof}

\section{\texorpdfstring{$P^*$}{P*}-Parabolic neighborhoods} \label{sec_variance_parab_nbhd}
\subsection{Statement of the results}
In this section we introduce a new notion of parabolic neighborhoods, called \emph{$P^*$-parabolic neighborhoods,} which is inspired by the monotonicity of the $W_1$-Wasserstein distance (see Lemma~\ref{Lem_monotonicity_W1}).
We will then show that, under certain curvature bounds, $P^*$-parabolic neighborhoods are comparable to conventional parabolic neighborhoods.
In this process, we will also obtain estimates for the location of $H_n$-centers under curvature bounds.
Lastly, we will analyze the geometry of time-slices of $P^*$-parabolic neighborhoods and establish a covering theorem.

Before defining $P^*$-parabolic neighborhoods, let us first recall the definition of a conventional parabolic neighborhood.
In the following, we denote by $(M,(g_t)_{t \in I})$ a super-Ricci flow on a compact manifold.
Let $(x_0, t_0) \in M \times I$ and $A, T^-, T^+ \geq 0$.
Then the conventional parabolic neighborhood is defined as
\begin{equation} \label{eq_def_par_nbhd}
 P(x_0, t_0; A, - T^-, T^+) := B(x_0, t_0, A) \times \big( [t_0 - T^-, t_0 + T^+] \cap I \big), 
\end{equation}
where we may omit $- T^-$ or $T^+$ if it is zero.
Note that we have introduced the separator ``$;$'' in order to avoid confusion if we consider parabolic neighborhoods on a Ricci flow spacetime. 

In the spacetime picture $M \times I$, the definition (\ref{eq_def_par_nbhd}) relies on the concept of worldlines, i.e. if a point $(x,t)$ is contained in $P(x_0, t_0; A, - T^-, T^+)$, then so are all other points of the form $(x, t')$, for $t'$ within the above time-interval.
One important observation from the previous sections is, however, that the concept of worldlines plays a secondary role in the analysis of the flow and that it may instead be better to measure the relation between points in different time-slices using the $W_1$-distance of conjugate heat kernel measures.
The following definition exploits this idea.

\begin{Definition}[$P^*$-Parabolic Neighborhoods]
Suppose that $(x_0, t_0) \in M \times I$, $A, T^+, T^- \geq 0$ and $t_0 - T^- \in I$.
The {\bf $P^*$-parabolic neighborhood} $P^* (x_0, t_0; A, -T^-, T^+) \subset M \times I$ is defined as the set of points $(x,t) \in M \times I$ with $t \in [t_0-T^-, t_0+T^+]$ and 
\[ d^{g_{t_0 - T^-}}_{W_1} (\nu_{x_0, t_0; t_0 - T^-}, \nu_{x,t;  t_0 - T^-}) < A. \]
\end{Definition}

In most situations, in which we only work up to a multiplicative constant, it suffices to consider a simplified notion:

\begin{Definition}[(Forward/Backward) $P^*$-Parabolic Ball]
If $(x_0, t_0) \in M \times I$ and $r > 0$ such that $t_0 -r^2 \in I$, then the {\bf $P^*$-parabolic $r$-ball} is defined as follows
\[ P^* (x_0, t_0; r) := P^* (x_0, t_0;  r, -r^2, r^2 ). \]
Similarly, we define the {\bf forward ($+$) and backward ($-$) $P^*$-parabolic $r$-balls} by
\[ P^{*+} (x_0, t_0; r) := P^* (x_0, t_0;  r, 0, r^2 ),
\qquad P^{*-} (x_0, t_0; r) := P^* (x_0, t_0; r, -r^2, 0 )
. \]
\end{Definition}

The following proposition shows that $P^*$-parabolic neighborhoods satisfy similar containment relationships as standard parabolic balls.

\begin{Proposition} \label{Prop_basic_parab_nbhd}
The following holds for any $(x_1, t_1), (x_2, t_2) \in M \times I$ as long as the corresponding $P^*$-parabolic neighborhoods or balls are defined:
\begin{enumerate}[label=(\alph*)]
\item \label{Prop_basic_parab_nbhd_a} For any $A \geq 0$ we have
\[ P^* (x_1, t_1; A,0,0) = B(x_1, t_1, A) \times \{ t_1 \}. \]
\item \label{Prop_basic_parab_nbhd_b} If $0 \leq A_1 \leq A_2$, $0 \leq T^\pm_1 \leq T^\pm_2$, then 
\[ P^{*} (x_1, t_1; A_1, -T_1^-, T_1^+) \subset P^{*} (x_1, t_1; A_2, -T_2^-, T_2^+). \]
\item \label{Prop_basic_parab_nbhd_bb} If $A, T^\pm \geq 0$, and $(x_1, t_1) \in P^*(x_2, t_2; A, -T^-, T^+)$, then
\[ (x_2, t_2) \in P^*(x_1, t_1; A, - (T^- + T^+), T^-) \]
and
\[  P^*(x_2, t_2; A, -T^-, T^+) \subset P^*(x_1, t_1;2 A, - (T^- + T^+), T^-+T^+). \]
Likewise, if $r > 0$ and $(x_1, t_1) \in P^*(x_2, t_2; r)$, then 
\[ (x_2, t_2) \in P^*(x_1, t_1; \sqrt{2} r) \quad \text{and} \quad P^*(x_2, t_2; r) \subset P^*(x_1, t_1; 2r). \]
\item \label{Prop_basic_parab_nbhd_c} If $A_1, A_2, T_1^\pm, T_2^\pm \geq 0$ and $(x_1, t_1) \in P^* (x_2, t_2; A_2, -T_2^-, T_2^+)$, then
\[ P^*(x_1, t_1; A_1, -T_1^-, T_1^+) \subset P^*(x_2, t_2; A_1  + A_2, - (T_1^- + T_2^-), T_1^+ + T_2^+). \]
Likewise, if $r_1, r_2 > 0$ and $(x_1, t_1) \in P^*(x_2,t_2; r_2)$, then
\[ P^*(x_1,t_1; r_1) \subset P^*(x_2, t_2; r_1+r_2). \]
The same containment relationship also holds for the forward or backward parabolic balls, if $t_1 \geq t_2$ or $t_1 \leq t_2$, respectively.
\item \label{Prop_basic_parab_nbhd_d} If $r_1, r_2 > 0$ and $P^* (x_1, t_1; r_1) \cap P^* (x_2, t_2; r_2) \neq \emptyset$, then $P^* (x_1, t_1; r_1) \subset P^* (x_2, t_2; 2r_1+r_2)$.
Again, the same containment relationship also holds for the forward or backward parabolic balls, if $t_1 \geq t_2$ or $t_1 \leq t_2$, respectively.
\end{enumerate}
\end{Proposition}

Next, we compare conventional and $P^*$-parabolic neighborhoods, assuming a local two-sided bound on the Ricci curvature.
The following proposition will be the basis of this discussion.

\begin{Proposition} \label{Prop_Var_location_Ric_bound}
Assume that $\alpha > 0$ and $K < \infty$ and consider a point $(x_0, t_0) \in M \times I$ and a scale $r > 0$ with $[t_0 - r^2, t_0] \subset I$.
Suppose that $|{\Ric}| \leq  K r^{-2}$ on $P(x_0, t_0; \alpha r, -r^2)$.
Then
\[ d_{W_1}^{g_{t_0 - r^2}} ( \nu_{x_0, t_0; t_0 - r^2}, \delta_{x_0} ) \leq C (\alpha, K) r. \]
\end{Proposition}

As a corollary, we obtain the following containment relationships between  conventional and $P^*$-parabolic neighborhoods.

\begin{Corollary} \label{Cor_P_sub_Pvar}
For any $0 < \alpha \leq A < \infty$, $K,  T^\pm \geq 0$ the following holds if $A' \geq \underline{A}' (\alpha, A, K, T^\pm)$.
Consider a point $(x_0, t_0) \in M \times I$ and a scale $r > 0$ with $[t_0 - T^- r^2, t_0] \subset I$.
Then:
\begin{enumerate}[label=(\alph*)]
\item \label{Cor_P_sub_Pvar_a} If $|{\Ric}| \leq K r^{-2}$ on $P(x_0, t_0; A r, - T^- r^2, T^+ r^2)$ or on $P^* (x_0, t_0; A' r, - T^- r^2, T^+ r^2)$, then
\[ P(x_0, t_0; A r, -T^- r^2, T^+ r^2) \subset P^* (x_0, t_0; A' r, - T^- r^2, T^+ r^2). \]
\item \label{Cor_P_sub_Pvar_b} If $|{\Ric}| \leq K r^{-2}$ on $P(x_0, t_0; A' r, - T^- r^2, T^+ r^2)$, then
\[  P^* (x_0, t_0; A r, - T^- r^2, T^+ r^2) \subset P(x_0, t_0; A' r, -T^- r^2, T^+ r^2).  \]
\end{enumerate}
\end{Corollary}

Next, we discuss the geometry of time-slices of $P^*$-parabolic neighborhoods, which are of the form
\begin{equation} \label{eq_time_slice_of_P_star}
 S_t := P^* (x_0, t_0; A, -T^-, T^+) \cap \big( M \times \{ t \} \big). 
\end{equation}
Note that any such time-slice may have complicated geometry; in particular its diameter may not be bounded in terms of $A, T^\pm$.
However, the following proposition implies that its volume can be bounded in terms of these quantities, the pointed Nash entropy and a global lower scalar curvature bound.

\begin{Theorem} \label{Thm_parab_time_slice_vol}
Let $\alpha > 0$, $A , T^\pm \geq 0$ and consider a point $(x_0, t_0) \in M \times I$ and a scale $r > 0$ with  $[t_0 - (T^- +\alpha ) r^2, t_0] \subset I$.
Then for any $t \in [t_0 - T^- r^2, t_0 + T^+ r^2]$ the time-slice $S_t$ from (\ref{eq_time_slice_of_P_star}) satisfies the following volume bound:
\begin{equation} \label{eq_thm_S_bound}
 | S_t |_t \leq C(  A,  T^-, T^+, \alpha) \exp ( \NN_{x_0, t_0} (T^- r^2) ) r^n. 
\end{equation}
Moreover, for any $A' < \infty$ we have the following bound on the $A' r$-neighborhood of $S_t$
\begin{equation} \label{eq_thm_BS_bound}
 | B(S_t, t, A'r) |_t \leq C( A, A', T^-, T^+, \alpha) \exp ( \NN_{x_0, t_0} (T^- r^2) ) r^n. 
\end{equation}
\end{Theorem}

An important application of Theorem~\ref{Thm_parab_time_slice_vol} is the following covering result, which we will state for $P^*$-parabolic \emph{balls} for convenience.
Using the same techniques, one can also obtain a covering result for general $P^*$-parabolic neighborhoods.

\begin{Theorem} \label{Thm_covering}
Let $\la_0 > 0$, $A , T^\pm \geq 0$ and consider a point $(x_0, t_0) \in M \times I$ and a scale $r > 0$ with with $[t_0 - (T^- +\la_0^2 ) r^2, t_0] \subset I$.

Then for any subset $X \subset P^* (x_0, t_0; A r, -T^- r^2, T^+ r^2)$ and $\lambda \in (0,\la_0]$ we can find points $(y_1, s_1), \lb \ldots, \lb (y_N, s_N) \in X$ with the property that
\begin{equation} \label{eq_thm_covering}
 X \subset \bigcup_{i=1}^N P^*(y_i, s_i; \lambda r), \qquad N \leq C(  A, T^-, T^+, \la_0 ) \la^{-n-2}. 
\end{equation}
\end{Theorem}

Note that the upper bound on $N$ is independent of the pointed Nash entropy at $(x_0, t_0)$.
We also note that it is unclear whether a similar covering theorem holds for conventional parabolic neighborhoods in the absence of curvature bounds.

\subsection{Proofs}

\begin{proof}[Proof of Proposition~\ref{Prop_basic_parab_nbhd}.]
Assertion~\ref{Prop_basic_parab_nbhd_a} is clear.
For Assertion~\ref{Prop_basic_parab_nbhd_b} observe that for any $(x,t) \in P^* (x_0, t_0; A_1, -T_1^-, T_1^+)$ we have due to Lemma~\ref{Lem_monotonicity_W1}
\begin{equation*}
 d_{W_1}^{g_{t_0 - T_2^-}} (\nu_{x_0, t_0; t_0 - T_2^-}, \nu_{x,t;t_0 - T_2^-} ) 
 \leq d_{W_1}^{g_{t_0 - T_1^-}} (\nu_{x_0, t_0; t_0 - T_1^-}, \nu_{x,t;t_0 - T_1^-} ) 
 < A_1 \leq A_2.
\end{equation*}

For Assertion~\ref{Prop_basic_parab_nbhd_bb}, observe that since $t_1 - (T^- + T^+) \leq t_2 - T^-$, we have by Lemma~\ref{Lem_monotonicity_W1}
\[ d_{W_1}^{g_{t_1 - (T^-+T^+)}} ( \nu_{x_1,t_1; t_1 - (T^-+T^+)}, \nu_{x_2,t_2; t_1 - (T^-+T^+)} )
\leq d_{W_1}^{g_{t_2 - T^-}} ( \nu_{x_1,t_1; t_2 - T^-}, \nu_{x_2,t_2; t_2 - T^-} )
< A. \]
Moreover, if $(x,t) \in P^*(x_2, t_2; A, -T^-, T^+)$, then
\begin{multline*}
 d_{W_1}^{g_{t_1 - (T^- + T^+)}} ( \nu_{x,t; t_1 - (T^- + T^+)}, \nu_{x_1,t_1; t_1 - (T^- + T^+)} )
\leq d_{W_1}^{g_{t_2 - T^-}} ( \nu_{x,t; t_2 - T^-}, \nu_{x_1,t_1; t_2 - T^- } ) \\
\leq d_{W_1}^{g_{t_2 - T^-}} ( \nu_{x,t; t_2 - T^-}, \nu_{x_2,t_2; t_2 - T^- } )
+ d_{W_1}^{g_{t_2 - T^-}} ( \nu_{x_2,t_2; t_2 - T^-}, \nu_{x_1,t_1; t_2 - T^- } )
< 2A. 
\end{multline*}

To see Assertion~\ref{Prop_basic_parab_nbhd_c}, let $(x,t) \in P^*(x_1, t_1; A_1, - T_1^-, T_1^+)$.
Then by Lemma~\ref{Lem_monotonicity_W1}
\begin{multline*}
 d_{W_1}^{g_{t_2 - (T_1^- + T_2^-)}} ( \nu_{x,t; t_2 - (T_1^- + T_2^-)},  \nu_{x_2,t_2; t_2 - (T_1^- + T_2^-)} ) \\
\leq d_{W_1}^{g_{t_2 - (T_1^- + T_2^-)}} ( \nu_{x,t; t_2 - (T_1^- + T_2^-)},  \nu_{x_1,t_1; t_2 - (T_1^- + T_2^-)} ) + d_{W_1}^{g_{t_2 - (T_1^- + T_2^-)}} ( \nu_{x_1,t_1; t_2 - (T_1^- + T_2^-)},  \nu_{x_2,t_2; t_2 - (T_1^- + T_2^-)} ) \\
\leq d_{W_1}^{g_{t_1 - T_1^- }} ( \nu_{x,t; t_1 - T_1^- },  \nu_{x_1,t_1; t_1 - T_1^- } ) + d_{W_1}^{g_{t_2 -  T_2^-}} ( \nu_{x_1,t_1; t_2 -  T_2^-},  \nu_{x_2,t_2; t_2 - T_2^-} ) 
< A_1 + A_2.
\end{multline*}

To see Assertion~\ref{Prop_basic_parab_nbhd_d}, choose $(x,t) \in P^* (x_1, t_1; r_1) \cap P^* (x_2, t_2; r_2)$.
Then by Assertions~\ref{Prop_basic_parab_nbhd_bb}, \ref{Prop_basic_parab_nbhd_c}
\[ P^*(x_1, t_1; r_1) \subset P^*(x,t; 2r_1) \subset P^*(x_2, t_2; 2r_1 + r_2). \qedhere \]
\end{proof}
\bigskip

The following lemma will be needed for the proof of Proposition~\ref{Prop_Var_location_Ric_bound}.

\begin{Lemma} \label{Lem_subsolution}
In the setting of Proposition~\ref{Prop_Var_location_Ric_bound}, we can find a compactly supported function $u \in C^0_c (P(x_0, t_0; \alpha r, -r^2))$ such that the following is true:
\begin{enumerate}[label=(\alph*)]
\item $0 \leq u \leq 1$.
\item $\square u \leq 0$ in the barrier and viscosity sense.
\item There are constants $c_0 (\alpha, K) \geq c_1 (\alpha, K) > 0$ such that
for all $t \in [t_0 - r^2, t_0]$ we have $\supp u(\cdot, t) \subset B(x_0, t, c_0 r) \subset P(x_0, t_0; \alpha r, -r^2)$ and $u \geq c_1$ on $B(x_0, t, c_1 r)$.
\end{enumerate}
\end{Lemma}

\begin{proof}
After parabolic rescaling and application of a time-shift, we may assume that $r =1$ and $t_0 - r^2 = 0$.
Let $A <\infty$ be a constant whose value we will determine later.
Choose a bump function $\phi : [0, \infty) \to [0,1]$ such that $\phi \equiv 1$ on $[0, \frac13 \alpha e^{-K}]$ and $\phi \equiv 0$ on $[\frac23 \alpha e^{-K}, \infty)$ and such that on $[0, \infty)$ we have
\[ \phi' \leq 0, \qquad A \phi' + \phi'' \geq - C_0 (A, \alpha, K) \phi. \]
Define
\[ u(x,t) := e^{-C_0 (A,\alpha,K) t} \phi ( d_t (x_0, x)). \]
By Laplace comparison we have for any $t \in [0,1]$ in the barrier sense
\begin{align*}
 \square u =  \partial_t u - \triangle u
&\leq e^{-C_0 (A,\alpha, K) t} \big( \phi' \, \partial_t d_t (x_0, \cdot) - C_0(A,\alpha, K) \phi  - \phi' \triangle_{g_t} d_t (x_0, \cdot ) - \phi'' \big) \\
&\leq e^{-C_0 (A,\alpha, K) t} \bigg( -K \phi' d_t(x_0, \cdot) -  \phi' \, \frac{(n-1) \sqrt{\frac{K}{n-1}}}{\tanh ( \sqrt{\frac{K}{n-1}} d_t (x_0, \cdot) )} - \phi'' - C_0 (A,\alpha, K) \phi \bigg) \\
&\leq e^{-C_0 (A,\alpha, K) t} \bigg( - \bigg( K \alpha e^{-K} +\frac{(n-1) \sqrt{\frac{K}{n-1}}}{\tanh (\frac13 \alpha e^{-K} \sqrt{\frac{K}{n-1}} )} \bigg)  \phi'  - \phi'' - C_0 (A,\alpha, K) \phi \bigg).
\end{align*}
So if $A = A(\alpha, K)$ is chosen to be larger than the term in parentheses, then $\square u \leq 0$ in the barrier and hence also in the viscosity sense.
\end{proof}
\bigskip

\begin{proof}[Proof of Proposition~\ref{Prop_Var_location_Ric_bound}.]
After parabolic rescaling and application of a time-shift, we may assume that $r =1$ and $t_0 - r^2 = 0$.
Consider the function $u \in C^0_c (P(x_0, 1; \alpha , -1))$ and the constants $c_0(\alpha, K) \geq c_1 (\alpha, K)> 0$ from Lemma~\ref{Lem_subsolution}.
We obtain that
\[ \nu_{x_0, 1; 0} \big( B(x_0,0, c_0) \big) \geq  \int_M u \, d\nu_{x_0, 1}(0) \geq  u(x_0, 1)  \geq c_1 . \]
Combining this with Proposition~\ref{Prop_nu_ball_bound} implies that for any $H_n$-center $(z,0)$ of $(x_0,1)$
\[ d_0 (z, x_0) \leq C(\alpha, K) \]
and therefore,
\[ d^{g_0}_{W_1}( \nu_{x_0, 1;0} , \delta_{x_0} ) 
\leq d^{g_0}_{W_1} ( \nu_{x_0, 1;0} , \delta_{z} ) 
+ d^{g_0}_{W_1} ( \delta_z, \delta_{x_0} ) 
\leq \sqrt{H_n} + C(\alpha, K) 
\leq C(\alpha, K) , \]
as desired.
\end{proof}
\bigskip

\begin{proof}[Proof of Corollary~\ref{Cor_P_sub_Pvar}.]
Assertion~\ref{Cor_P_sub_Pvar_a}, assuming $|{\Ric}| \leq K r^{-2}$ on $P(x_0, t_0; A r, - T^- r^2, T^+ r^2)$, follows directly from Proposition~\ref{Prop_Var_location_Ric_bound} and a standard distance distortion estimate.
Next assume that $|{\Ric}| \leq K r^{-2}$ on $P^* (x_0, t_0; A' r, - T^- r^2, T^+ r^2)$.
Choose $r^* \in (0,r]$ maximal with the property that $|{\Ric}| \leq K r^{* -2}$ on $P(x_0, t_0; A r^*, - T^- r^{*2}, T^+ r^{*2})$.
If $r^* = r$, then we are done by the previous conclusion.
If $r^* < r$, then by the previous conclusion
\[ P(x_0, t_0; A r^*, - T^- r^{*2}, T^+ r^{*2}) \subset P^*(x_0, t_0; A' r^*, - T^- r^{*2}, T^+ r^{*2}). \]
So $|{\Ric}| \leq r^{-2} < r^{*-2}$ on $P(x_0, t_0; A r^*, - T^- r^{*2}, T^+ r^{*2})$, which contradicts the maximal choice of $r^*$.

Let us now prove Assertion~\ref{Cor_P_sub_Pvar_b}.
By parabolic rescaling, we may assume that $r = 1$.
Next, let us argue that we may assume without loss of generality that $T^- = 0$.
To see this note that by Proposition~\ref{Prop_Var_location_Ric_bound} for any $(x,t) \in P^* (x_0, t_0; A , - T^- , T^+  )$ we have
\begin{multline*}
 d_{W_1}^{g_{t_0 - T^-}} (\nu_{x_0, t_0 - T^-;  t_0 - T^-}, \nu_{x,t; t_0 - T^-} )
\leq d_{W_1}^{g_{t_0 - T^-}}(\delta_{x_0}, \nu_{x_0,t_0; t_0 - T^-} ) 
+ d_{W_1}^{g_{t_0 - T^-}} (\nu_{x_0, t_0 ;t_0 - T^-}, \nu_{x,t; t_0 - T^-} ) \\
\leq C( \alpha, K, T^-)+ A. 
\end{multline*}
Therefore
\[ P^* (x_0, t_0; A , - T^- , T^+  )  \subset P^* (x_0, t_0 - T^-; C( \alpha, K, T^-) + A , 0 , T^- + T^+  ). \]
By standard distance-distortion estimates we also have
\[ P(x_0, t_0; A', -T^-, T^+) \supset P(x_0, t_0 - T^-; c( K, T^-) A', 0, T^- + T^+). \]
So after replacing $t_0$ with $t_0 - T^-$, $T^+$ with $T^- + T^+$ and adjusting $A'$ appropriately, we may indeed assume in the following that $T^-  = 0$ and write $T=T^+$.
After application of a time-shift, we may furthermore assume that $t_0 =0$.
So we have reduce the problem to showing that if $A' \geq \underline{A}' ( A, K,  T)$ and $|{\Ric}| \leq K$ on $P(x_0, 0; A',  T)$, then
\begin{equation} \label{eq_Pvar_P_AT}
 P^* (x_0, 0; A,  T) \subset P(x_0, 0; A', T). 
\end{equation}

Fix $\alpha, K, A, T$.
We will determine a lower bound of the form $A' \geq \underline{A} ( A,K, T)$ in the course of the proof.
Choose $T^* \in [0, T]$ maximal with the property that for all $T' \in [0,T^*)$  we even have
\begin{equation} \label{eq_P_var_P_12A_p}
 P^* (x_0, 0; A,  T') \subset P(x_0, 0; \tfrac12 A',  T'). 
\end{equation}
Assuming $A' > 2A$, we must have $T^* > 0$.

Our goal will be to show that we can improve (\ref{eq_P_var_P_12A_p}).
For this purpose, consider some $(x,t) \in P^* (x_0, 0; A,  T')$.
Then by (\ref{eq_P_var_P_12A_p}) and a basic distance distortion estimate we have for  $A' \geq \underline{A} (K, A, T)$
\[ P(x, t; 1,- t) \subset P(x_0, 0;  A',  T'). \]
Thus by Proposition~\ref{Prop_Var_location_Ric_bound} we have for $A' \geq \underline{A} ( K, A, T)$
\begin{equation*}
d_0 (x_0, x)
= d^{g_0}_{W_1} ( \delta_{x_0}, \delta_x)
\leq
 d^{g_0}_{W_1} ( \delta_{x_0}, \nu_{x,t;0} )
 + d^{g_0}_{W_1}  (\nu_{x,t;0} , \delta_x )
 \leq A+ C( K, T) \leq \tfrac14 A',
\end{equation*}
which implies $(x,t) \in P(x_0, 0; \tfrac14 A',  T')$.
So by continuity and the maximal choice of $T^*$, we must have $T^* = T$, which proves (\ref{eq_Pvar_P_AT}).
\end{proof}
\bigskip

\begin{proof}[Proof of Theorem~\ref{Thm_parab_time_slice_vol}.]
By parabolic rescaling and application of a time-shift we may assume that $r=1$ and $t_0 - T^- = 0$.
Let $(z_0,0)$ be an $H_n$-center of $(x_0, t_0)$.
Fix some time $t \in [t_0 - T^- , t_0 + T^+]$, $x \in S_t$ and choose an $H_n$-center $(z, 0)$ of $(x,t)$.
Then
\begin{multline*}
 d_0 (z_0,z) 
  = d_{W_1}^{g_0} (\delta_{z_0}, \delta_{z})
\leq d_{W_1}^{g_0}  (\delta_{z_0}, \nu_{x_0, t_0; 0}) + d_{W_1}^{g_0} (\nu_{x_0, t_0;0}, \nu_{x,t;0}) + d_{W_1}^{g_0} (\nu_{x,t;0}, \delta_{z}) \\
\leq \sqrt{H_n t_0} +A + \sqrt{H_n t} \leq C( A, T^-, T^+) .
\end{multline*}
So by Theorem~\ref{Thm_Gaussian_integral_bound} there is a constant $D( A, T^-, T^+)<\infty$ such that for $B := B(z_0,0, D)$ we have
\begin{equation} \label{eq_nu_B_12}
 \nu_{x,t;0} (B) \geq \nu_{x,t;0} \big( B(z, 0, \sqrt{2H_n t} ) \big) \geq \frac12. 
\end{equation}

Let now $u \in C^0(M \times [0, t]) \cap C^\infty(M \times (0,t])$ be the solution to the heat equation $\square u = 0$ with initial condition $u(\cdot, 0) = \chi_B$.
By (\ref{eq_nu_B_12}) we have $u \geq \frac12$ on $S_t$ and since, by Lemma~\ref{Lem_lower_scal},
\[ \frac{d}{dt} \int_M u \, dg_t = - \int u R \, dg_t \leq  \frac{n}{2\alpha} \int_M u \, dg_t, \]
we obtain
\begin{equation} \label{eq_12S_B_0}
 \frac12 |S|_t \leq \int_M u(\cdot, t) dg_t \leq e^{ \frac{n}{2\alpha} t} |B|_0. 
\end{equation}

Since $d_{W_1}^{g_0} (\delta_{z_0} , \nu_{x_0, t_0} (0)) \leq \sqrt{H_n t_0}$,
we can use Proposition~\ref{Prop_NN_basic_properties} and Corollary~\ref{Cor_NN_variation} to deduce that $\NN_{z_0,0} (\frac12 \alpha) \leq \NN_{x_0,t_0} (\frac12 \alpha) + C( T^-, \alpha)$.
So, using Theorem~\ref{Thm_upper_volume_bound}, we get
\[  |B|_0 \leq C(  A, T^-, T^+, \alpha) \exp ( \NN_{z_0,0} (\tfrac12 \alpha) ) \leq C(  A, T^-, T^+, \alpha) \exp ( \NN_{x_0,t_0} (\tfrac12 \alpha) ).  \]
Combining this with (\ref{eq_12S_B_0}) implies (\ref{eq_thm_S_bound}).

Lastly, we remark that (\ref{eq_thm_BS_bound}) can be reduced to (\ref{eq_thm_S_bound}), because Proposition~\ref{Prop_basic_parab_nbhd} implies
\begin{equation*} 
 B(S_t, t, A' ) \times \{ t \} \subset P^* (x_0, t_0;A+A', - T^-, T^+). \qedhere
\end{equation*}
\end{proof}
\bigskip

\begin{proof}[Proof of Theorem~\ref{Thm_covering}.]
After parabolic rescaling, we may assume that $r = 1$.
Next, we may assume that $\lambda \in (0, \frac{\la_0}{3}]$ and aim to prove the theorem for $\lambda$ replaced with $3 \lambda$.
To do this, choose a maximal collection of points $(y_1, s_1), \ldots, (y_N, s_N) \in X$ with
\[ P^* (y_i, s_i; \lambda) \cap P^* (y_j, s_j; \lambda) = \emptyset \qquad \text{for all} \quad i \neq j. \]
To see that these points satisfy the covering identity in (\ref{eq_thm_covering}) with $\lambda$ replaced with $3 \lambda$, consider a point $(y', s') \in X$.
By maximality we must have $P^*(y', s'; \lambda) \cap P^*(y_i, s_i; \lambda) \neq \emptyset$ for some $i \in \{ 1, \ldots, N \}$, which implies using Proposition~\ref{Prop_basic_parab_nbhd} that $(y', s') \in P^* (y', s'; \lambda) \subset P^* (y_i, s_i; 3 \lambda)$.

It remains to derive an upper bound on $N$.
For this purpose, let $\beta \in (0, \frac12]$ be some constant whose value we will determine later and observe that there is some time $t^* \in [t_0-T^- - 2\beta \la^2, t_0+T^+ - \beta\la^2]$ and a subset $\mathcal{I} \subset \{ 1, \ldots, N \}$ with the property that 
\begin{equation} \label{eq_lower_II}
|\mathcal{I}| \geq \lfloor c(T^-, T^+, \beta) \la^2 N \rfloor
\end{equation}
 and $t^* \in [s_i -2\beta \la^2, s_i - \beta\la^2]$ for all $i \in \mathcal{I}$.
For each $i \in \mathcal{I}$ let $(z_i, t^*)$ be an $H_n$-center of $(y_i, s_i)$.
If $\beta \leq \ov\beta$, then for all $i,j \in \mathcal{I}$, $i \neq j$
 \begin{multline*}
 d_{t^*} (z_i, z_j) 
 \geq d^{g_{t^*}}_{W_1} (\nu_{y_i, s_i;t^*},\nu_{y_j, s_j;t^*}) 
 - d^{g_{t^*}}_{W_1}(\delta_{z_i}, \nu_{y_i, s_i;t^*}) - d^{g_{t^*}}_{W_1} (\delta_{z_j}, \nu_{y_j, s_j; t^*}) \\
\geq d^{g_{s_i -   \la^2}}_{W_1}  (\nu_{y_i, s_i;s_i -  \la^2},\nu_{y_j, s_j;s_i -   \la^2})   - 2 \sqrt{2H_n \beta \la^2} 
 \geq \la  - 2 \sqrt{2H_n \beta \la^2}
 \geq \tfrac12\la.
\end{multline*}
So the balls $B_i := B(z_i, t^*,  \frac14 \la)$, $i \in \mathcal{I}$, are pairwise disjoint.
By Theorem~\ref{Thm_lower_volume_H_center}, Proposition~\ref{Prop_NN_basic_properties}, Corollary~\ref{Cor_NN_variation} and Lemma~\ref{Lem_lower_scal} we have, if $\beta \leq \ov\beta$ is chosen such that $\sqrt{2H_n} \cdot 2\beta  \leq \tfrac14$,
\begin{equation} \label{eq_Bi_lower_vol}
 |B_i|_{t^*} \geq c(\la_0, \beta) \exp ( \NN_{y_i, s_i} (\lambda^2 ) ) \lambda^n \geq c( A,T^-, T^+, \la_0, \beta) \exp \big(  \NN_{x_0, t_0} \big( (\tfrac{\la_0}{3})^2 \big) \big) \lambda^n. 
\end{equation}
On the other hand, Theorem~\ref{Thm_parab_time_slice_vol} implies that for
\[ S_{t^*} := P^* (x_0 ,t_0; A, T^-, T^+) \cap M \times \{ t^* \} \]
we have
\begin{equation} \label{eq_BPDy0_upper}
 |B( S_{t^*}, t^*, \tfrac14 \la_0)|_{t^*} \leq C( A, T^-, T^+, \la_0)  \exp \big(  \NN_{x_0, t_0} \big( (\tfrac{\la_0}3 )^2 \big) \big). 
\end{equation}
Combining (\ref{eq_lower_II}), (\ref{eq_Bi_lower_vol}), (\ref{eq_BPDy0_upper}) implies that
\[ \lfloor c(T^-, T^+, \beta) \lambda^2 N   \rfloor \leq | \mathcal{I} | \leq C( A,T^-, T^+, \la_0, \beta) \lambda^{-n}, \]
which finishes the proof.
\end{proof}
\bigskip

\section{An \texorpdfstring{$\eps$}{epsilon}-regularity theorem} \label{sec_eps_reg}
\subsection{Statement of the results}
In this section we improve the $\eps$-regularity theorem in \cite[Theorem~1.6]{Hein-Naber-14}.
The advantage of the following theorem is that it is purely local, i.e. it does not require any global entropy bounds.

Let in the following $(M, (g_t)_{t \in I})$ be again a Ricci flow on a compact manifold.
We will use the curvature scale radius.\footnote{Note that we are using a slightly different definition of the scale $\rrm$ than in \cite{Hein-Naber-14}. Due to Perelman's pseudolocality theorem \cite{Perelman1}, both scales are comparable, given a non-collapsing bound.}

\begin{Definition}[Curvature Scale Radius]
For any point $(x,t) \in M \times I$ we define the {\bf curvature scale radius} at $(x,t)$ as follows:
\[ \rrm (x,t) := \sup \big\{ r > 0  \;\; : \;\; |{\Rm}| \leq r^{-2} \;\; \text{on} \;\; P(x,t;r)  \big\}. \]
Note that $P(x,t;r)$ denotes the conventional parabolic neighborhood.
\end{Definition}

Our $\eps$-regularity theorem can now be stated as follows.

\begin{Theorem} \label{Thm_simple_eps_regularity}
There is a constant $\eps (n) > 0$ such that the following holds.
Consider a point $(x,t ) \in M \times I$, $r > 0$ with $[t-r^2,t] \subset I$.
If $\NN_{x,t} (r^2) \geq - \eps$, then $\rrm (x,t) \geq \eps r$.
\end{Theorem}

Using this result, we will also prove the following stronger theorem.

\begin{Theorem} \label{Thm_eps_regularity}
For any $\eps > 0$ there is a $\delta (\eps) > 0$ such that the following holds.
Consider a point $(x,t ) \in M \times I$, $r > 0$ with $[t-r^2,t] \subset I$.
If $\NN_{x,t} (r^2) \geq - \delta$, then 
\[ |{\Rm}| \leq \eps r^{-2} \qquad \text{on} \quad P(x,t; \eps^{-1} r,- (1-\eps) r^2,   \eps^{-1} r^2). \]
Moreover, we have $\NN^*_{t-r^2} \geq - \eps$ on $P(x,t; \eps^{-1} r, - (1-\eps) r^2)$.
\end{Theorem}

The following theorem provides some sort of converse to Theorem~\ref{Thm_simple_eps_regularity}.

\begin{Theorem} \label{Thm_eps_regu_converse}
For any $\eps > 0$ and $Y < \infty$ there is a $\delta (\eps, Y) > 0$ such that the following holds.
Consider a point $(x,t ) \in M \times I$, $r > 0$ with $[t-r^2,t] \subset I$.
Suppose that $|{\Rm}| \leq r^{-2}$ on $P(x,t; r,-r^2)$ and $\NN_{x,t} (r^2) \geq - Y$.
Then $\NN_{x,t} ( \delta r^2) \geq - \eps$.
\end{Theorem}

\subsection{Proofs}
\begin{proof}[Proof of Theorem~\ref{Thm_simple_eps_regularity}.]
Without loss of generality, we may assume that $r=1$ and $t = 1$.
We first apply a point-picking procedure.

\begin{Claim}
Let $A > 0$ and assume that $10 A \rrm (x,t) \leq \frac12$.
Then there is a point $(x', t') \in P^{*-} (x,t; 10 A \rrm(x,t))$ such that 
\[ \rrm(x',t') \leq \rrm(x,t), \qquad \text{and} \qquad \rrm \geq \tfrac1{10} \rrm(x',t') \quad \text{on} \quad P^{*-} (x',t'; A \rrm(x',t')). \]
\end{Claim}

\begin{proof}
We can iteratively pick a maximal sequence $(x'_0, t'_0) := (x,t), (x'_1, t'_1), \ldots$ such that for $r'_i :=\rrm(x'_i,t'_i)$ we have
\[ (x'_{i+1}, t'_{i+1}) \in P^{*-} (x'_i, t'_i; A r'_i), \qquad r'_{i+1} < \tfrac1{10}  r'_i. \]
Since $\inf_{M \times [0,1]} \rrm > 0$, this sequence must be finite.
We claim that its last element $(x'_k, t'_k) =: (x',t')$ has the desired properties.
For this purpose, we will show by induction that 
\begin{equation} \label{eq_xptp_PDxpitpi}
(x', t') \in P^{*-} (x'_i, t'_i; 10 A r'_i).
\end{equation}
The claim follows for $i = 0$.
(\ref{eq_xptp_PDxpitpi}) is trivially true for $i = k$.
If it is true for $i+1$, then by Proposition~\ref{Prop_basic_parab_nbhd} we have
\[ (x',t') \in P^{*-} (x'_{i+1}, t'_{i+1}; 10 A r'_{i+1}) \subset P^{*-} (x'_i, t'_i; A(  r'_i + 10 r'_{i+1})) \subset  P^{*-} (x'_i, t'_i; 10 A r'_i), \]
which concludes the induction.
\end{proof}

Consider now a sequence of counterexamples $(M_i, (g_{i,t})_{t \in [0, 1]})$, $(x_i, t_i = 1) \in M_i \times [0,1]$ to the theorem for some sequence $\eps_i \to 0$.
Then $r^*_i := \rrm (x_i,t_i) \to 0$.
Choose a sequence $A_i \to \infty$ with $10 A_i r^*_i \to 0$. 
Let $(x'_i, t'_i) \in P^{*-}(x_i, t_i, 10 A_i r^*_i)$ be the point from the Claim and set $r'_i := \rrm (x'_i, t'_i) \to 0$.
Since $10 A_i r^*_i \to 0$ and $\eps_i \to 0$, we obtain from Corollary~\ref{Cor_NN_variation} that for all $\tau \in (0, t'_i - \frac12]$
 \[ 0 \geq \NN_{x'_i, t'_i} (\tau) \geq \NN_{x'_i, t'_i} (t'_i - \tfrac12) = \NN^*_{\frac12} (x'_i, t'_i) \to 0. \]
 Note that here we have used the fact that $R \geq - n$ on $M \times [\frac12, 1]$, which follows from Lemma~\ref{Lem_lower_scal}.
 
By Corollary~\ref{Cor_P_sub_Pvar}, we can find a sequence $A'_i \leq A_i$, $A'_i \to \infty$ such that
\[ P^- (x'_i, t'_i; A'_i r'_i) \subset P^{*-} (x'_i, t'_i; A_i r'_i), \]
which implies that
\[    \rrm \geq \tfrac1{10} r'_i \qquad \text{on} \quad P^- (x'_i, t'_i; A'_i r'_i),. \]

Therefore, after applying a time-shift of $-t'_i$ and parabolic rescaling by $r_i^{\prime -2}$, we obtain a sequence of pointed flows of the form $(M'_i, (g'_{i,t})_{t \in [- (A'_i)^2, 0]}, x'_i)$ with $\rrm \geq \frac1{10}$ on $P^-(x'_i, 0; A'_i)$ and $\NN_{x'_i, 0} (T) \to 0$ for any fixed $T > 0$.
By Theorem~\ref{Thm_NLC}, the injectivity radius at $(x'_i, 0)$ is uniformly bounded from below.
So, after passing to a subsequence, these pointed flows converge to a smooth and complete pointed ancient flow of the form $(M_\infty, \lb (g_{\infty, t})_{t \leq 0}, \lb x_\infty)$ with $\rrm \geq \frac1{10}$ everywhere.
By Proposition~\ref{Prop_Var_location_Ric_bound} and Theorems~\ref{Thm_Gaussian_integral_bound}, \ref{Thm_HK_Linfty_bound} it follows moreover that the conjugate heat kernels $K_i (x'_i, 0; \cdot, \cdot) =: (4\pi\tau)^{-n/2} e^{-f_i}$ on $(M', (g'_t)_{t \in [-( A'_i)^2, 0]})$, based at $(x'_i, 0)$, converge locally uniformly to a positive solution $v_\infty = (4\pi \tau)^{-n/2} e^{-f_\infty} \in C^\infty (M_\infty \times \IR_-)$ of the conjugate heat equation.

Since for any fixed $T > 0$ we have $\NN_{x'_i, 0} (T) \to 0$, we obtain using Proposition~\ref{Prop_NN_basic_properties} that for any fixed $T > 0$
\[ \WW [ g'_{i, -T}, f_{i,-T}, T] =  \frac{d}{d\tau} \big( \tau \NN_{x'_i, 0} (\tau) \big) \bigg|_{\tau = T} \longrightarrow 0, \]
which implies
\[ \int_{-T}^0 \int_M \tau \bigg| \Ric + \nabla^2 f_{i,t} - \frac1{2\tau} g_{i,t} \bigg|^2 (4\pi \tau)^{-n/2} e^{-f_{i,t}} dg_{i,t} dt \longrightarrow 0. \]
Passing this to the limit implies that on $M_\infty \times \IR_-$ we have $\Ric + \nabla^2 f_\infty - \frac1{2\tau} g_\infty = 0$.
So $(M_\infty, (g_{\infty, t})_{t <0})$ must be a gradient shrinking soliton.
Due to the uniform curvature bound, this implies that the limit $(M_\infty, (g_{\infty, t})_{t \leq 0})$ must be flat.
We therefore obtain a contradiction to $\rrm (x'_i, 0) = r'_i$ using Perelman's Pseudolocality Theorem \cite[10.3]{Perelman1}.
\end{proof}
\bigskip

\begin{proof}[Proof of Theorem~\ref{Thm_eps_regularity}.]
Without loss of generality, we may assume that $r=1$ and $t  = 1$.
Fix some $\eps > 0$ and consider a sequence of counterexamples $(M_i, (g_{i,t})_{t \in [0, 1]})$, $(x_i, t_i = 1) \in M_i \times [0,1]$ for some sequence $\delta_i \to 0$.

\begin{Claim}
For any  $A , \theta > 0$ we have for $P_{i,A, \theta} := B(x_i, 1, A) \times [\theta, 1]$
\begin{equation} \label{eq_inf_PAtheta_NN}
 \inf_{P_{i,A, \theta}} \NN^*_0 \xrightarrow{\quad i \to \infty \quad} 0. 
\end{equation}
\end{Claim}

\begin{proof}
Assume that the claim holds for some $A, \theta > 0$ or set $A = 0$,  $\theta = 1$.
Let $A' > A$ and $0 < \theta' < \theta$ be constants whose values we will determine later.
We will show that (\ref{eq_inf_PAtheta_NN}) holds with $A, \theta$ replaced with $A', \theta'$ as long as $A' \leq A + c(\theta)$ and $\theta' \geq (1-c) \theta$.
This will imply that we can choose $A, \theta$ arbitrarily large/small.

Assume that (\ref{eq_inf_PAtheta_NN}) was false after replacing $A, \theta$  with $A', \theta'$.
After passing to a subsequence, we may even assume that (\ref{eq_inf_PAtheta_NN}), with $A, \theta$ replaced with $A', \theta'$, remains false for any subsequence.

By Theorem~\ref{Thm_simple_eps_regularity} we obtain that if $\theta'' \geq (1-c) \theta$, $A'' \leq A + c(\theta)$, then after passing to a subsequence, the flows restricted to $P_{i, A'', \theta''}$ and pointed at $(x_i, 1)$ converge to a smooth flow with bounded curvature on a parabolic neighborhood of the form $P_{\infty, A'', \theta''} := B(x_\infty, 1, A'') \times (\theta'', 1]$.

Next, fix some $t^* \in (\theta'', 1]$ and consider the solutions $u_i \in  C^\infty (M_i \times [t^*, 1]))$ to the heat equation $\square u_i = 0$ with initial condition $u_i (\cdot , t^*) = \NN^*_0(\cdot, t^*)$.
By Theorem~\ref{Thm_NN_dependence} we have that 
\begin{equation} \label{eq_NN_u_i_bound}
 \NN^*_0 \leq u_i \leq 0 \qquad \text{on} \quad M_i \times [t^*, 1] .
\end{equation}
So by Corollaries~\ref{Cor_NN_variation}, \ref{Cor_P_sub_Pvar}, we obtain uniform bounds on $u_i$ over $P_{i, A'', \theta''}$.

It follows that for any subsequence of the given counterexamples, we may find another subsequence for which we have $u_i \to u_\infty \in C^\infty (B(x_\infty, 1, A'') \times (t^*, 1] )$, $\square u_\infty = 0$, $u_\infty \leq 0$.
By $\NN_0^* (x_i, 1) \to 0$ and (\ref{eq_NN_u_i_bound}), we have $u_\infty (x_\infty, 1) =0$.
So by the strong maximum principle, we must have $u_\infty \equiv 0$.
On the other hand, Theorem~\ref{Thm_NN_dependence} implies that for any $\beta > 0$ with $t^* + \beta \leq 1$ we have
\[ \NN^*_0 \geq u_i - \frac{n}{2 t^*} \beta \geq u_i - \frac{n}{2\theta''} \beta \qquad \text{on} \quad M_i \times [t^*, t^*+\beta].  \]
Since the choices of $t^*, \beta$ were arbitrary, this implies that if $A' \in (A, A'')$, $\theta' \in (\theta'', \theta)$, then the claim holds for $A, \theta$ replaced with $A', \theta'$, which finishes the proof.
\end{proof}

By combining the Claim with Theorem~\ref{Thm_simple_eps_regularity}, we obtain that after passing to a subsequence, the pointed flows $(M_i, (g_{i,t})_{t \in [0, 1]}, x_i)$ smoothly converge to a pointed flow $(M_\infty, (g_{\infty, t})_{t \in (0,1]}, x_\infty)$ with complete time-slices and a curvature bound of the form $|{\Rm}| \leq C/t$.
As in the proof of Theorem~\ref{Thm_simple_eps_regularity}, we obtain that this limit must be a gradient shrinking soliton and therefore be flat.
This implies curvature bounds on backwards parabolic neighborhoods around $(x_i, 1)$ of larger and larger radius, which can be extended forward in time by Perelman's Pseudolocality Theorem \cite[10.3]{Perelman1} and another limit argument.
As a result, we obtain a contradiction to the choice of the flows $(M_i, (g_{i,t})_{t \in [0, 1]}, x_i)$ for large $i$.
\end{proof}
\bigskip 

\begin{proof}[Proof of Theorem~\ref{Thm_eps_regu_converse}.]
Suppose that the theorem was false for some fixed $\eps, Y$ and choose a sequence of counterexamples $(M_i, (g_{i, t})_{t \in I_i})$, $(x_i, t_i)$, $r_i$ for some sequence $\delta_i \to 0$.
After parabolic rescaling and application of a time-shift, we may assume that $t_i = 0$ and $\delta_i r_i^2 = 1$.
Then $r_i \to \infty$ and, using Theorem~\ref{Thm_NLC}, we obtain that after passing to a subsequence, the pointed flows $(M_i, (g_{i, t})_{t \in I_i}, x_i)$ converge to the constant flow $(M_\infty, (g_{\infty, t})_{t \leq 0}, x_\infty)$ on a flat manifold with positive asymptotic volume ratio.
Thus $(M_\infty, (g_{\infty, t})_{t \leq 0})$ must be isometric to the constant flow on $\IR^n$.
By Proposition~\ref{Prop_Var_location_Ric_bound} and Theorems~\ref{Thm_Gaussian_integral_bound}, \ref{Thm_HK_Linfty_bound} we also obtain that the conjugate heat kernels $K(x_i, 0; \cdot, \cdot) = (4\pi \tau)^{-n/2} e^{-f_i}$ subsequentially converge to a solution $v_\infty =  (4\pi \tau)^{-n/2} e^{-f_\infty}\in C^\infty ( \IR^n \times \IR_-)$ of the backwards heat equation such that for some $\vec p \in \IR^n$
\[ v_\infty \leq C |t|^{-n/2}, \qquad \Var (\delta_{\vec p} , v_{\infty,t} d\vec x) \leq C |t|. \]
Therefore, $v_\infty$ must the the standard Gaussian in $\vec p$ and we obtain 
\[  \NN_{x_\infty, 0} (1) = \int_{M_\infty} f_{\infty, -1} \, v_{\infty,t} \, dg_{\infty,-1} - \frac{n}2 = 0. \]

To see that $\NN_{x_i, 0} (1) \to 0$, fix some large $Q < \infty$ and observe that for large $i$ we have using Theorem~\ref{Thm_HK_Linfty_bound}, Proposition~ \ref{Prop_nu_ball_bound}
\begin{equation} \label{eq_MisetminusBxiQ}
  \int_{M_i \setminus B(x_i, -1, Q)} f_{i,-1} (4\pi )^{-n/2} e^{-f_{i,-1}} dg_{i, -1}  
\geq -C(Y)\int_{M_i \setminus B(x_i, -1, Q)}  (4\pi )^{-n/2} e^{-f_{i,-1}} dg_{i, -1} 
\geq - \frac{C(Y)}{Q^2}
\end{equation}
and
\begin{equation} \label{eq_BxiQ_NN}
  \int_{ B(x_i, -1, Q)} f_i (4\pi )^{-n/2} e^{-f_{i,-1}} dg_{i, -1} \to
\int_{B(x_\infty, -1, Q)}  f_{\infty,-1} (4\pi )^{-n/2} e^{-f_{\infty,-1}} dg_{i, -1}
\geq \frac{n}2 - \Psi(Q), 
\end{equation}
where $\lim_{Q \to \infty} \Psi(Q)=0$.
Combining (\ref{eq_MisetminusBxiQ}),  (\ref{eq_BxiQ_NN}) implies $\NN_{x_i, 0} (1) \to 0$, which yields the desired contradiction for large $i$.
\end{proof}
\bigskip

\section{Poincar\'e inequalities} \label{sec_Poincare}
\subsection{Statement of the result}
In this section we prove an $L^p$-Poincar\'e inequality for a manifold that is equipped with a conjugate heat kernel measure.
This inequality generalizes an $L^2$-Poincar\'e inequality due to Hein and Naber \cite[Theorem~1.10]{Hein-Naber-14}.

Let in the following $(M, (g_t)_{t \in I})$ be a Ricci flow on a compact manifold.

\begin{Theorem} \label{Thm_Poincare}
Consider the conjugate heat kernel measure $d\nu = (4\pi \tau)^{-n/2} e^{-f} dg$ based at some point $(x_0,t_0) \in M \times I$.
If $[t_0 - \tau , t_0] \subset I$ for some $\tau > 0$, then we have for any $h \in C^1(M)$ and $p \geq 1$
\begin{equation} \label{eq_basic_Poincare}
 \int_M h  \, d\nu_{t_0-\tau} = 0 \qquad \Longrightarrow \qquad \int_M |h|^p  d\nu_{t_0-\tau} \leq C(p) \tau^{p/2} \int_M |\nabla h|^p d\nu_{t_0-\tau} 
\end{equation}
and
\begin{equation} \label{eq_Poincare_no_average}
 \int_M |h|^p d\nu_{t_0-\tau} \leq C(p) \tau^{p/2} \int_M |\nabla h|^p d\nu_{t_0-\tau} + C(p) \bigg(  \int_M h \, d\nu_{t_0-\tau} \bigg)^p. 
\end{equation}
We may choose $C(1) = \sqrt{\pi}$ and $C(2) = 2$.
\end{Theorem}

The fact that we can choose $C(2) = 2$ is already the content of \cite[Theorem~1.10]{Hein-Naber-14}.
We have repeated this statement for completeness.

\subsection{Proof}
\begin{proof}
After parabolic rescaling and application of a time-shift, we may assume without loss of generality that $\tau = 1$ and $t_0 =1$.

In the case $p =2$, the inequality (\ref{eq_basic_Poincare}) follows from \cite[Theorem~1.10]{Hein-Naber-14}.

Next, let us show the inequality for $p =1$.
The proof is inspired by \cite{BakryBarthe08, Ledoux04}.
Let $u \in C^\infty(M)$ with $|u| < 1$.
View $h, u$ as functions at time $t=0$ and extend both functions onto $M \times [0,1]$ by solving the heat equation $\square h = \square u = 0$.
Then
\[ h(x_0,1) = \int_M h  \, d\nu_{0} = 0. \]
The bound $|u | < 1$ remains preserved by the heat equation and by Theorem~\ref{Thm_gradient_estimate} we know that $u = \Phi_{t} \circ v$ for some $v \in C^\infty(M \times [0,1])$ with $|\nabla v| \leq 1$.
Therefore
\[ |\nabla u| = |\nabla v|  \big( \Phi'_{t} \circ v \big) \leq (4\pi t)^{-1/2}. \]
By Kato's inequality, we have $\square |\nabla h| \leq 0$ in the viscosity sense, which implies that
\[ \frac{d}{dt}  \int_M |\nabla h|  \, d\nu_t \leq 0. \]
Next, observe that $\square (hu) = -2 \nabla h \cdot \nabla u$, so
\begin{equation*}
 \frac{d}{dt} \int_M hu \, d\nu_t = - 2 \int \nabla h \cdot \nabla u \, d\nu_t 
  \geq - \sup_M |\nabla u|(\cdot, t) \int_M |\nabla h| \, d\nu_t 
 \geq - (4\pi t)^{-1/2} \int_M |\nabla h| \,  d\nu_{0}.
\end{equation*}
Integrating this over $t$ from $0$ to $1$ yields
\begin{equation*}
 \int_M hu \, d\nu_{0} 
 \leq (hu)(x_0,1) + \sqrt{\pi}  \int_M |\nabla h|  \, d\nu_{0} \\
 = \sqrt{\pi}  \int_M |\nabla h|  \, d\nu_{0}.
\end{equation*}
The bound (\ref{eq_basic_Poincare}) now follows by taking a limit of the form $u \to \sign (h)$.

Next, we show (\ref{eq_basic_Poincare}) for $p > 1$, using the fact that the inequality is already known in the case $p=1$.
Choose $a \in \IR$ in such a way that $\nu(\{ h - a \leq 0 \} ), \nu(\{ h - a \geq 0 \} ) \geq \frac12$. 
Then
\begin{equation} \label{eq_a_bound_p}
 |a| \leq 2 \int_M |h| \, d\nu \leq C \int_M |\nabla h| \, d\nu \leq C \bigg( \int_M |\nabla h|^p d\nu \bigg)^{1/p}. 
\end{equation}
Set $h' := h - a$ and recall that $h'_+ := \max \{ h', 0 \}$ and $h'_- := \max \{- h', 0 \}$.
By (\ref{eq_a_bound_p}) it suffices to show that 
\begin{equation} \label{eq_h_p_goal}
 \int_M (h'_\pm)^p d\nu_0 \leq C(p) \int_M |\nabla h'_\pm |^p d\nu_{0}. 
\end{equation}
Let $u := (h'_\pm)^p$ and recall that $u \geq 0$ and $\nu ( \{ u = 0 \} ) \geq \frac12$.
Set
\[ b := \int_M u \, d\nu_0 \geq 0. \]
Then using the case $p=1$ we have
\[ \int_M |u - b|\,  d\nu_0 \leq C \int_M |\nabla u| \, d\nu_0.  \]
Since
\[ b/2 \leq \int_{\{ u = 0 \}} |u-b| \, d\nu_0 \leq C \int_M |\nabla u| \, d\nu_0, \]
this implies
\[ \int_M u \, d\nu_0 \leq C \int_M |\nabla u| \, d\nu_0.  \]
Thus we have
\[ \int_M (h'_\pm)^p d\nu_0 
\leq Cp \int_M |\nabla h'_\pm| \, (h'_\pm)^{p-1} d\nu_0
\leq Cp \bigg( \int_M |\nabla h'_\pm|^{p} d\nu_0 \bigg)^{1/p}
\bigg( \int_M (h'_\pm)^{p} d\nu_0 \bigg)^{(p-1)/p}, \]
which implies (\ref{eq_h_p_goal}) and finishes the proof of (\ref{eq_basic_Poincare}) in the case $p > 1$.

Finally, to see (\ref{eq_Poincare_no_average}), we apply (\ref{eq_basic_Poincare}) to $h -a$ with $a := \int_M h \, d\nu_0$.
\end{proof}
\bigskip

\section{Hypercontractivity} \label{sec_hypercontractive}
\subsection{Statement of the result}
In this section we show a hypercontractivity property for the heat equation on a Ricci flow background equipped with a conjugate heat kernel measure.
The result, whose proof is based in a technique by Gross \cite{Gross_log_Sobolev}, was also observed by Hein.
It is recorded here for completeness and since it will be needed in subsequent work.

Let $(M,(g_t)_{t \in I})$ be a Ricci flow on a compact manifold.

\begin{Theorem}
Suppose that $(x_0, t_0) \in M \times I$ and that $0 < \tau_1 < \tau_2$ with $[t_0 - \tau_2, t_0] \subset I$.
Denote by $d\nu = (4\pi \tau)^{-n/2} e^{-f} dg$ the conjugate heat kernel based at $(x_0, t_0)$ and let $u \in C^2 (M \times [t_0 - \tau_2 ,t_0 - \tau_1])$ be a solution to the heat equation $\square u = 0$ or a non-negative subsolution to the heat equation, $\square u \leq 0$.
If $ 1< q \leq p < \infty$ with
\[ \frac{\tau_2}{\tau_1} \geq \frac{p-1}{q-1}, \]
then
\begin{equation} \label{eq_hyp_cont_ineq}
 \bigg( \int_M |u|^p d\nu_{t_0 - \tau_1} \bigg)^{1/p} \leq \bigg( \int_M |u|^q d\nu_{t_0 - \tau_2} \bigg)^{1/q}. 
\end{equation}
\end{Theorem}

\subsection{Proof}
\begin{proof}
After parabolic rescaling and application of a time-shift, we may assume that $t_0 = 0$ and $\tau_2 = 1$.
If $\square u = 0$, then $\square |u| \leq 0$ in the viscosity sense, so we may assume that $u \geq 0$ is continuous and satisfies $\square u \leq 0$ in the viscosity sense.
Moreover, if $u' \in C^2 (M \times [- 1 , 0])$ is a solution to the heat equation $\square u' = 0$ with initial condition $u' (\cdot, -1) = u(\cdot, -1)$, then $u( \cdot, -\tau_1) \leq u' (\cdot, -\tau_1)$.
This shows that it suffices to consider the case in which $u \geq 0$ and $\square u = 0$.

Set $p(t) := 1 + (q-1) |t|^{-1}$ and observe that $p(-1) = q$.
Using \cite[Theorem~1.10]{Hein-Naber-14}, we obtain
\begin{align*}
 \frac{d}{dt} \int_M u^{p(t)} d\nu_t
&\leq \int_M \big( \dot p u^p \log u - p (p-1) |\nabla u|^2 u^{p-2} \big) d\nu_t \\
&= \frac{\dot p}{p} \int_M u^p \log u^p  d\nu_t - \frac{p-1}{p} \int_M \frac{|\nabla u^p|^2}{u^p} d\nu_t \\
&\leq \frac{\dot p}{p} \bigg( \int_M u^p d\nu_t \bigg) \log \bigg( \int_M u^p d\nu_t \bigg) + \bigg( \frac{\dot p}{p} |t| - \frac{p-1}{p} \bigg)\int_M \frac{|\nabla u^p|^2}{u^p} d\nu_t \\
&= \frac{\dot p}{p} \bigg( \int_M u^p d\nu_t \bigg) \log \bigg( \int_M u^p d\nu_t \bigg). 
\end{align*}
So
\begin{equation} \label{eq_ddt_int_pt}
 \frac{d}{dt} \bigg(  \int_M u^{p(t)} d\nu_t \bigg)^{1/p(t)} \leq 0. 
\end{equation}
Since
\[ p(-\tau_1) = 1 + \frac{q-1}{\tau_1} \geq 1 + (q-1) \frac{p-1}{q-1} = p, \]
we obtain that
\[ \bigg(  \int_M u^{p} d\nu_{-\tau_1} \bigg)^{1/p}
\leq \bigg(  \int_M u^{p(-\tau_1)} d\nu_{-\tau_1} \bigg)^{1/p(-\tau_1)}
\leq \bigg(  \int_M u^{p(-1)} d\nu_{-1} \bigg)^{1/p(-1)}
= \bigg(  \int_M u^{q} d\nu_{-1} \bigg)^{1/q}, \]
which finishes the proof.
\end{proof}

\bibliography{bibliography}{}
\bibliographystyle{amsalpha}

\end{document}